\documentclass[a4paper,11pt]{article}

\usepackage{amsmath}
\usepackage{amssymb}
\usepackage{amsthm}
\usepackage{graphicx}
\usepackage{verbatim}
\usepackage{color}
\usepackage[%
ocgcolorlinks,%
linkcolor=blue,%
filecolor=blue,%
citecolor=blue,%
urlcolor=blue]{hyperref}
        
\usepackage[font={small,it}]{caption}

\newcommand{\bbW}{\mathbb{W}}
\newcommand{\bbB}{\mathbb{B}}

\newcommand{\bbT}{\mathbb{T}}
\newcommand{\bbM}{\mathbb{M}}
\newcommand{\bbE}{\mathbb{E}}
\newcommand{\bbV}{\mathbb{V}}
\newcommand{\bbX}{\mathbb{X}}
\newcommand{\bbY}{\mathbb{Y}}
\newcommand{\Var}{\bbV{\rm ar}}
\newcommand{\Cov}{\bbC{\rm ov}}
\newcommand{\bbP}{\mathbb{P}}

\newcommand{\bbN}{\mathbb{N}}
\newcommand{\bbZ}{\mathbb{Z}}
\newcommand{\bbR}{\mathbb{R}}
\newcommand{\bbC}{\mathbb{C}}
\newcommand{\bbU}{\mathbb{U}}
\newcommand{\bbJ}{\mathbb{J}}

\newcommand{\ol}{\overline}
\newcommand{\ul}{\underline}
\newcommand{\wh}{\widehat}

\newcommand{\cA}{\mathcal A}
\newcommand{\cP}{\mathcal P}
\newcommand{\cZ}{\mathcal Z}

\newcommand{\cX}{\mathcal X}

\newcommand{\cB}{\mathcal B}
\newcommand{\cG}{\mathcal G}

\newcommand{\cE}{\mathcal E}
\newcommand{\cN}{\mathcal N}
\newcommand{\cF}{\mathcal F}

\newcommand{\cU}{\mathcal U}
\newcommand{\cV}{\mathcal V}
\newcommand{\cD}{\mathcal D}

\newcommand{\frg}{\mathfrak g}

\newcommand{\frb}{\mathfrak b}

\newcommand{\rmc}{{\rm c}}
\newcommand{\rmF}{{\rm F}}
\newcommand{\rmd}{{\rm d}}
\newcommand{\rme}{{\rm e}}
\newcommand{\rmB}{{\rm B}}

\newcommand{\rmC}{{\rm C}}
\newcommand{\rmA}{{\rm A}}
\newcommand{\rmV}{{\rm V}}
\newcommand{\rmW}{{\rm W}}
\newcommand{\rmU}{{\rm U}}
\newcommand{\rmD}{{\rm D}}
\newcommand{\rmG}{{\rm G}}

\newcommand{\rmX}{{\rm X}}
\newcommand{\rmY}{{\rm Y}}
\newcommand{\rmZ}{{\rm Z}}
\newcommand{\rmQ}{{\rm Q}}
\newcommand{\rmO}{{\rm O}}
\newcommand{\rmR}{{\rm R}}
\newcommand{\rmE}{{\rm E}}
\newcommand{\rmJ}{{\rm J}}

\newcommand{\bfN}{{\mathbf N}}
\newcommand{\bfL}{{\mathbf L}}
\newcommand{\bfl}{{\mathbf l}}
\newcommand{\bfT}{{\mathbf T}}
\newcommand{\bfR}{{\mathbf R}}
\newcommand{\bfC}{{\mathbf C}}

\newcommand{\bft}{{\mathbf t}}

\newcommand{\wt}{\widetilde}

\newcommand{\ck}{\check}

\newcommand{\laweq}{\overset{\rm d}{=}}

\newtheorem{theorem}{Theorem}[section]

\newtheorem{lem}[theorem]{Lemma}
\newtheorem{prop}[theorem]{Proposition}
\newtheorem{thm}[theorem]{Theorem}
\newtheorem*{rem*}{Remark}
\newtheorem{rem}[theorem]{Remark}
\newtheorem{defn}[theorem]{Definition}
\newtheorem{mainthm}{Theorem}

\numberwithin{equation}{section}
\graphicspath{{pic/}}

\newcommand{\FromS}[1]{}
\newcommand{\FromO}[1]{}
\newcommand{\ToS}[1]{}
\newcommand{\ToSi}[1]{}

\newcommand{\texte}{\text{\rm e}\mkern0.7mu}

\newcommand{\1}{{1\mkern-4.5mu\textrm{l}}}
\renewcommand{\1}{\text{\sf 1}}

\newcommand{\NN}{\mathcal N}

\newcommand{\E}{\mathbb E}

\newcommand{\N}{\mathbb N}

\newcommand{\R}{\mathbb R}

\newcommand{\Z}{\mathbb Z}

\title
{Tightness for the Cover Time of Wired Planar Domains}
\author
{Oren Louidor\thanks{oren.louidor@gmail.com,  saglietti.sj@uc.cl.} 
\\ Technion, Israel \and Santiago Saglietti\footnotemark[1]\\ Pontificia Universidad\\ Cat\'olica de Chile}
\date{}

\begin{document}
\maketitle

\begin{abstract} 
We consider a continuous time simple random walk on a subset of the square lattice with wired boundary conditions: the walk transitions at unit edge rate on the graph obtained from the lattice closure of the subset by contracting the boundary into one vertex. We study the cover time of such walk, namely the time it takes for the walk to visit all vertices in the graph. Taking a sequence of subsets obtained as scaled lattice versions of a nice planar domain, we show that the square root of the cover time normalized by the size of the subset, is tight around $\frac{1}{\sqrt{\pi}} \log N - \frac{1}{4 \sqrt{\pi}} \log \log N$, where $N$ is the scale parameter. This proves an analog, for the wired case, of a conjecture by Bramson and  Zeitouni~\cite{bramson2009tightness}. The proof is based on comparison with the extremal landscape of the discrete Gaussian free field.
\end{abstract}

\setcounter{tocdepth}{2}
\tableofcontents

\section{Introduction and results}

\subsection{Setup and main result}
\label{s:1}
Take $\rmD \subset \bbR^2$ to be an open, bounded and connected planar set containing the origin, whose boundary is a union of disjoint rectifiable curves having positive diameter, and, for any $N > 1$, consider its discretized $N$-scale:
\begin{equation}
\label{e:1.1}
  \rmD_N := \Big\{ x \in \bbZ^2 :\: \rmd(x/N, \rmD^\rmc) > 1/N \Big\} \,. 
\end{equation} 
Above $\rmd(\cdot, \cdot)$ is the usual point-to-set distance with respect to the Euclidean norm on $\bbR^2$. Viewing subsets of $\bbZ^2$ as sub-graphs of the lattice, let $\wh{\rmD}_N$ be the graph obtained from $\rmD_N \cup \partial \rmD_N$, where $\partial \rmD_N$ is the outer boundary of $\rmD_N$, by contracting all vertices in $\partial \rmD_N$ into a single vertex to be denoted by $\partial$. 

On $\wh{\rmD}_N$ run a continuous time simple random walk $X=(X_u :\: u \geq 0)$, which starts from $\partial$ and transitions at (edge) rate $1$. The {\em cover time} of $\wh{\rmD}_N$ by $X$ is defined as the first time by which all of the vertices of $\wh{\rmD}_N$ has been visited, namely
\begin{equation}
\label{e:1.1a}
	\bfT^\bfC_N := \min \big\{t \geq 0:\: \cup_{u \in [0,t]}\{ X_u\} = \wh{\rmD}_N \big\} \,.
\end{equation}
The main result of this work is:
\begin{mainthm}
\label{t:A}
The collection of random variables
\begin{equation}
\label{e:2.4}
\left(\sqrt{\frac{\bfT^\bfC_N}{|\wh{\rmD}_N|}} - \sqrt{\bft_N^\bfC} \quad : \quad N > 1\right)\,,
\end{equation}
with $\bft^\bfC_N$ defined via the relation
\begin{equation}
\label{e:1.5}
\sqrt{\bft^\bfC_N} = \frac{1}{\sqrt{\pi}} \log N - \frac{1}{4 \sqrt{\pi}} \log \log N \,,
\end{equation}
is tight.
\end{mainthm}  

Let us make some remarks. First, here and throughout the sequel, $\log$ is always the natural logarithm. In addition, an immediate and more explicit reformulation of~\eqref{e:2.4} is the existence of a tight collection of random variables $(\bfR_N)_{N > 1}$ such that
\begin{equation}
\label{e:1.5.1}
\bfT^\bfC_N = \frac{1}{\pi} |\rmD_N| (\log N)^2 - \frac{1}{2\pi} |\rmD_N| \log N \log \log N + |\rmD_N| (\log N) \bfR_N \,,
\end{equation}
i.e. the cover time fluctuates around $\frac{1}{\pi} |\rmD_N| (\log N)^2 - \frac{1}{2\pi} |\rmD_N| \log N \log \log N$ with deviations of order $O\big(|\rmD_N| (\log N) \bfR_N\big)$.

Concerning the precise set-up, the theorem remains true regardless of the initial vertex of the walk. Indeed, the commute time from any vertex to $\partial$ is $\Theta(|\rmD_N|)$, which is of smaller order than the allowed fluctuations, as shown in~\eqref{e:1.5.1}. Also, the assumption of continuous time is not a necessary one, and replacing $X$ with a discrete time random walk on $\wh{\rmD}_N$ does not effect the validity of the result, except for a multiplicative factor of $4$
	applied to $\bft_N^\bfC$, to account for the change in the mean speed of the walk. 

The wired boundary conditions, on the other hand, are crucially used in the proof (see Subsection~\ref{ss:2.4}). While it is a common conjecture that $\sqrt{\bfT_N^\bfC}$ is still tight around $\sqrt{\bft_N^\bfC}$, as defined in~\eqref{e:1.5}, when one takes periodic (the underlying graph is $\bbZ^2/(N \bbZ^2)$) or free (the underlying graph is $\rmD_N$) boundary conditions (see~\cite{abe2021second},~\cite{bramson2009tightness} and the discussion in~\cite{ding2012cover}), our proof cannot be easily modified to handle this case. 

\subsection{Context and historical background}
The cover time can be alternatively expressed in terms of the time spent at each vertex of the underlying graph. To be more precise, for $x \in \wh{\rmD}_N$ and $u \geq 0$, we define the {\em local time} at $x$ by {\em real time} $u$ as
\begin{equation}
\bfL_{t}(x) := \int_{0}^t 1_x\big(X_{u}\big) \rmd u 
\end{equation}
(we do not explicate the dependence on $N$). Then, the cover time of $\wh{\rmD}_N$ can be equivalently defined as
\begin{equation}
\label{e:1.4a}
\bfT^\bfC_N := \min \big\{t \geq 0:\: \min_{x \in \wh{\rmD}_N} \bfL_{t}(x) > 0 \big\} \,. 
\end{equation}
That is, the cover time is the first real time by which the minimal local time is positive. 

Display~\eqref{e:1.4a} recasts the cover time as an extreme order statistic of the field of local times $\bfL_t := (\bfL_t(x) :\: x \in \wh{\rmD}_N)$. This field (and its square root) is known to have approximate logarithmic correlations in two dimensions. Indeed, the covariance kernel of $\bfL_t$ (as well as that of $\sqrt{\bfL_t}$) is closely related to the discrete Green function on $\rmD_N$ (with Dirichlet boundary conditions), which in our normalization takes the form
\begin{equation}
 G_{\rmD_N}(x,y) = \frac{1}{2 \pi} \log N - \frac{1}{2 \pi}\log \big(\|x-y\| \vee 1\big) + O(1) 
 \quad x,y \in \text{bulk of } \rmD_N \,.
\end{equation} 
(see, e.g. proof of Lemma~\ref{l:2.3} for a more precise statement).

Extreme order statistics of (approximately) logarithmically correlated fields have been a subject of considerable research in the past two decades. Other notable examples which have been studied include Branching Brownian Motion, the discrete Gaussian free field on planar domains or tree-like graphs, the (logarithm of the) modulus of the Riemann zeta function on a uniformly chosen interval on the critical line and the (logarithm of the) modulus of the characteristic polynomial of a CUE matrix on the unit circle. A recent survey on this subject can be found in~\cite{arguin2016extrema}.

In all known examples, statistics of extreme values exhibit similar asymptotic features. These,  common ``signatures'' include a negative log-log correction (in the size of the domain of the field) to the typical height of the maximum compared to the case of an i.i.d. field with the same variance, $O(1)$ fluctuations of the maximum around its typical height with a randomly shifted Gumbel Law in the limit, the clumping of extreme values into finite size clusters which are macroscopically apart and the emergence of a clustered Cox process as the limiting extremal process of the field.

Theorem~\ref{t:1.1} therefore provides additional evidence to support the conjectural universality of such features among all logarithmically correlated fields. Indeed, it can be shown (see e.g. Lemma~\ref{lem:nhprob} in conjunction with Proposition~\ref{p:2.2}) that for all $x$ in the bulk of $\rmD_N$,
\begin{equation}
\label{e:1.8}
	\bbP \big(\bfL_{t|\wh{\rmD}_N|}(x) = 0\big) \approx \Big(1 - \frac{1}{G_{\rmD_N}(x,y)}\Big)^t
	\approx \rme^{-2\pi \frac{t}{\log N}} \,,
\end{equation}
so that had the events above been independent, one would have gotten for $t \gg \log N$, 
\begin{equation}
\label{e:1.9}
\bbP \Big(\bfT_N^{\bfC}/|\wh{\rmD}_N| \leq t\big) \approx \Big(1-\rme^{-2\pi \frac{t}{\log N}}\Big)^{|\rmD_N|}
\approx \exp \Big(-\rme^{-2\pi \frac{t}{\log N} + 2\log N + C_\rmD}\Big) \,,
\end{equation}
This would have implied that $\sqrt{\bfT_N^\rmC/|\wh{\rmD}_N|}$ is tight around the centering sequence $\frac{1}{\sqrt{\pi}} \log N$, which is indeed larger than $\sqrt{\bft_N^\bfC}$ by a double logarithm term.

Let us also briefly survey the historical development of this problem. One of the earliest mentioning of the search for asymptotics of the cover time can be found in the editorial piece of Hebert Wilf~\cite{wilf1989editor} who, for self-amusement, wrote a Basic program which simulates a random walk on his screen, acting as a two dimensional discrete torus $\bbZ^2/(N\bbZ^2)$ of side length $N \geq 1$. In this casual-reading article, the author wondered about the time it took for the screen to become ``completely white''. 

An upper bound of the correct first order was found by Aldous~\cite{aldous1991threshold} in '89. A non-matching first order lower bound (in expectation) was then given by Lawler~\cite{lawler1993covering} in '92. This was later improved by the seminal work of Ding, Peres, Rosen and Zeitouni~\cite{dembo2004cover} in '04, thereby settling the question of first order asymptotics. The proof in~\cite{dembo2004cover} is based on deriving sharp estimates for the first and second moments of the number of unvisited vertices, subject to a ``truncation'' event, which is added in order to make this random quantity concentrate around its mean.

Turning to the second order term, Ding~\cite{ding2012cover} in '12 showed that it is of order $O(\log \log N)$. This result relied on an earlier work by Ding, Lee and Peres~\cite{ding2011cover} in `11 who exploited, for the first time (as far as we know), the connection between the cover time and the maximum of the discrete Gaussian free field on the same underlying graph. As our proof is more in line with this approach, further details on this connection could be found in Subsection~\ref{ss:2.4}. The full derivation for the second order term was done by Belius and Kistler~\cite{belius2017subleading} in `17 for the analogous problem of the cover time by a Brownian Motion (Sausage) and then by Abe~\cite{abe2021second} in '20 for a random walk. These works, however, stop short of proving tightness. For the case of Brownian Motion on the two dimensional sphere, this was recently settled by Dembo, Rosen and Zeitouni~\cite{belius2020tightness}.

It should be stressed that in all of the above mentioned works (except for those dealing with the cover time of Brownian Motion), the underlying graph was that of the two dimensional discrete torus, not a wired domain as we have here. Concerning other graphs, the case of the discrete torus $\bbZ^d/(N\bbZ^d)$ in dimension $d \geq 3$ was settled by Belius~\cite{belius2013gumbel} in `12. This work shows that, normalized by the volume of the torus and then centered by a constant times $\log N^d$, the cover time converges in law to a Gumbel. We remark that in three and above dimensions, in analog to the behavior of the Green Function, the correlations of the local time field decay polynomially fast, making the i.i.d. analysis in~\eqref{e:1.8}-\eqref{e:1.9} give the correct asymptotics.

Somehow surprising at first sight, a more related set-up to the one here is that of the underlying graph being the binary (or $d$-ary, in general) rooted tree of depth $n$. In this set-up, the Green Function and hence the local time field decreases linearly in the graph distance, which in a dyadic embedding of the leaves on an interval on the real line, is approximately logarithmic in the distance. In this case, it was shown in~\cite{cortines2021scaling} and later by a different method by Dembo, Rosen and Zeitouni~\cite{dembo2021limit} that 
\begin{equation}
\label{e:1.10}
	\sqrt{\frac{\bfT_n^\rmC}{|\bbT_n|}}
- \Big( \sqrt{\log 2}\, n - \tfrac{1}{2 \sqrt{\log 2}} \log n \Big) 
\underset{n \to \infty}\Longrightarrow
	G + \log \ol{Z} \,.
\end{equation}
Above $\bbT_n$ is the set of vertices in the tree, $G$ is a Gumbel random variable, and $\ol{Z}$ is a random variable which is positive and finite almost-surely and independent of $G$. Moreover, the law of $\ol{Z}$ is closely related to that of the limit of the derivative martingale of a branching random walk with  deterministic binary branching and $\cN(0,1/2)$ steps. 

Writing $n$ as $\log_2 (N+1)-1 = \frac{1}{\log 2} \log N - 1 + o(1)$, with $N$ being the number of vertices in the tree,~\eqref{e:1.10} is another demonstration of the universality of the aforementioned extreme value theory. It is believed that a similar limiting statement should hold for the cover time in the case of planar domains with wired, periodic or free boundary conditions.

Lastly, it is worth mentioning the substantial work that has been conducted on the ``opposite side'' of extremality in local times of random walks and Brownian Motion. Here one typically studies the so-called most favorite (or thick) and $\lambda$-favorite points of the underlying motion, namely points where the local time is maximal or at least a $\lambda \in (0,1)$ fraction of the typical value of the maximum. Research in this area started in the pioneering work of Erd\"os and Taylor~\cite{ErdosTaylor} who derived non-matching upper and lower bound on the order of $(\log n)^2$ for the highest local time of a planar simple random walk run until time $n$. The lower bound was ``corrected'' more than forty years later, in the work of Dembo, Peres, Rosen and Zeitouni~\cite{DPJZ_Thick}, establishing $\pi^{-1} (\log n)^2$ as the leading order term in the asymptotics for the local time at the most favorite point (for both the simple random walk and Brownian Motion). The second order term along with tightness, in the case of Brownian Motion, was recently shown by Rosen~\cite{Rosen_Thick}.

$\lambda$-favorite points were also treated in~\cite{DPJZ_Thick}, where it was showed that with high probability there are $n^{1-\lambda + o(1)}$ many points whose local time is at least $\lambda \pi^{-1} (\log n)^2$ for $\lambda \in (0,1)$. These asymptotics were sharpened considerably by Jego~\cite{Jego1, Jego2, Jego3}, who recently established a weak scaling limit for a random-measure encoding of these points. The limiting law, describing the asymptotics of both the number and spatial distribution of the $\lambda$-favorite points turns out to be the so-called (sub-critical) Brownian Multiplicative Chaos. A similar derivation in the case where the underlying walk is run for times comparable with the cover time of the domain (with periodic boundary conditions), was carried out by Abe and Biskup~\cite{AbeBiskup1, AbeBiskup2}. In this very different order of time scales, the limit turns out to be closely related to the (sub-critical) Gaussian Multiplicative Chaos associated with the continuous Gaussian free field on the domain. Similar problems when the underlying graph is the rooted regular tree were studied by Abe~\cite{Abe_Tree} and by Biskup and the first author~\cite{BL_FavoritePoint}.

\section{Top-level proof}
\label{s:2}
In this section we present a ``top-level proof'' of Theorem~\ref{t:A}, namely the ``top layer'' in the hierarchy of arguments which compose the full proof of the theorem. While this gives the outline of the complete proof, it is not merely an overview as it includes full details of the top-level part of the entire argument. In the end of this section we also include a (real) overview of some of the lower level arguments in the proof.
\subsection{Notation}
Let us introduce some notation. As the proofs follow a multi-scale analysis along an exponential scale, we shall henceforth treat all length parameters as if given on such scale. Thus, 
from now on and with a slight override of notation, we shall henceforth write $\rmD_n$ to mean 
$\rmD_N$ with $N=\rme^n$, as was defined in~\eqref{e:1.1} (throughout $n > 0$ need not be an integer).

For~$x \in \bbZ^2$, $k \in \bbN_0$, the open ball in $\bbZ^2$ of exponential radius $k$ and center point $x \in \Z^2$ will be denoted by $\rmB(x;k) := \{y \in \Z^2 :\: \|x-y\| < \rme^k\}$, with $\|\cdot\|$ being always the Euclidean norm.  Given $\rmU \subset \bbZ^2$ non-empty, the \textit{diameter} of $\rmU$ will be as usual the supremal (Euclidean) distance between any + points in $\rmU$. For $r \geq 0$, the $r$-\textit{bulk} of $\rmU$ will be defined as $\rmU^r := \{x \in \bbZ^2 :\: \rmd(x, \rmU^\rmc) > \rme^r \}$. Whenever $U=\cD_n$ is the $\rme^n$-scale of some domain $\cD \subseteq \R^2$, we will write $\cD_n^r$ as short for $(\cD_n)^r$.
 	
As in the introduction, we will view subsets $\rmU \subset \bbZ^2$ also as sub-graphs of the square lattice, inheriting from it all edges which connect two vertices in $\rmU$. For such sets, $\partial \rmU$ shall denote its \textit{outer boundary}, namely the set of vertices in $\bbZ^2 \setminus \rmU$ which share an edge of $\bbZ^2$ with at least one vertex in $\rmU$, and $\ol{\rmU}$ will be defined as $\rmU \cup \partial \rmU$. Scaled versions of the lattice~$\bbZ^2$ will be denoted by $\bbX_k := \lfloor \rme^k\rfloor \bbZ^2$ for $k \in \N_0$.
Given~$\rmU \subset \bbZ^2$, we define its \textit{log-scale} $\rho(\rmU)$ as $\rho(\rmU):=\min \{ k \in \bbN_0 : \exists y \in \bbX_k \text{ s.t. }\rmU \subseteq \rmB(y;k)\}$. Notice that the collection $\big\{\rmB(x;k) : x \in \bbX_k\big\}$ is a cover of $\bbZ^2$ by overlapping balls of log-scale $k$.

Finally, all constants throughout are assumed to be finite and positive and their value may change from one use to the next, and for $f:\mathcal{X} \rightarrow \mathcal{Y}$ and $\mathcal{Z} \subseteq \mathcal{X}$ we write $f(\cZ)$ for $\{ f(z) : z \in \cZ\}$.

\subsection{Time reparametrization}
The first step in the proof is to ``reparameterize'' time so that time is ``measured in terms of the local time at the boundary''. Formally, we recall that $\partial$ is the boundary vertex and let 
\begin{equation}
\bfL^{-1}_t(\partial) :                                                                                                                                                                                                                                                                                                                                                                                                                                                                                                                                                                                                                                                                                                                                                                                                                                                                                                                                                                                                                                                                                                                                                                                                                                                                                                                                                                                                                                                                                                                                                                                                                                                                                                                                                                                     = \inf \big\{u \geq 0 :\: \bfL_{u}(\partial) > t \big\} \,.
\end{equation}
be the generalized inverse of $u \mapsto \bfL_{u}(\partial)$. We shall loosely refer to this quantity as the real time by which the local time at $\partial$ is $t$, or $\partial$-time for short. The local time at a vertex $x$ by $\partial$-time $t$ is then
\begin{equation}
L_{t}(x) := \bfL_{ \bfL^{-1}_t(\partial)}(x)
\end{equation}
To distinguish between quantities in which time means real time and those in which time means $\partial$-time, we shall denote the former quantities in bold, e.g. $\bfL_u$.
The cover time in the $\partial$-time parameterization, or {\em cover $\partial$-time} is then
\begin{equation}
\label{e:1.4}
T^\rmC_n := \min \big\{t \geq 0:\: \min_{x \in \rmD_n} L_{t}(x) > 0 \big\} \,.
\end{equation}

Next, we also re-tune the rates of the underlying walk, so that the (edge) transition rates are now $1/(2\pi)$ instead of $1$. By the scaling properties of the exponential distribution, this modification only amounts to a change in the (scaling of the) cover time by a factor of $2\pi$. At the same time, this makes all constants a bit simpler and, as such, reduces visual complexity.

The analog of Theorem~\ref{t:A} for the new time parametrization and the revised rates is:
\begin{thm}
\label{t:1.1}
The collection of random variables
\begin{equation}
\label{e:1.6}
\left(\sqrt{T^\rmC_n} - \sqrt{t_n^\rmC} \quad :\quad   n > 0\right) \,,
\end{equation}
with $t_n^\rmC$ defined via the relation
\begin{equation}
\sqrt{t_n^\rmC} = \sqrt{2\pi} \sqrt{\bft_{n}^\rmC}
= \sqrt{2}n - \frac{1}{2 \sqrt{2}} \log n \,,
\end{equation}
is tight.
\end{thm}  
Most of this manuscript is devoted to proving this theorem, we present the top-level structure of its proof in the next subsection.

Once we have Theorem~\ref{t:1.1}, Theorem~\ref{t:A} follows rather easily. Indeed, we first observe that by definition,
\begin{equation}
	T_n^\rmC = \inf \big\{ t \geq 0 :\: \bfL^{-1}_t(\partial) > \bfT^\bfC_n \big\} \,.
\end{equation}
Since $t \mapsto \bfL_{t}^{-1}(\partial)$ is monotone increasing and right-continuous, this means that 
\begin{equation}
\label{e:2.6}
	\bfL_{T^\rmC_n-}^{-1}(\partial) \leq \bfT^\bfC_n \leq \bfL_{T^\rmC_n}^{-1}(\partial) \,,
\end{equation}
where the left-hand side is $\bfL^{-1}_{t-}(\partial) := \lim_{s \uparrow t} \bfL^{-1}_s(\partial)$ evaluated at $t = T^\rmC_n$.
We then only need to show:
\begin{prop}
\label{p:2.2}
The collection of random variables,
\begin{equation}
	\sqrt{\frac{\bfL_t^{-1}(\partial)}{|\wh{\rmD}_n|}} - \sqrt{t} \quad , \qquad
	\sqrt{\frac{\bfL_{t-}^{-1}(\partial)}{|\wh{\rmD}_n|}} - \sqrt{t} 
\qquad : n \geq 0 \,,\,\, 0 \leq t \leq |\wh{\rmD}_n| \,,
\end{equation}
is tight \,.
\end{prop}
Indeed:
\begin{proof}[Proof of Theorem~\ref{t:A}]
Under the revised transition rates of the walk and recalling that $n$ stands for a scaling factor of $N=\rme^n$, it suffices to show
\begin{equation}
\label{e:2.8}
\lim_{u \to \infty} \limsup_{n \to \infty} \bbP \Big(\sqrt{\bfT^\bfC_n/|\wh{\rmD}_n|} - \sqrt{2\pi} \sqrt{\bft_{n}^\rmC} > u\Big) \!= \!
\lim_{u \to \infty} \limsup_{n \to \infty} \bbP \Big(\sqrt{\bfT^\bfC_n/|\wh{\rmD}_n|} - \sqrt{2\pi} \sqrt{\bft_{n}^\rmC} <- u\Big) = 0 \,.
\end{equation}
Theorem~\ref{t:1.1} tells us that
\begin{equation}
\label{e:2.9}
\lim_{u \to \infty} \sup_n \bbP \Big(\sqrt{T^\rmC_n} - \sqrt{t_n^\rmC} > u/2\Big) = 
\lim_{u \to \infty} \sup_n \bbP \Big(\sqrt{T^\rmC_n} - \sqrt{t_n^\rmC} <- u/2\Big) = 0 \,.
\end{equation}
Now, suppose that the event in the first probability in~\eqref{e:2.8} occurs but 
the event in the first probability in~\eqref{e:2.9} does not. Then, by~\eqref{e:2.6} and monotonicity of $t \mapsto \bfL^{-1}_t$,
we must have
\begin{equation}
\sqrt{\frac{\bfL^{-1}_{(\sqrt{t_n^\rmC}+u/2)^2}(\partial)}{|\wh{\rmD}_n|}}
\geq 
\sqrt{\frac{\bfL^{-1}_{T_n^\rmC}(\partial)}{|\wh{\rmD}_n|}}
\geq \sqrt{\frac{\bfT_n^\rmC}{|\wh{\rmD}_n|}} > \sqrt{t_n^\rmC} + u 
\,.
\end{equation}
But then, by Proposition~\ref{p:2.2}, the probability that the above inequality holds tends to zero in the limit as $n \to \infty$ followed by $u \to \infty$. Together with the vanishing of the first limit in~\eqref{e:2.9} and the union bound, this shows the the first limit in~\eqref{e:2.8} is equal to zero.

In the same way, if the event in the second probability in~\eqref{e:2.8} occurs, but 
the event in the second probability in~\eqref{e:2.9} does not, then by~\eqref{e:2.6} and monotonicity of
$t \mapsto \bfL^{-1}_{t-}$,
\begin{equation}
\sqrt{\frac{\bfL^{-1}_{(\sqrt{t_n^\rmC}-u/2)^2-}(\partial)}{|\wh{\rmD}_n|}}
\leq 
\sqrt{\frac{\bfL^{-1}_{T_n^\rmC-}(\partial)}{|\wh{\rmD}_n|}}
\leq \sqrt{\frac{\bfT_n^\rmC}{|\wh{\rmD}_n|}} < \sqrt{t_n^\rmC} - u 
\,.
\end{equation}
This again goes to zero as $n \to \infty$ followed by $u \to \infty$ by Proposition~\ref{p:2.2} and, exactly as before, this implies that the second limit in~\eqref{e:2.8} holds in light of the second limit in~\eqref{e:2.9}.
\end{proof}


\subsection{Separation into Phases A and B}
Given $n > 0$, for $t,u \geq 0$ we let
\begin{equation}
\rmF_{n,t}(u) := \big\{ x \in \rmD_n :\: L_t(x) \leq u \big\} \,,
\end{equation}
so that $\rmF_{n,t}(0)$ is the set of unvisited vertices by $\partial$-time $t$. Since $\{T_n^\rmC \leq t\} = \{\rmF_{n,t}(0) = \emptyset\}$, we need to study this set at times $t$ where $\sqrt{t} = \sqrt{t_n^\rmC} + \Theta(s)$. 

To this end, we split the running $\partial$-time of the walk into two consecutive phases: A and B. The $\partial$-time at these phases will be $t_n^A$ and $t^B_n+sn$ respectively, where $s \in \R$ and  
\begin{equation}
\label{e:2.4.1}
\sqrt{t_n^A} := \sqrt{2} n - \frac{3}{4\sqrt{2}}  \log n 
\quad, \qquad
t_n^B : = \frac{1}{2} n\log n \,.
\end{equation}
Then, the total running $\partial$-time of the walk $t_n(s)$ will satisfy,
\begin{equation}
\label{e:2.5}
\sqrt{t_n(s)} \equiv \sqrt{t_n^A + t_n^B + sn} = \sqrt{t_n^\rmC} + \frac{1}{2\sqrt{2}} s + o(1) \,,
\end{equation}
where $o(1) \to 0$ as $n \to \infty$ for fixed $s$. Moreover, by definition, 
\begin{equation}
\label{e:2.16}
\rmF_{n,t_n(s)}(0) = \rmF_{n,t_n^A}(0) \cap \rmF'_{n,t_n^B+sn}(0) \,,
\end{equation}
where $\rmF'_{n,t_n^B+sn}(0)$ is the set of unvisited vertices during phase B. We shall therefore study the set of unvisited sites during each phase separately.

\subsubsection{Phase A}
Let us start with phase A, which takes up most of the work in the proof of Theorem~\ref{t:1.1}. Here we show that by the end of the phase, with high probability, most of the unvisited vertices are arranged in $\Theta(\sqrt{n})$ many clusters which are macroscopically separated and microscopic in size.  For a precise statement, we thus define the bulk of $\rmD_n$ as
\begin{equation}
	 \rmD_n^\circ:= \rmD_n^{n-2\log n}\,,
\end{equation} 
and the restriction of $F_{n,t_n^\rmA}(u)$ to this set by
\begin{equation}
\rmW_n(u) :=  \rmF_{n, t_n^A}(u) \cap \rmD_n^\circ \,.
\end{equation}
While not needed in this subsection, the inclusion of vertices with $O(1)$ but not necessarily zero local time in this definition will be important later on.

To define the clusters, we fix once and for all some $\eta_0 \in (0,1/2)$ (whose value is not important) and let
\begin{equation}
\label{e:2.3}
	r_n := n^{1/2-\eta_0} \,.
\end{equation} 
We then define 
\begin{equation}
	\bbW_n(u) := \Big \{ z \in \bbX_{\lfloor n-r_n \rfloor}  :\:  \rmW_n(u)  \cap \rmB(z; \lfloor n-r_n \rfloor) \neq \emptyset \Big \} \,.
\end{equation}
For $z \in \bbW_n(u)$, we shall refer to the set $\rmB(z;\lfloor n-r_n \rfloor) \cap \rmW_n(u)$ as the $(n-r_n)$-{\em cluster} of vertices with local time at most $u$ and {\em center} $z$. Notice that the same $(n-r_n)$-cluster can have more than one center. 

The restriction of $\bbW_n(u)$ to centers of $(n-r_n)$-clusters whose log-scale is $k \in \bbN_0$ is  then
\begin{equation}
\label{e:2.10}
\bbW^k_n(u) := \Big \{ z \in \bbW_n(u) :\: \rho\big(\rmB(z; \lfloor n-r_n \rfloor) \cap \rmW_n(u)\big) = k \Big\}
\end{equation}
with the cluster vertices themselves given by
\begin{equation}
\label{e:2.11}
\rmW^k_n(u) := \bigcup_{z \in \bbW^k_n(u)}  \rmW_n(u) \cap \rmB(z; \lfloor n-r_n \rfloor)\,.
\end{equation} 
The union of $\bbW^k_n(u)$ or $\rmW_n^k$ over all $k \in \N_0$ in some set $K \subset [0, \infty)$ will be denoted by $\bbW^K_n(u)$ and $\rmW_n^K$, respectively.

We thus show:
\begin{thm}[Phase A]
\label{t:2.1}
With the definitions above,
\begin{equation}
\lim_{\delta \downarrow 0} \limsup_{n \to \infty}
\bbP \bigg(\frac{ \big|\bbW_n(0)\big|}{\sqrt{n}} \notin \big(\delta, \delta^{-1} \big) \bigg) = 0 \,.
\end{equation}
Moreover,
\begin{equation}\label{eq:phAex}
\tfrac{1}{\sqrt{n}} \Big|\rmW^{[r_n, n-r_n]}_n(0) \Big|
\overset{\bbP}{\underset{n \to \infty}\longrightarrow} 0 
\quad  \text{and} \quad
\Big| \rmF_{n, t_n^A}(0) \setminus \rmW_n(0)  \Big|
\overset{\bbP}{\underset{n \to \infty}\longrightarrow} 0 \,.
\end{equation}
\end{thm}
The theorem states that at time $t_n^A$, with high probability, except for a subset of size $o(\sqrt{n})$, all unvisited vertices are arranged in ``clusters'' whose log-scale is $O(r_n)$. Moreover, the number of such clusters is $\Theta(\sqrt{n})$ and they all lie well inside the bulk of $\rmD_n$.

\subsubsection{Phase B}
Next we address the set of unvisited sites during phase B. Here, we show that $t_n^B + \Theta(n)$ is the right ``order'' for the $\partial$-time it takes our walk to visit all vertices in clustered sets such as the set of unvisited sites by the end of phase A. 

To be more explicit, given $n > 0$, we shall say that a set $\rmA \subset \rmD_n$ is {\em $(r_n, n-r_n)$-clustered}, if for all $z \in \bbX_{\lfloor n-r_n \rfloor}$ it holds that $\rho(\rmB(z;\lfloor n-r_n \rfloor) \cap \rmA) \leq r_n$.
It is not difficult to see that if $\rmA$ is such a set then for all pairs $x,y \in \rmA$ we have that $\log \rmd(x,y) \notin [\lfloor r_n \rfloor+1,\lfloor n-r_n \rfloor-2]$. The {\em clustering number} of $\rmA$ will be denoted by $\chi_n(\rmA)$ and defined as the minimal $k \in \N$ for which there exist $z_1, \dots,z_k \in \bbX_{\lfloor n-r_n \rfloor}$ such that 
$\rmA \subset \cup_{i=1}^k \rmB(z_i; \lfloor n-r_n \rfloor)$. 

We then show:
\begin{thm}[Phase B]
\label{t:2b}
For all $s \in \bbR$,
\begin{equation}
\lim_{n \to \infty} \sup_{\rmA} 
\bigg| \bbP \Big( \rmA \cap \rmF_{n,t_n^{B}+sn}(0)  = \emptyset \Big) -
	\exp \bigg( -\frac{\rme^{-s}\chi_n(\rmA)}{\sqrt{n}} \bigg) \bigg| = 0 \,,
\end{equation}
where the supremum above is over all $(r_n, n-r_n)$-clustered sets $\rmA \subset \rmD^\circ_n$. Moreover, for any fixed $s \in \bbR$, if $n$ is large enough then for all $\rmE \subseteq \rmD_n$ (not necessarily clustered),
\begin{equation}
\label{e:2.10a}
\bbP \Big( \rmE \cap  \rmF_{n,t_n^{B}+sn}(0)  \neq \emptyset \Big) \leq 2 \rme^{-s}\frac{|\rmE|}{\sqrt{n}} \,.
\end{equation}
\end{thm}
The theorem says that the probability that an $(r_n, n-r_n)$-clustered set is entirely covered by a random walk on $\wh{\rmD}_n$ is asymptotically the probability that a Poisson random variable with rate $\rme^{-s} \chi_n(\rmA)/\sqrt{n}$ is equal to zero (the proof of the theorem actually shows that the clustering number of $\rmA \cap \rmF_{n,t_n^{B}+sn}(0)$ is asymptotically such a Poisson). It also says that any set of size $o(\sqrt{n})$ will be entirely visited with high probability. 

\subsection{Proof of the main theorem}
Combining Theorem~\ref{t:2.1} and Theorem~\ref{t:2b}, the proof of Theorem~\ref{t:1.1} is fairly straightforward:
\begin{proof}[Proof of Theorem~\ref{t:1.1}] 
In view of~\eqref{e:2.5} it will suffice to show that
\begin{equation}
\label{e:2.13}
0 = \lim_{s \to -\infty} \limsup_{n \to \infty} \bbP (\rmF_{n,t_n(s)}(0) = \emptyset)
\leq \lim_{s \to \infty} \liminf_{n \to \infty} \bbP (\rmF_{n,t_n(s)}(0) = \emptyset) = 1 \,.
\end{equation} 
Setting $\rmA := \rmW_n(0) \setminus \rmW_n^{[r_n,n-r_n]}(0)$  and recalling~\eqref{e:2.16}, the set $\rmF_{n,t_n(s)}(0)$ is equal to 
\begin{equation}
\label{e:2.14}
\big(\rmA \cap \rmF'_{n,t_n^B+sn}(0) \big) 
	\cup \big(\big(\rmF_{n,t_n^A}(0) \setminus \rmW_n(0)\big) \cap \rmF'_{n,t_n^B+sn}(0) \big) 
	\cup \big(\rmW_n^{[r_n, n-r_n]}(0) \cap \rmF'_{n,t_n^B+sn}(0) \big)  \,.
\end{equation}
As in~\eqref{e:2.16}, the set $\rmF'_{n,t_n^B+sn}(0)$ above is that of the unvisited vertices during phase B and, thanks to the strong Markov property, it is independent of $\rmF_{n,t_n^A}$.

The second and third sets in the above union are empty with probability tending to one as $n \to \infty$. Indeed, for the second set we can just use the last assertion in Theorem~\ref{t:2.1}.
For the~third set, by conditioning on the trajectory of the random walk during phase A (that is, up to $\partial$-time $t_n^A$) and using the second part of Theorem~\ref{t:2b} with 
$\rmE=  \rmW_n^{[r_n, n-r_n]}(0)$, for any $\delta > 0$ and $s \in \bbR$ we get 
\begin{equation}
\bbP \big( \rmW_n^{[r_n, n-r_n]}(0) \cap \rmF'_{n,t_n^\rmB+sn}(0) \neq \emptyset  \big) 
\leq  \bbP \big( \big|\rmW_n^{[r_n, n-r_n]}(0)\big| > \delta \sqrt{n} \big) +
2 \rme^{-s} \delta \,,
\end{equation}
for all $n$ large enough.  Taking first $n \to \infty$ and then $\delta \to 0$, by the second part of Theorem~\ref{t:2.1} we obtain that for all $s$,
\begin{equation}
\label{e:2.17}
\lim_{n \to \infty} \bbP \big( \rmW_n^{[r_n, n-r_n]}(0) \cap \rmF'_{n,t_n^\rmB+sn}(0) \neq \emptyset  \big)  = 0 \,.
\end{equation}

Focusing therefore on the first set of the union in~\eqref{e:2.14}, we first observe that $\rmA$ is $(r_n,n-r_n)$-clustered by definition and that its number of clusters satisfies
\begin{equation}
\tfrac19 \big(|\bbW_n(0)| - |\rmW_n^{[r_n, n-r_n]}(0)|\big) \leq \chi_n(\rmA) \leq |\bbW_n(0)| \,.
\end{equation}
Then, by invoking Theorem~\ref{t:2.1} and the union bound, one gets
\begin{equation}
	\lim_{\delta \downarrow 0} \limsup_{n \to \infty}
	\bbP \bigg(\frac{ \chi_n(\rmA)}{\sqrt{n}} \notin \big(\delta, \delta^{-1} \big) \bigg) = 0 \,.
\end{equation}
Conditioning again on the walk up to time $t_n^A$ and appealing now to the first part of Theorem~\ref{t:2b}, we get
\begin{equation}
\rme^{- \rme^{-s} \delta^{-1}} - g_1(n,\delta) - g_2(n,s) \leq
\bbP \big( \rmA \cap \rmF'_{n,t_n^B+sn}(0) = \emptyset  \big) 
\leq
\rme^{-\rme^{-s} \delta} + g_1(n,\delta) + g_2(n,s) \,,
\end{equation}
for all $s$ and $\delta$, where $g_1(n,\delta) \to 0$ as $n \to \infty$ followed by $\delta \to 0$ and $g_2(n,s) \to 0$ as $n \to \infty$ for each $s$. 
In light of \eqref{e:2.14}, \eqref{e:2.17} and the last assertion in Theorem~\ref{t:2.1}, in order to obtain~\eqref{e:2.13} it only remains to take $n \to \infty$, then $s \to \pm \infty$ and finally $\delta \to 0$ in the last display.
\end{proof}

\subsection{Overview of the rest of the proof}
\label{ss:2.4}
We conclude this section with an overview of the proofs of Theorem~\ref{t:2.1}, Theorem~\ref{t:2b}. We start with the proof of the former, which occupies the largest part of this manuscript. In the heart of this proof is a comparison argument, in which the law of the local time field $(L_t(x))_{x \in \rmD_n}$ is related to the law of the discrete Gaussian free field (DGFF) $(h_n(x))_{x \in \rmD_n}$ on the same underlying domain. This comparison allows us to use the well-developed extreme value theory for the DGFF to study the min-extremes of the local time field which are, trivially, those vertices which were not visited by the walk by time $t$.

To state things more precisely, let us first recall that the DGFF on $\rmD_n$ is a centered Gaussian field $h_n=(h_n(x))_{x \in \rmD_n}$ with covariance given by $\bbE h_n(x) h_n(y) = \frac12 G_n(x,y)$ for $x,y \in \rmD_n$ where, as before, $\rmG_n$ is the discrete Green Function on $\rmD_n$. Extreme values of this field have been studied extensively in the past few decades. In particular, it was shown that the all the min-extremes of the field (the $k$ lowest values for any finite $k$) take $\Theta(1)$ values around $-m_n$, where
\begin{equation}
m_n	:= \sqrt{2} n - \frac{3}{4\sqrt{2}}  \log n \,.
\end{equation}
Moreover, the size of the extremal level sets are finite almost-surely so that, for any $u > 0$, the sizes of the following sets
\begin{equation}
\rmG_n(u) := \big \{ x \in \rmD^\circ_n :\: h_n(x) \in -m_n + [-\sqrt{u}, \sqrt{u}] \big \} \,.
\end{equation}
are a tight sequence in $n$ (for each fixed $u$) and positive with probability tending to $1$ as $u \to \infty$ uniformly in $n$.

At the same time, a result by Eisenbaum, Marcus, Kaspi, Rosen, Shi from '00~\cite{Gang}, known as the generalized second Ray-Knight Theorem, or (somehow inaccurately) Dynkin's Isomorphism, shows that, for any $t \geq 0$, $n \geq 0$, there exists of a coupling between $L_t$ and two identically distributed copies of the the DGFF, $h_n$ and $h'_n$, such that $h_n$ and $L_t$ are independent and 
\begin{equation}
\label{e:2.35}
L_{t}(x) + h^2_n(x) = 
\big(h'_n(x) + \sqrt{t})^2 
\quad ; \qquad \forall x \in \rmD_n \,
\end{equation}
(see Theorem~\ref{t:103.1} below for a more precise formulation).

Now, at $t = t_n^A$, i.e. at the end of phase A, since (by choice) $\sqrt{t_n^A} = m_n$, one has
\begin{equation}
\rmG_n(u) = \big\{x \in \rmW_n(u) :\: h^2_n(x) \leq u - L_t(x) \big\} \,,  
\end{equation} 
so that a vertex in $\rmW_n(u)$ is also in $\rmG_n(u)$ if its corresponding value under $h_n$ is $O(1)$ and sufficiently small. Since $h_n(x)$ is a centered Gaussian with variance 
$\frac12 G_n(x,x) = \frac12 n (1+ o(1))$ for $x$ in the bulk, the probability that it takes an $O(1)$ value is $\asymp \frac{1}{\sqrt{n}}$. Had $\big(h_n(x) :\: x \in \rmW_n(u)\big)$ been independent, tightness and positivity with high probability (as $u \to \infty$, uniformly in $n$) of $|\rmG_n(u)|$ would have implied the same for $\frac{1}{\sqrt{n}}|\rmW_n(u)|$. 

Studying the asymptotics of the Green function, we see that 
\begin{equation}
\label{e:2.38}
\Cov \big(h_n(x), h_n(y)\big) = o(n) 
\quad \text{if } \log \|x-y\| = n - o(n)\,,
\end{equation}
  so that at macroscopic scales one indeed has approximate independence. The problem is that vertices in $\rmW_n(u)$ are not macroscopically separated. In fact, one would expect a clustering of vertices with $O(1)$ local time, due to the ``locality'' of the underlying random walk dynamics. This clustering picture is indeed confirmed by Theorem~\ref{t:3.1} and Proposition~\ref{p:6.1}, in which the small scales $[r, r_n]$ and, resp., large scales $[r_n, n-r_n]$ are handled separately. We refer to the result of Theorem~\ref{t:3.1} as {\em sharp clustering} and that of Proposition~\ref{p:6.1} as {\em coarse clustering}.
  
More explicitly Theorem~\ref{t:3.1} implies that
\begin{equation}
\label{e:2.39}
\forall \delta > 0 \ , \quad 
\lim_{r \to \infty}
\limsup_{n \to \infty} \bbP \Big( \big| \bbW_n^{[r, r_n]}(u) \big| > \delta \sqrt{n}  \Big) = 0 \,,
\end{equation}
while Proposition~\ref{p:6.1} implies that
\begin{equation}
\label{e:2.39b}
\forall \delta > 0 \ , \ \ 
\lim_{n \to \infty} \bbP \Big( \big|\bbW^{[r_n, n-r_n]}_n(u) \big| > \delta \sqrt{n} \Big) = 0 \,.
\end{equation}
Together they show that outside a subset of size $o(\sqrt{n})$ with high probability, all $(n-r_n)$-clusters have (essentially) a finite diameter. Thus, the $O(1)$-local time vertices are (mostly) arranged in clusters which are finite in size and macroscopically separated.

This leads to considering the number of clusters $|\bbW_n(u)|$ instead of the number of $O(1)$-local time vertices themselves. Since (again by studying the Green function),
\begin{equation}
\Var\, (h_n(x) - h_n(y)) = O(1) \quad \text{if }\|x-y\| = O(1) \,,
\end{equation}
this together with~\eqref{e:2.38} and the previous reasoning shows that $|\rmG_n(u)|$ is essentially a {\em Binomial Thinning} of $|\bbW_n(u)|$ by $\Theta(1/\sqrt{n})$, so that as before, one gets tightness and positivity with high probability (as $u \to \infty$, uniformly in $n$) of $\frac{1}{\sqrt{n}}|\bbW_n(u)|$. This readily gives tightness of $\frac{1}{\sqrt{n}}|\bbW_n(0)|$ but an argument is still required to convert the positivity with high probability of $\frac{1}{\sqrt{n}}|\bbW_n(u)|$ to that of $\frac{1}{\sqrt{n}}|\bbW_n(0)|$ (see Lemma~\ref{l:7.4}). Altogether this gives (the first part of) Theorem~\ref{t:2.1}.

We now turn to the proof of the sharp (Theorem~\ref{t:3.1}) and coarse (Proposition~\ref{p:6.1}) clustering results, beginning with the former, which is the more involved of the two. Since the min-extremes of the DGFF are known to be clustered, it makes sense to try to use again a comparison argument to demonstrate a similar structure for the local time field. For the DGFF, clustering of extreme-value vertices can be viewed as a consequence of the entropic repulsion of the trajectory of the average field height along a sequence of sub-domains (e.g. balls) shrinking to a given extreme point. 

More precisely, let $\ol{h_n}(x;k)$ denote the harmonic average of the DGFF on $\partial \rmB(x;k)$ for some $x \in \rmD_n$. In the case of the DGFF, $\ol{h_n}(x;k)$ turns out to be equal, up to $O(1)$ corrections, to the conditional mean of the field in the bulk of $\rmB(x;k)$ given its values on $\partial \rmB(x;k)$.
\FromS{By this you mean that the binding field is close to the standard average? Perhaps you want to say ``average'' as before to enforce this rather than mean?} 
\ToS{Nope - conditional mean. The $O(1)$ comes from the fact that the harmonic average is actually the conditional mean at $x$, while we look at a general vertex in the bulk of $\rmB(x;k)$.}
Then, defining for any small $\eta > 0$ and $K \subseteq [0,n]$,
\begin{equation}
\label{e:2.40a}
\wh{\rmG}_{n}^K(u) := \Big\{ x \in \rmG_n(u) :\: \ol{h_n}(x;k) > -\tfrac{n-k}{n} m_n + k^{1/2-\eta} 
\ , \ \  \forall k \in K \Big\}\,,
\end{equation}
it can be shown that,
\begin{equation}
\label{e:2.42a}
\forall \delta > 0 \ , \ \ 
\lim_{r \to \infty} \limsup_{n \to \infty} \bbP \Big(\big|\rmG_{n}(u) \setminus \wh{\rmG}_{n}^{[r, r_n]} (u)\big| > \delta \Big) = 0 \,.
\end{equation}
Thus at min-extreme value vertices, the conditional mean of the field decreases linearly to $-m_n + O(1)$, up-to a smaller order $k^{1/2-\eta}$ repulsion term. It is due to this repulsion term that there are no additional min-extreme values at logarithmic distance $k \in [r, r_n]$ from $x$. Indeed, the recentered (by the conditional mean) field at scale $k$, which is in fact a DGFF on $\rmB(x;k)$, now needs to produce additional values in the bulk of $\rmB(x;k)$ at height smaller than
\begin{equation}
\label{e:2.42}
-m_n - \big(-\tfrac{n-k}{n} m_n + k^{1/2-\eta})  \leq -m_k -  k^{1/2-\eta} \,,
\end{equation}
which is probabilistically unlikely.

To obtain an analogous repulsion picture for the local time field, we appeal again to the Isomorphism (\eqref{e:2.35} with $t=t_n^A$). Taking harmonic averages of both sides and exchanging the order of squaring and averaging (which one may do up to controllable errors), we get
\begin{equation}
\label{e:2.43}
\ol{L_t}(x;k) + \ol{h_n}(x;k)^2 \cong
\big(\ol{h'_n}(x;k) + m_n\big)^2 
\quad ; \qquad x \in \rmD_n,\, k \geq 0 \,.
\end{equation}
Since the law of $h$ is Gaussian with well-controlled covariances, conditional on the event $\{h_n(x) = \ol{h_n}(x;0) = O(1)\}$, it is easy to show that $\ol{h_n}(x;k)^2 = O(k)$ with high probability. This makes the second term in the left-hand side negligible compared to the right hand side of~\eqref{e:2.43}, whenever $k \in [r, r_n]$ and $x$ is in $\wh{\rmG}^{[r,r_n]}_{n}(u)$. Indeed, in this case,
\begin{equation}
\begin{split}
\big(\ol{h'_n}(x;k) + m_n\big)^2  & \geq \big(\tfrac{k}{n}m_n + k^{1/2-\eta}\big)^2 
\cong \big(\sqrt{t_k^\rmC} + k^{1/2-\eta}\big)^2  \\ 
& \cong \big(\sqrt{2} k + k^{1/2-\eta}\big)^2 \gg \ol{h_n}(x;k)^2\,.
\end{split}
\end{equation}

Setting in analog to~\eqref{e:2.40a}, 
\begin{equation}
\label{e:2.46}
\begin{split}
\wh{\bbW}^K_{n}(u) := \Big\{ z \in \bbW_n(u) :\: & 
\sqrt{\ol{L_t}(x;k)} > \sqrt{t_k^\rmC} + k^{1/2-\eta}\,, \\
& \qquad \forall x \in \rmW_n(u) \cap \rmB(z; \lfloor n-r_n\rfloor)\,,\,\, k \in K \Big\} \,,
\end{split}
\end{equation} 
and using the binomial thinning argument from before (which, for the sake of an upper bound, does not require a priori clustering), one obtains from~\eqref{e:2.42a},
\begin{equation}
\label{e:2.47}
\forall \delta > 0 \ , \ \ 
\lim_{r \to \infty} \limsup_{n \to \infty} \bbP \Big(\big|\bbW_{n}(u)\big| - \big|\wh{\bbW}_n^{[r, r_n]} (u)\big| > \delta \sqrt{n} \Big) = 0 \,.
\end{equation}

In order to get from~\eqref{e:2.47} to~\eqref{e:2.39}, one would would now want to use the repulsion of the average local time:
\begin{equation}
\label{e:2.48}
\sqrt{\ol{L_t}(x;k)} > \sqrt{t_k^\rmC} + k^{1/2-\eta} \,,
\end{equation}
to claim that there are no additional $O(1)$ local time vertices in the bulk of $\rmB(x;k)$, other then those near $x$. This makes sense as it is plausible that for $s \geq 0$,
\begin{equation}
\label{e:2.49}
\begin{split}
	\Big(\big(L_t(z)\big)_{z \in \rmB(x;k)} \Big| \ol{L_t}(x;k) = s\Big) 
& \overset{d} =
	\Big(\big(L_t(z)\big)_{z \in \rmB(x;k)} \Big|  \frac{1}{
	|\partial \rmB(x;k)|}
	\sum_{z \in \partial \rmB(x;k)} L_{n,t}(z) \cong s 
	 \Big) \\
& \overset{d} \cong \big(\wt{L}_{s}(y)\big)_{y \in \rmB(x;k)} \,,
\end{split}
\end{equation}
where the right-hand side is the local time field for a random walk on $\rmB(x;k)$ with wired boundary condition on $\partial \rmB(x;k)$ run until $\partial$-time $s$. Thus, at times $s$ which are much higher than the order of the cover time $t_k^\rmC$, such as those given in~\eqref{e:2.48}, additional $O(1)$ local time vertices are indeed unlikely.

While a statement such as~\eqref{e:2.49} may be provable, the overall argument requires that it also holds conditionally, when values of the local time field on $\rmB(x;k)^\rmc$ (or at least at some away distance from $\rmB(x;k)$) are specified. In other words, one needs to show an approximate, quantitative spatial Markov property for the local time field, under conditioning on the total local time spent on boundaries of sub-domains (which  handles conditioning on potentially exotic values for the field on the complement of the sub-domain). 

Since we could not derive strong enough approximate Markov estimates, we chose to take a different route and follow the approach, pioneered in~\cite{dembo2004cover}, of using downcrossings. To define the latter, given $l > k > 0$ and $x \in \bbZ^2$ such that $\ol{\rmB(x; l)} \subseteq \rmD_n$, we denote by $N_t(x; k,l)$ the number of times that the walk has crossed from $\rmB(x;l)^\rmc$ to $\rmB(x;k)$ by $\partial$-time $t$, and refer to each such crossing as a $(x;k,l)$-{\em downcrossing}. The advantage of working with downcrossings is that they provide the needed approximate spatial-Markov structure for the otherwise non-Markovian local time field. Indeed, conditioning on $N_t(x;k,l)$ and on the entry (in $\partial (\rmB(x;k)^\rmc)$) and exit (in $\partial \rmB(x;l)$) in each of the $N_t(x;k,l)$ downcrossings, the law of $(L_t(y) :\: y \in \rmB(x;k))$ is independent of $(L_t(y) :\: y \in \rmB(x;l)^\rmc)$ and does not depend on $x$. See Subsection~\ref{ss:locdownprelim} for precise definitions and statements.

The connection between the harmonic average of local time and number of downcrossings is given by
\begin{equation}
\label{e:2.50}
		\bbP \Big(\big|\ol{L_{t}}(x;k') - m\big| > z
		\,\Big| \,(l-k){N}_{t}(x;k, l) = m\,, 
			\text{ all entry and exit points}
		\Big) 
		\leq \rme^{-c\frac{z^2}{(l-k) m}}\,,
\end{equation} 
which holds for all $k' \ll k \ll l$ and all reasonable ranges of $m$ and $z$. See Proposition~\ref{l:4.4b} for a precise statement.
If we then introduce for $\gamma > 0$,
\begin{equation}
	\wh{N}_t(x; k) := \frac12 k^\gamma N_t \big(x; k+\tfrac12 k^\gamma, k+k^\gamma \big) \,,
\end{equation}
and, in analog to~\eqref{e:2.46}, define
\begin{equation}
\label{e:2.52}
\begin{split}
\wt{\bbW}^K_{n}(u) := \Big\{ z \in \bbW_n(u) :\: & 
\sqrt{\wh{N}_t(x;k)} > \sqrt{t_k^\rmC} + k^{1/2-\eta}\,, \\
& \qquad \forall x \in \rmW_n(u) \cap \rmB(z; \lfloor n-r_n\rfloor)\,,\,\, \forall k \in K \Big\} \,,
\end{split}
\end{equation}
then, using~\eqref{e:2.50} and~\eqref{e:2.47}, one can show that for $\gamma$ suitably small, 
\begin{equation}
\label{e:2.51}
\forall \delta > 0 \ , \ \ 
\lim_{r \to \infty} \limsup_{n \to \infty} \bbP \Big(\big|\bbW_{n,r}(u)\big| - \big|\wt{\bbW}_n^{[r, r_n]} (u)\big| > \delta \sqrt{n} \Big) = 0 \,.
\end{equation}
A first moment bound together with well-known estimates for
\begin{equation}
\bbP\Big( \sqrt{L_t(y)} \leq v \,\Big|\, \sqrt{\wh{N}_t(x;k)} \geq \sqrt{t_k^\rmC} + s\,,\,\,\text{all entry and exit points}\Big) \,,
\end{equation} 
(see Proposition~\ref{prop:Abulk2}) are then sufficient to convert the repulsion of the downcrossing trajectories in~\eqref{e:2.51} to the required sharp clustering statement in~\eqref{e:2.39}. We remark that the proof of Theorem~\ref{t:3.1} is a bit more involved than what was described above, as the theorem includes, in addition, the statement
\begin{equation}
	\lim_{r \to \infty}
\limsup_{n \to \infty} \bbP \Big( \rmW_n^{[r, r_n]}(u) \cap  \rmG_n(u)  \neq \emptyset \Big) = 0 \,,
\end{equation}
which does not follow from~\eqref{e:2.39}.

As in Theorem \ref{t:3.1}, the proof of Proposition~\ref{p:6.1} relies on showing repulsion for the local time trajectories of vertices with $O(1)$ local time. Derived, in turn again, from a similar result for the DGFF, the repulsion statement needed here is weaker, as it involves the bulk scales $k \geq r_n$. On the other hand, the real statement in Proposition~\ref{p:6.1} is stronger than that 
	presented in~\eqref{e:2.39b} (and in analog that of~\eqref{e:2.39}), namely
\begin{equation}
\forall \delta > 0 \ , \ \ 
\lim_{n \to \infty} \bbP \Big( \big|\rmW^{[r_n, n-r_n]}_n(u) \big| > \delta \sqrt{n} \Big) = 0 \,.
\end{equation}
Thus, the total size (and not just the number) of clusters with large log-scale is $o(\sqrt{n})$ with high probability.

Moving to Theorem~\ref{t:2b} (Phase B), here the proof is rather straightforward. It can be easily shown that, once the target set $\rmA$ is $(r_n, n-r_n)$-clustered, then with high probability. (as $n \to \infty$) in each of the excursions away from $\partial$ during Phase B, either no cluster of $\rmA$ is visited or exactly one such cluster is visited and, when it does, it gets completely covered by the walk before returning to $\partial$. In other words, no excursion visits two clusters of $\rmA$ at the same time or visits a cluster partially. Moreover, the probability that a cluster is visited during one excursion is $(2\pi+o(1))\deg(\partial)^{-1}n^{-1}$. Since the number of excursions is Poisson with rate $(2\pi)^{-1}\deg(\partial)(t_n^\rmB + sn)$, it follows from standard Poisson Thinning that a given cluster of $\rmA$ is not visited during phase B with probability
\begin{equation}
 1-\frac{1}{\sqrt{n}}\rme^{-s} \big(1 + o(1)\big) 
= \exp \Big(-\frac{\rme^{-s}\big(1+o(1)\big)}{\sqrt{n}}\Big)\,,
\end{equation}
independently of the other clusters. This readily gives the first part of Theorem~\ref{t:2b}. The second is a simple first moment bound.

\subsection*{Organization of the Paper}
The remainder of the paper is organized as follows. In Section~\ref{s:3} we include general results for random walks and the discrete Gaussian free field, which are needed in the sequel. As they are standard and sometimes lengthy, their proofs are deferred to Appendix~\ref{s:A}. Two key technical lemmas, which we call the {\em Thinning Lemma} and the {\em Resampling Lemma} are then presented in Section~\ref{s:4}. They are used repeatedly in the proof of the sharp clustering result. The statement and proof of sharp clustering is the subject of Section~\ref{s:5}, with the coarse clustering counterparts being the subject of Section~\ref{s:6}. Building on the sharp and coarse clustering results, the proof of the main theorem for phase A, Theorem~\ref{t:2.1}, is presented in Section~\ref{s:7}. This section also contains the relatively soft argument for the main theorem of phase B, Theorem~\ref{t:2b}. Lastly,  Section~\ref{s:8} includes the proof of the Time-Reparameterization Proposition~\ref{p:2.2}.

\section{Preliminaries}
\label{s:3}

In this section we collect various basic statements that we will use throughout the sequel. These are mostly coarse estimates, which will be used as building blocks to derive the sharper estimates which are in turn needed to derive the tightness of the cover time. All statements here are either taken from existing literature, standard (even if sometimes non-trivial) adaptations of existing results or simple to prove. We shall therefore defer all their proofs to the appendix. 

\subsection{Random walk preliminaries}
Recall that our continuous time simple random walk $X$ runs in $\wh{\rmD}_n$, transitions at (edge) rate~$1/2\pi$ and starts from $\partial$. To denote a different initial vertex $x \in \bbZ^2$, we shall write $\bbP_x$ and $\bbE_x$ for the underlying probability and expectation. The hitting-time-of and first-return-time-to the subset $\rmA \subseteq\bbZ^2$ will be denoted by $\tau_\rmA$ and $\ol{\tau}_\rmA$ respectively. In both cases we shall use the subscript $x$ if $\rmA = \{x\}$ is a singleton. 

\subsubsection{Discrete harmonic analysis and hitting time estimates}\label{s:dga}
Given $\emptyset \subsetneq \rmU \subsetneq \bbZ^2$, we let $G_\rmU: \ol{\rmU} \times \ol{\rmU} \to [0,\infty)$ and $\Pi_\rmU: \ol{\rmU} \times \partial \rmU \to [0,1]$
denote respectively the Green Function and Poisson Kernel associated with a continuous time simple random walk on $\bbZ^2$ with (edge) transition rates $1/2\pi$, which is killed upon exit from $\rmU$. Explicitly, if $\wt{X} = (\wt{X}_t :\: t \geq 0)$ is a continuous time simple random walk on $\Z^2$ with the aforementioned transition rates then, for any $x,y \in \ol{\rmU}$ and $z \in \partial \rmU$,
\begin{equation}
G_{\rmU}(x,y) := \bbE_x \int_{0}^{\tau_{\partial \rmU}} 1_y(\wt{X}_s) \rmd s 
\qquad \text{ and } \qquad 
\Pi_{\rmU}(x,z) := \bbP_x \big(\wt{X}_{\tau_{\partial \rmU}} = z\big) \,.
\end{equation} 
Above, as in the case of $X$, we write $\bbP_x$ and $\E_x$ to indicate that $\wt{X}_0 \equiv x$ and we write $\tau_{\partial \rmU}$ for the hitting time of $\partial \rmU$ by $\wt{X}$.
We remark that our Green Function is $\pi/2$ times the usual one, i.e., the inverse of the negative discrete Laplacian on $\rmU$, with zero boundary conditions on $\partial \rmU$. Our Potential Kernel $a:\bbZ^2 \times \bbZ^2 \to [0,\infty)$ (see \cite[Section~4.4]{LL} for a precise definition) will admit the same normalization, so that $a(0)=0$ and for $x \neq 0$ we have
\begin{equation}
\label{e:3.2}
a(x) \equiv a(0,x) = \log \|x\| + \gamma^* + O\big(\|x\|^{-2} \big) \,,
\end{equation}
for some $\gamma^* \in (0,\infty)$ whose value is related to Euler's constant (see, e.g., \cite[Theorem~4.4.4]{LL}). The last three objects defined are related via the equation
\begin{equation}\label{e:rel}
	G_{\rmU}(x,y)= \sum_{z \in \partial \rmU} \big[a(z-x) - a(y-x)\big] \Pi_{\rmU}(y,z) \,.
\end{equation}

Given $\rmU$ as above and $f: \bbZ^2 \to \bbR$, we let $\ol{f}_{\rmU}$ denote the unique real-valued function on $\ol{\rmU}$ which agrees with $f$ on $\partial \rmU$ and is discrete harmonic and bounded in $\rmU$. That is, for $x \in \ol{\rmU}$, 
\begin{equation}
	\ol{f}_{\rmU}(x) = \sum_{z \in \partial \rmU}  \Pi_\rmU(x, z) f(z)\,.
\end{equation}
Whenever $\rmU = \rmB(x;k)$, we abbreviate 
\begin{equation}
	\label{e:3.4}
	\ol{f}(x;k) := \ol{f}_{\rmB(x;k)}(x) \,.
\end{equation}

Throughout the sequel, we will say that a set $\cD \subseteq \R^2$ is \textit{nice} if it satisfies the same properties as the domain $\rmD$ from Section~\ref{s:1}. The most common example of a nice set used in this manuscript, besides our main domain $\rmD$, is that of the unit disc in $\bbR^2$, $\cB:=\{ x \in \bbR^2 : \| x \| < 1\}$. Notice that for each scale $k \in \N_0$ we have $\rmB(0;k) = \cB_{n(k)}$ for $n(k)=\log(\rme^k +1) =k + O(\rme^{-k})$. This will be used in the sequel to derive estimates for balls $\rmB(x;k)$ from estimates for $N$-scales of the set $\cB$. 

The first of such estimates is Lemma~\ref{l:103.2} below, which contains some elementary bounds for the Green Function on scales of a nice set $\cD$.
\begin{lem}
\label{l:103.2}
For any nice set $\cD \subseteq \bbR^2$ and $q > 0$ there exist $C = C(\cD)$ and $C' = C'(\cD, q) < \infty$ such that for all $n \geq 0$, the following holds: 
\begin{enumerate}
	\item For all $x,y \in \cD_n$,
\begin{equation}
\label{e:3.19.1}
 G_{\cD_n}(x,y) \,-\, \big(n - \log (\|x-y\| \vee 1) \big) \leq C \,,
\end{equation} 
and for all $x,y \in \cD_n^{n-q}$,
\begin{equation}
	\label{e:3.19.1b}
	G_{\cD_n}(x,y) \,-\, \big(n - \log (\|x-y\| \vee 1) \big) \geq -C' \,.
\end{equation} 
\item For all $x,y \in \cD_n$, 
\begin{equation}
\label{e:3.19.2a}
\big| G_{\cD_n}(x,y) - \Big(\log \rmd \big(\{x,y\}, (\cD_n)^\rmc\big) - a(x-y) \Big) \Big| \leq C \bigg( 1 + \frac{\|x-y\|}{\rmd(\{x,y\}, (\cD_n)^\rmc)}\bigg) \,.
\end{equation}
In particular for all $x \in \cD_n$, 
\begin{equation}
\label{e:3.19.2}
-C + \log \rmd \big(x, (\cD_n)^\rmc \big) \leq  G_{\cD_n}(x,x) \leq \log \rmd \big(x, (\cD_n)^\rmc \big) + C \,.
\end{equation}
 
\item
For all  $x,y \in \cD_n$,
\begin{equation}
\big|\big(G_{\cD_n}(x,x) + G_{\cD_n}(y,y) - 2 G_{\cD_n}(x,y)\big) - 2a(x-y) \Big|\! \leq \! \frac{2\|x-y\|}{\rmd(\{x,y\}, (\cD_n)^\rmc)} + C \rmd\big(\{x,y\}, (\cD_n)^\rmc\big)^{-2} \,.
\end{equation}
\end{enumerate}
\end{lem}

The following are simple but useful estimates for the probabilities of hitting balls inside $\rmD_n$. We point out that, even if we state these estimates only for the random walk on $\wh{\rmD}_n$ in Section~\ref{s:1} to avoid introducing additional notation, the same estimates hold (and will be used) whenever $\rmD$ is replaced by any other nice set $\cD \subseteq \bbR^2$.

\begin{lem}\label{lem:g} For all $x \in \rmD_n$, 
\begin{equation}\label{e:h1}
\bbP_\partial( \tau_x < \ol{\tau}_\partial) = \frac{2\pi}{\deg(\partial)}\frac{1}{G_{\rmD_n}(x,x)} \,,
\end{equation}
where $\deg(\partial)$ above stands for the degree of $\partial$ in $\wh{\rmD}_n$.
Furthermore, if $n$ is sufficiently large then, given any $x\in \rmD_n^\circ$ and $r \in [1,n-3\log n)$ we have that
\begin{equation} \label{eq:form2.1a}
\bbP_\partial( \tau_{\rmB(x;r)} < \ol{\tau}_\partial) = \frac{2\pi}{\deg(\partial)} 
	\frac{1 + O\big(\rme^{-(r \wedge(n-2\log n-r))}\big)}{G_{\rmD_n}(x,x) - r - \gamma^*}\,, 
\end{equation} where $\gamma^*$ is the constant from~\eqref{e:3.2}, and in addition that, for any $y \in \rmB(x;n-3\log n)\setminus \rmB(x;r)$,
\begin{equation} \label{eq:form2.1b}
\bbP_y( \tau_\partial < \tau_{\rmB(x;r)} ) = \frac{\log \|y-x\|-r + O\big(\rme^{-(r \wedge (n-2 \log n - \log \|y-x\|))}\big)}{G_{\rmD_n}(x,x) - r - \gamma^*}\,.
\end{equation}
\end{lem}

\subsubsection{Local times and downcrossings}
\label{ss:locdownprelim}

In this subsection we give coarse estimates for the local time spent at any given vertex as well as for harmonic averages of the local time over outer boundaries of balls. We also include similar estimates for the number of downcrossings made by the walk between the outer and inner radii of a given annulus. To define the latter precisely, given $x \in \rmD_n$ and $l > k \geq 0$ with $\ol{\rmB(x;l)} \subseteq \rmD_n$, we introduce the sequence of stopping times $(T_i)_{i \geq 0}$ and $(T'_i)_{i \geq 1}$ inductively by setting $T_0 := 0$ and for $i \geq 1$ letting
\begin{equation}
	T'_i := \inf \{ t  \geq T_{i-1} :\: X_t \in \rmB(x;k) \}
 \ , \quad T_i := \inf \{ t  \geq T'_{i} :\: X_t \in \rmB(x;l)^\rmc \}\,.
\end{equation} We shall call any excursion of the walk $X$ from $\rmB(x;k)$ to $\rmB(x;l)^\rmc$ an $(x;k,l)$-\textit{excursion} and any excursion from $\rmB(x;l)^\rmc$ to~$\rmB(x;k)$ an $(x;k,l)$-\textit{downcrossing}, often calling these simply excursions or downcrossings whenever the choice of $x,k$ and $l$ is clear. Bearing in mind the stopping~times above and recalling that the walk $X$ starts from $\partial$,  $(x;k,l)$-excursions correspond to the paths $(X_t : t \in [T'_i,T_{i}])$ and $(x;k,l)$-downcrossings to the paths $(X_t : t \in [T_{i-1},T'_i])$ for $i \geq 1$. 
Then, the number of $(x;k,l)$-{\em downcrossings} made by the walk until the total running time is $t$ is 
\begin{equation}\label{eq:defntz}
\bfN_t(x;k,l) := \sup\, \{i \geq 1 :\: T'_i \leq t \} \,.
\end{equation} The number of $(x;k,l)$-downcrossings until the walk accumulates local time $t$ at the boundary~$\partial$ is then given by $N_t(x;k,l) := \bfN_{\bfL^{-1}_t(\partial)}(x;k,l)$. Note that $N_t(x;k,l)$ coincides with the number of $(x;k,l)$-excursions until such time, since each downcrossing accounted for in $N_t(x;k,l)$ must be followed by an $(x;k,l)$-excursion prior to the next downcrossing and/or return to the boundary. We also introduce a {\em normalized} version of $N_t(x;k,l)$ given by
\begin{equation}\label{eq:defntz2}
	\wh{N}_t(x;k,l) := (l-k) N_t(x;k,l) \,.
\end{equation} The normalized number of downcrossings of an annulus will be a quantity of interest for us since,  as we will see in Proposition~\ref{l:4.4b}, it is comparable to the harmonic average of the local time~field over its inner boundary (and it is often simpler to control than the latter).

For $i \geq 1$, the \textit{entry} and \textit{exit} points of the $i$-th excursion will be defined as $X_{T'_i}$ and $X_{T_i}$, respectively. 
We also let 
\begin{equation}
	\cF(x;k,l) := \sigma \big( X_t :\: t \in \cup_{i\geq 1} \, [T_{i-1}, T'_i] \big)
\end{equation}
 be the $\sigma$-algebra generated by $X$ when ``observed only during downcrossings'', and notice that both $N_t(x;k,l)$ and $\wh{N}_t(x;k,l)$ for any $t \geq 0$ as well as the entry and exit points of any $(x;k,l)$-excursion are all measurable with respect to $\cF(x;k,l)$. 

Finally, as we will mostly be interested in the particular case where the inner and outer radii are respectively of exponential scale $k+\tfrac{1}{2}k^\gamma$ and $k+k^\gamma$ for some fixed constant $\gamma \in (0,1/2-\eta_0)$, we shall henceforth abbreviate
\begin{equation}
\label{e:2.2}
N_t(x; k) := N_t\big(x; k+\tfrac12 k^\gamma, k+k^\gamma \big) 
\ , \ \ 
\wh{N}_t(x; k) := \wh{N}_t \big(x; k+\tfrac12 k^\gamma, k+k^\gamma \big)  
\end{equation}
and
\begin{equation}
	\cF(x;k) := \cF\big(x; k+\tfrac12 k^\gamma, k+k^\gamma \big)  \,,
\end{equation} and also write 
\begin{equation}\label{eq:defnk}
	\cN_k:=\Big\{ \sqrt{ \tfrac{1}{2}k^\gamma m} : m \in \N_0\Big\}
\end{equation} for the range of values of $\sqrt{\wh{N}_t(x;k)}$.
We remark that, although all the quantities defined above depend on $n$, we shall not exhibit~this in the notation as it will always be clear from the context. 

We now list all the preliminary estimates we shall need. We remind the reader that, unless stated explicitly, in the sequel the random walk $X$ is always assumed to start at the boundary.
We begin with a simple upper bound on the probability that a given vertex has low local time.

\begin{lem}\label{lem:nhprob} 
There exists a constant $C=C(\mathrm{D})>0$ such that, for all $x \in \rmD_n$, $t>0$ and $u \leq t$, we have
\begin{equation} \label{eq:form2}
	\bbP(L_t(x)=0) = \mathrm{e}^{-\frac{t}{G_{\rmD_n}(x,x)}} \leq \rme^{-\frac{t}{n+C}}
\end{equation} and
\begin{equation}\label{eq:form3}
	\bbP(L_t(x) \leq u) \leq \rme^{-\frac{(\sqrt{t}-\sqrt{u})^2}{G_{\rmD_n}(x,x)}}\leq \rme^{-\frac{t}{n+C}(1-2\sqrt{\frac{u}{t}})}\,.
\end{equation} 
\end{lem}

The next proposition shows that  $\ol{L_t}(x;k+1)$ and $\wh{N}_t(x;k)$ are comparable quantities. Recall that $\ol{L_t}(x;k+1)$ is the harmonic average of $L_t$ along $\partial \rmB(x;k+1)$, as defined in~\eqref{e:3.4}. 
\begin{prop}
	\label{l:4.4b} There exist $c,\delta \in (0,1)$ such that if $k$ is sufficiently large then, given any $x,n$ such that $\ol{\rmB(x;k+k^\gamma)} \subseteq \rmD_n$, for all $t>0$, $\hat{m} \geq \frac{1}{2}k^\gamma$ we have that 
	\begin{equation}
		\label{e:3.19b}
		\bbP \big(\ol{L_t}(x;k+1) > \wh{m}+z
		\,\big|\,\wh{N}_t(x;k)\leq \wh{m}\,;\,\cF(x;k) \big) 
		\leq \rme^{-c\frac{z^2}{k^{\gamma}\wh{m}}}
	\end{equation} for any $\rme^{-\delta k^\gamma}\wh{m}\leq z \leq \delta \wh{m}$ and 
	\begin{equation}
		\label{e:3.19c}
		\bbP \big(\ol{L_t}(x;k+1) < \wh{m} - z
		\,\big|\,\wh{N}_t(x;k)\geq \wh{m}\,;\,\cF(x;k) \big) 
		\leq \rme^{-c\frac{z^2}{k^{\gamma}\wh{m}}}\,.
	\end{equation} for any $z \geq \rme^{-\delta k^\gamma}\wh{m}$. 
	
\end{prop}
		
The next three propositions provide some coarse estimates for the number of downcrossings or the local time value at one or two vertices, given the number of downcrossings at an encapsulating scale. We remark that no barrier estimates are involved.
\begin{prop}\label{prop:Abulk} 
There exists $\delta > 0$ such that, if $k$ is sufficiently large, given any $n \geq l \geq k$ and $x,y \in \rmD_n$ such that $\rmB(x; k+k^\gamma) \subseteq \rmB(y;l)$ and $\ol{\rmB(y;l+l^\gamma)} \subseteq \rmD_n$, for any $t> 0$ and $u,v \geq 0$ such that $v + 1 \leq  u \leq \rme^{\delta k^\gamma}$ we have
\begin{equation}\label{eq:proofAbulk2}
\bbP\Big( \sqrt{\wh{N}_t(x;k)} \leq v \,\Big|\, \sqrt{\wh{N}_t(y;l)} \geq u\,;\cF(y;l)\Big) \leq 2 \exp\bigg\{ -\frac{(u-v)^2}{l+l^\gamma - (k+\frac{1}{2}k^\gamma)}\bigg\}.
\end{equation}
In addition, for each fixed $\eta \in (\gamma,1)$, if $n$ is sufficiently large (depending only on $\eta$ and $\gamma$) then, 
given any $k \in [n^{\eta},n-n^\eta]$ and $x \in \rmD_n$ such that $\ol{\rmB(x;k+k^\gamma)} \subseteq \rmD_n$,  for any $t>0$ and  $v \geq 0$ such that $v + 1 \leq \sqrt{t}  \leq \rme^{\frac{1}{8}n^\eta}$ we have
\begin{equation}\label{eq:downesta}
\bbP\Big( \sqrt{\wh{N}_t(x;k)} \leq v \Big) \leq 2 \exp\bigg\{ -\frac{(\sqrt{t}-v)^2}{n+C- (k+\frac{1}{2}k^\gamma)}\bigg\}\,,
\end{equation} for some constant $C=C(\rmD) > 0$.
\end{prop} 		
				
\begin{prop}\label{prop:Abulk2} 
For each $v>0$ there exists $C=C(v) > 0$ such that, if $l$ is sufficiently large,  
given any $n \geq l$ and $x,y \in \rmD_n$ such that $x \in \rmB(y;l)$ and $\ol{\rmB(y;l+l^\gamma)} \subseteq \rmD_n$,
for any $t > 0$ and $s \in [-l,\infty)$ we have
\begin{equation}\label{eq:proofAbulk3}
\bbP\Big( \sqrt{L_t(x)} \leq v \,\Big|\, \sqrt{\wh{N}_t(y;l)} \geq \sqrt{2}l + s\,;\cF(y;l)\Big) \leq \exp\bigg\{ - 2l -2\sqrt{2}s -\frac{s^2}{l}\Big(1-\frac{C}{l^{1-\gamma}}\Big) + Cl^\gamma \bigg\}
\end{equation} and
\begin{equation}\label{eq:proofAbulk3b}
\exp\bigg\{-2l -2\sqrt{2}s - \frac{s^2}{l}\Big(1+\frac{C}{l^{1-\gamma}}\Big) -Cl^\gamma\bigg\}\leq \bbP\Big( \sqrt{L_t(x)} \leq v \,\Big|\,\sqrt{\wh{N}_t(y;l)} \leq \sqrt{2}l+s\,; \cF(y;l)\Big).	
\end{equation} 
\end{prop}

\begin{prop}
\label{l:4.6} For any $v,\delta>0$ there exists $C=C(v,\delta) > 0$ such that, if $l$ is large enough, given any $n \geq l$ and $x,y,z \in \rmD_n$ with $x,z \in \rmB(y;l)$, $\log|x-z|\geq l -\delta l^\gamma$ and $\ol{\rmB(y;l+l^\gamma)} \subseteq \rmD_n$, for all $t > 0$ and $s \in [-l,\infty)$ we have 
\begin{equation}\label{eq:4.6bc}
	\bbP \Big(\sqrt{L_t(x)} \leq v ,\, \sqrt{L_t(z)} \leq v \,\Big|\, \sqrt{\wh{N}_t(y; l)} \geq \sqrt{2}l+s;\; 
\cF(y;l)\Big) \\ \leq  \rme^{ - 4l -4\sqrt{2}s -\frac{2s^2}{l}\big(1-\frac{C}{l^{1-\gamma}}\big) + Cl^\gamma}.
\end{equation}
\end{prop}


Finally, coarse estimates for the probability of the existence of a vertex with a low local time are given by the following proposition. 
\begin{prop}
\label{l:4.5}
For any $\gamma \in (0,1/2)$ and $u \geq 0$ there exists $C=C(\gamma,u) \in (0,\infty)$ such that, if $k$ is sufficiently large, given any $n \geq k$ and $x \in \rmD_n$ such that $\ol{\rmB(x;k+k^\gamma)} \subseteq \rmD_n$, for any $t > 0$ and $s \in [0, \infty)$ with $\sqrt{2}k+s \in \cN_k$ we have
	\begin{equation}
	\label{e:4.11}
\bbP \bigg( \min_{y \in \rmB(x;k)} L_t(y) \leq u \,\Big|\, \sqrt{\wh{N}_t(x; k)} = \sqrt{2} k + s ;\; 
\cF(x;k)\bigg) 
\in \Big(\rme^{-2\sqrt{2} s -\frac{s^2}{k}- Ck^{\gamma}}, \rme^{-2\sqrt{2} s + C k^{\gamma}} \Big)\,.\end{equation}
\end{prop}

\subsection{Discrete Gaussian free field preliminaries}
Our next set of preliminary statements concern the discrete Gaussian free field.

\subsubsection{Basic properties and estimates}
Given $\rmU \subseteq\bbZ^2$ finite and non-empty, we will write $h_\rmU = (h_{\rmU}(x) :\: x \in \bbZ^2)$ to denote the discrete Gaussian free field on $\rmU$ with zero boundary condition on $\bbZ^2 \setminus \rmU$ (DGFF on $\rmU$, for short), defined as the centered Gaussian field on $\bbZ^2$ with covariances given by 
\begin{equation}\label{eq:covgff}
\bbE h_\rmU(x) h_\rmU(y) = \tfrac12 G_{\rmU}(x,y) 1_{\rmU\times \rmU}(x,y)\,.
\end{equation} We remark that, with our particular choice of transition rates for the underlying random walk, the field given by~\eqref{eq:covgff} is $\sqrt{\pi}/2$ times the DGFF on $\rmU$ as commonly defined in the probabilistic literature.  Let us now recall some know facts about the DGFF.

If $\rmU,\rmV \subseteq \bbZ^2$ are finite non-empty sets such that $\rmd(\rmU,\rmV) \geq 2$ then, for any DGFF $h_{\rmU \cup \rmV}$ on $\rmU \cup \rmV$, the fields $h_{\rmU \cup \rmV}1_{\rmU}$ and $h_{\rmU \cup \rmV}1_{\rmV}$ are independent and distributed as DGFFs on $\rmU$ and $\rmV$, respectively. Furthermore, if whenever $\rmV \subseteq \rmU$ we define the conditional expectation field $\varphi_{\rmU, \rmV}$, also known as the {\em binding field}, via
\begin{equation}
\label{103.6}
\varphi_{\rmU, \rmV}(x) := \bbE \big(h_\rmU(x) \,\big|\, h_{\rmU}(\rmV^\rmc)\big)
\ , \quad x \in \bbZ^2 \,,
\end{equation}
then the Gibbsianity and Gaussianity of the law of $h_\rmU$ imply that
\begin{equation}
\label{e:3.19}
\varphi_{\rmU, \rmV}|_{\ol{\rmV}} = \ol{\big(h_{\rmU} \big)}_\rmV
\ , \quad
\varphi_{\rmU,\rmV} \perp  h_\rmU - \varphi_{\rmU,\rmV} \laweq h_\rmV \,,
\end{equation} where $\perp$ denotes independence.
In particular, we have that $\varphi_{U,V}$ is a centered Gaussian field. It follows from these two properties above that, if $\rmV_1,\dots,\rmV_k \subseteq \rmU$ are finite and non-empty sets with $\rmd(\rmV_i,\rmV_j) \geq 2$ for all $i\neq j$, we can write any DGFF $h_\rmU$ on $\rmU$ as
\begin{equation}\label{eq:gibbs-markov}
h_{\rmU} = \varphi_{\rmU,\rmV} + \sum_{i=1}^k h_{\rmV_i}\,,	
\end{equation} where $\varphi_{\rmU,\rmV}$ is given by \eqref{103.6} for $\rmV:=\bigcup_{i=1}^k \rmV_i$, $h_{\rmV_i}$ is  a DGFF on $\rmV_i$ for each $i$ and all terms on the right-hand side of \eqref{eq:gibbs-markov} are independent of each other. This decomposition of $h_\rmU$ is known as the \textit{Gibbs-Markov decomposition}.

We now collect some general estimates on the DGFF. The next two lemmas provide control on the fluctuations of its binding field.
\begin{lem}
\label{l:3.25}
For any pair of nice sets $\cU, \cV \subseteq\bbR^2$ there exist some constants $C = C(\cU, \cV) < \infty$ and $c = c(\cU, \cV) > 0$ such that, for all $n,n' \geq 0$ and all translates $\wt{\cV}_{n'}$ of $\cV_{n'}$ such that $\wt{\cV}_{n'} \subseteq\cU_n$, the following holds: 
\begin{enumerate}
\item Given $q > 0$, for all $x \in \wt{\cV}_{n'}^{n'-q}$,
\begin{equation}
	\tfrac{1}{2}(n-n'-q)-C  \leq \bbE \big(\varphi_{\cU_n, \wt{\cV}_{n'}}(x)\big)^2 \leq \tfrac12 (n-n'+q) + C  \,,
\end{equation} 
\item For all $u \geq 0$, $l > 0$ and $x \in \wt{\cV}_{n'}$ such that $\rmB(x,l) \subseteq \wt{\cV}_{n'}^{n'-q}$,
\begin{equation}
	\label{e:3.24.2}
	\bbP \bigg( \max_{y \in \rmB(x;l)} \rme^{(n'-l)/2} \big|
	\varphi_{\cU_n, \wt{\cV}_{n'}}(x) - \varphi_{\cU_n, \wt{\cV}_{n'}}(y) \big| > u \bigg) \leq C \rme^{-c u^2} \,.
\end{equation}
\end{enumerate}
\end{lem}


\begin{lem}
\label{l:103.1}
For any nice set $\cU \subseteq\bbR^2$ there exist some $C = C(\cU), c = c(\cU) \in (0,\infty)$ such~that the following statement holds: if
$n,k,l \geq 0$ and $\rmY,\rmZ \subseteq \cU_n$ are such that $\{\rmB(z;k+1) :\: z \in \rmZ\}$ are disjoint subsets of $\cU_n$ and for each $y \in \rmY$ there exists $z(y) \in \rmZ$ such that $\rmB(y;l) \subseteq \rmB(z(y);k+\tfrac{1}{2})$, with $z(y) \neq z(y')$ whenever $y \neq y$, then for all  $u \geq 0$ we have
\begin{equation}
\bbP \Big(\max_{y \in \rmY} \big|\varphi_{\cU_n, \rmV} (y)\big| > \sqrt{2} (n-k) + u \Big) \leq C \rme^{-c u} \,,
\end{equation} where $\rmV$ above is given by $\rmV := \bigcup_{z \in \rmZ} \rmB(z; k+1)$. 
Moreover, there exists also $C' = C'(\cU) \in (0,\infty)$ such that, whenever $l \in (0,k)$, for all $u \geq 0$ we have
\begin{equation}
\label{e:3.30}
\bbP \Big(\max_{y \in \rmY} \max_{x \in \rmB(y;l)} 
\rme^{(k-l)/2} \big|\varphi_{\cU_n, \rmV} (x) - \varphi_{\cU_n, \rmV} (y) \big| >  C' \sqrt{n-k} + u \Big) \leq C \rme^{-c u} \,.
\end{equation}
\end{lem}

Next we control the harmonic average of $h_n^2$ on the outer boundary of a ball of scale $k$ around a vertex conditioned to have a low value under $h_{\cU_n}$.
\begin{lem}
	\label{l:3.6}
	For any nice set $\cU \subseteq \bbR^2$ there exist some $C,c,r_0 \in (0,\infty)$ depending only on $\cU$ such that, if $n$ is large enough then, given any $k \in [r_0,n-r_n]$, $x \in \bbZ^2$ and $y \in (\bbX_k \cup \{x\}) \cap \mathcal{U}_n^{k+2}$ such that $x \in \rmB(y;k)$, for all $t >0$ and $v \in (-n^{1/2}, n^{1/2})$ we have 
	\begin{equation}
		\bbP \Big(\sqrt{\ol{(h_{\cU_n}-\tfrac{n-k}{n}v)^2}(y;k+1)} >  t \,\Big|\, h_{\cU_n}(x) = v \Big) 
		\leq C\rme^{-c \frac{t^2}{k}} \,.
	\end{equation}
\end{lem}

\subsubsection{The Isomorphism Theorem}
\label{sss:3.2.2}
A relation between the local time of the random walk on $\wh{\rmD}_n$ with time measured in terms of the local time at $\partial$ and the DGFF on $\rmD_n$ is given by the so-called second Ray-Knight Theorem and is due to~\cite{Gang}. 
\begin{thm}[Second Ray-Knight Theorem]
\label{t:103.1}Given $t \geq 0$ and $n \geq 0$, there exists a coupling of $L_{t} = (L_{t}(x) : x \in \rmD_n)$ with two copies of the DGFF on $\rmD_n$, $h_n = (h_n(x) : x \in \rmD_n)$
and $h'_n = (h'_n(x) : x \in \rmD_n)$, such that $L_{t}$ and $h_n$ are independent of each other and almost-surely,
\begin{equation}
\label{e:3.1}
L_{t}(x) + h^2_n(x) = 
\big(h'_n(x) + \sqrt{t})^2 
\quad , \ x \in \rmD_n \,.
\end{equation}
\end{thm}
Henceforth, whenever we use~\eqref{e:3.1} we implicitly assume that the processes $L_t$, $h_n$ and $h'_n$ are all defined in our probability space and that they are coupled and defined as in Theorem~\ref{t:103.1}. Furthermore, the coupling above will mostly be used with $t = t_n^{A}$ given by~\eqref{e:2.4.1}, in which case we shall write the right-hand side of~\eqref{e:3.1} as $f^2_n(x)$, where
\begin{equation}
f_n(x) := h'_n(x) + m_n 
\quad , \qquad m_n := \sqrt{t_n^{A}} \,.
\end{equation}

For $u \geq 0$, the sub-level set of the field $f^2_n$ at height $u$ is defined by
\begin{equation}
\label{e:103.3}
\rmG_n(u) := \big \{ x \in \rmD^\circ_n :\: f^2_n(x) \leq u \big \} \,.
\end{equation}
Observe that, under the coupling in Theorem~\ref{t:103.1}, we always have
\begin{equation}
\label{e:3.3}
\rmG_n(u) \subseteq \rmW_n(u)  \,,
\end{equation}
since the second term on the left-hand side of~\eqref{e:3.1} is always positive. 
Accordingly, we shall sometimes say that a given vertex $x \in \rmW_n(u)$ {\em survives the isomorphism} to mean that it is also in $\rmG_n(u)$.

\subsubsection{Extreme values}
Let us now recall some well-known results in the theory of extreme values for the DGFF as well as some variants of these, which we will need later in the manuscript.
\begin{prop}
\label{p:3.1}
Given a nice set $\cU \subseteq\R^2$, there exist constants $C, C',N > 0$ depending only on $\cU$ such that, for all $n \geq N$ and $u \geq 0$, 
\begin{equation}
\bbP \Big(\Big|\max_{x \in \cU_n} h_{\cU_n}(x) - m_n \Big| > u \big) \leq C \rme^{-C' u^{2/3}} \,.
\end{equation}
\end{prop}

\begin{prop}
\label{p:3.2}
The sequence $(|\rmG_n(u)| :\: n \geq 1)$ is tight for any $u > 0$ and, moreover,
\begin{equation}
\lim_{u \to \infty} \limsup_{n \to \infty} \bbP \big(\rmG_n(u) = \emptyset\big) = 0 \,.
\end{equation}
\end{prop}

We shall also need the fact that if $x \in \rmG_n(u)$, then the average value of $f^2_n$ on the boundary of a ball of log-scale $k$ around $x$ is typically within $(\sqrt{2} k + k^{1/2-\eta}, \sqrt{2} k + k^{1/2+\eta})$ for any $\eta > 0$.
For an explicit formulation, for $x \in \rmD_n$, $k \in [0,r_n]$ , $u \geq 0$ and $\eta \in (0,1/2)$, we let
\begin{multline}
\rmQ_n^{k, \eta}(u) := \Big\{x \in \rmG_n(u) :\: 
\exists y \in \bbX_k \cap \rmD_n^{k+2} \text{ s.t. } x \in \rmB(y;k) ,\, \\
\sqrt{\ol{f^2_n}(y;k+1)} - \sqrt{2} k \notin \big(k^{1/2-\eta} ,\, k^{1/2+\eta} \big) \Big\}\,,
\end{multline}
which are extended to all $K \subseteq [0,r_n]$ via
\begin{equation}
\rmQ_n^{K, \eta}(u) := \bigcup_{k \in K} \rmQ_n^{k,\eta}(u)\,. 
\end{equation} We then have:

\begin{prop}
\label{p:3.3}
Given $\eta \in (0,1/2)$, for all $u \geq 0$ we have
\begin{equation}
\label{e:3.9a}
\lim_{r \to \infty} \limsup_{n \to \infty} 
\bbP \big(\rmQ_n^{[r, r_n], \eta}(u) \neq \emptyset \big) = 0 \,.
\end{equation}
\end{prop}

For the analogous case of atypically down-repelled harmonic averages, we have:
\begin{prop}
\label{p:3.4}
Given $0 < \eta < \eta' < 1/2$, there exist $C, c > 0$ such that for all $0 \leq k \leq n$,
\begin{equation}
	\label{e:3.9}
	\bbP \Big(\exists y \in \bbX_k \cap \rmD_n^{n-2\log n -1} :\:  \sqrt{\ol{f^2_n}(y;k+1)} \leq \sqrt{2} k - k^{1/2-\eta}  \Big)
	\leq C \rme^{-k^{1/2-\eta'}} \,.
\end{equation}
\end{prop}

Finally, we shall also need the following point estimate:
\begin{prop}
\label{p:3.17}
For all $\eta > 0$ small enough and any $u \geq 0$,
\begin{equation}
	\begin{split}
	\lim_{n \to \infty} \sup_{x \in \rmD_n^\circ} \rme^{2n} \,  \bbP \Big(f_n^2(x) \leq u,\,\,
	& \exists \ell \in [r_n, n-r_n] :\, \sqrt{\ol{f^2_n}(x;l+1)}  <  \sqrt{2} l + n^\eta \,, \\
	& \forall k \in [n^{\eta}, n-n^{\eta}] :\, \sqrt{\ol{f^2_n}(x;k+1)}  > \sqrt{2} k -n^\eta  \Big) = 0 \,.
	 \end{split}
\end{equation}
\end{prop}

\section{Two key technical lemmas}
\label{s:4}
In this section we include two technical tools that will allow us to upper bound the number of clusters which satisfy some prescribed conditions. While indeed technical, these lemmas are key to the proof of the sharp clustering result in Section~\ref{s:5}. 

For what follows, introduce a partition of the scaled lattice $\bbX_k$ into disjoint subsets that can be used to cover $\bbZ^2$. Concretely, for any $l \geq k$ we fix a priori a partition $\big(\bbX_{k,l}(j) :\: j=1, \dots, \big(\lfloor\frac{ 2\lfloor \rme^l\rfloor}{\lfloor \rme^k\rfloor}\rfloor +1\big)^2\big)$ of $\bbX_k$ with the property that, for each $j$ and any $x,y \in \bbX_{k,l}(j)$ with $x\neq y$, we have $\rmB(x;l) \cap \rmB(y;l) = \emptyset$. By definition, $\cup_{x \in \bbX_{k,l}(j)} \rmB(x;l)$ are unions of disjoint balls of log-scale $l$ which, taken together for all $j$, cover $\bbZ^2$.
In particular, $\bbZ^2$ can be partitioned into $9$ disjoint subsets: $\bigcup_{x \in \bbX_{k,k}(j)} \rmB(x;k)$ with $j=1, \dots, 9$. For convenience, throughout the following we shall abbreviate $J_{k,l}:=\big(\lfloor\frac{ 2\lfloor \rme^l\rfloor}{\lfloor \rme^k\rfloor}\rfloor +1\big)^2$. In addition, we notice that, by definition of $r$-bulk, if $x \in \rmD_n^{r+1}$ for some $r \geq 1$ then we have $\ol{\rmB(x;r)} \subseteq \rmD_n$. This fact will be used repeatedly in the statement of the results below.

\subsection{Thinning Lemma}

\begin{lem}[Thinning Lemma]
\label{l:3.2g} Given $\eta,r > 0$, for each $k \geq r$ let $\cA_k \subseteq \bbR$ be Borel measurable and define  
\begin{equation}
\label{e:4.2a}
	\cB_k := 	\Big\{\sqrt{a^2 + b^2} :\: a \in \cA_k,\, b \in [0, k^{1/2+\eta}] \Big\} \,.
\end{equation}
In addition, given $n \geq r$, let $\rmY_n \subseteq \rmD_n$ be a (possibly) random subset which is $L_{t^A_n}(\rmD_n)$-measurable and, for each $u \geq 0$, define the sets
\begin{equation}
\rmZ_{n}(u) :=  \Big\{ x \in \rmW_n(u) \cap \rmY_n : \exists k \in [r,r_n] , y \in \bbX_k \cap \rmD_n^{k+2} \text{ s.t. }  x \in \rmB(y ;k) \,, 
	\sqrt{\ol{L_{t_n^A}}(y; k+1)} \in \cA_k \Big \}\,,
\end{equation}
\begin{equation}
\begin{split}
\rmX_{n}(u+1) :=  \Big\{ x \in \rmG_n(u+1) \cap \rmY_n : \exists k \in [r,r_n],\, y \in \bbX_k \cap \rmD_n^{k+2} \text{ s.t. }& x \in \rmB(y ;k)\,, \\&\sqrt{\ol{f_n^2}(y; k+1)} \in \cB_k \Big \}
\end{split}
\end{equation}
and
\begin{equation}
\bbZ_n(u) := \big\{ z \in \bbX_{\lfloor n-r_n\rfloor} :\: \rmZ_n(u) \cap \rmB(z; \lfloor n-r_n\rfloor ) \neq \emptyset \big\} \,.
\end{equation}
Then, there exist constants $C, c,r_0 \in (0,\infty)$ which depend only on $u$ and $\eta$ such that, if $r > r_0$, for all $s > 0$ and $n$ sufficiently large (depending only on $r$, $\eta$, $\eta_0$ and $\rmD$) we have
\begin{equation}
	\label{e:4.1b}
	\bbP \Big( |\bbZ_{n}(u)| >  s \sqrt{n} \Big) 
	\leq 
	C \big(1+s^{-1}\big) \bbP( \big|\rmX_{n}(u+1)\big| >  cs)\,.
\end{equation}
\end{lem}

As an immediate application of the Thinning Lemma above and a demonstration of its use, we have the following upper bound of order $\sqrt{n}$ on the total number of clusters:
\begin{lem}
	\label{l:3.2}
	For each $u \geq 0$,
	\begin{equation}
		\lim_{\delta \to 0}
		\limsup_{n \to \infty} \bbP \Big(\big|\bbW_n(u)\big| > \delta^{-1} \sqrt{n} \Big) = 0\,.
	\end{equation}
\end{lem}
\begin{proof}
	Taking $\rmY_n := \rmD_n$ and $\cA_k := \bbR$ for each $k\geq r$ in Lemma~\ref{l:3.2g} above, if $r > r_0$ then for all $n$ large enough the inequality in \eqref{e:4.1b} becomes
	\begin{equation}
		\label{e:4.1a}
		\bbP \Big( \big|\bbW_n(u) \big| >  s \sqrt{n} \Big) 
		\leq 
		C \big(1+s^{-1}\big) \bbP( |\rmG_n(u+1) \big| >  cs) \,,
	\end{equation}
	for some $C,c > 0$ and all $s > 0$. The result now follows from the tightness of $\rmG_n(u+1)$ as stated in Proposition~\ref{p:3.2}.
\end{proof}

\begin{proof}[Proof of Lemma~\ref{l:3.2g}]
Since $J_{k,k+2} \leq (4\rme^2+1)^2$ for any $k \geq 0$, by the union bound it is sufficient to show~\eqref{e:4.1b} but with $\bbZ_n(u)$ being replaced by $\bbZ_n(u;j) := \bbZ_{n}(u) \cap \bbX_{\lfloor n-r_n\rfloor, \lfloor n-r_n\rfloor+2}(j)$ for each $j=1,\dots,J_{\lfloor n-r_n\rfloor, \lfloor n-r_n\rfloor+2}$ on the left-hand side thereof. Fixing any such $j$, for each $z \in \bbZ_n(u;j)$ we may choose  $x(z) \in \rmW_n(u) \cap \rmY_n \cap \rmB(z; \lfloor n-r_n\rfloor)$, $k(z) \in [r,r_n]$ and $y(z) \in \bbX_{k(z)} \cap \rmD_n^{k(z)+2}$ such that 
$x(z) \in \rmB(y(z); k(z))$ and 
\begin{equation}
	\sqrt{\ol{L_{t_n^A}}\big(y(z); k(z)+1\big)} \in \cA_k \,.
\end{equation}
Then, by~\eqref{e:3.1} and~\eqref{e:4.2a}, we have that $|\rmX_{n}(u+1)|$ is at least the size of the set
\begin{equation}
\wt{\rmX}_n(u+1;j):=\Big\{z \in \bbZ_n(u;j) :\: \big|h_n\big(x(z)\big)\big| \leq 1
	\,,\,\, 
	\sqrt{\ol{h^2_n}\big(y(z);k(z)+1\big)} \leq k(z)^{1/2+\eta} \Big\}\,.
\end{equation} Our goal is to show that, conditional on $L_{t^A_n}$, $|\wt{X}_n(u+1;j)|$ (essentially) stochastically dominates a Binomial random variable with $|\bbZ_n(u;j)|$ number of trials and success probability $\frac{c}{\sqrt{n}}$ for some~$c$, from where \eqref{e:4.1b} will then follow by a standard argument. 

To this end, let us define $\rmV := \bigcup_{z \in \bbZ_n(u;j)} \rmB(z; \lfloor n-r_n\rfloor+1)$. Observe that if $n$ is large enough (depending only on the value of $\eta_0$) then we have that $\rmV \subseteq \rmD_n$ since $x(z) \in \rmD_n^\circ$ by definition. Furthermore, since $\rmd(\rmB(z;\lfloor n-r_n\rfloor+1),\rmB(z';\lfloor n-r_n\rfloor +1))\geq 2$ for any pair of different vertices $z,z' \in \bbX_{n-r_n,n-r_n+2}(j)$, the Gibbs-Markov decomposition implies that we can write $h_n$ as 
\begin{equation}
	h_n = \varphi_{\rmD_n,\rmV} + \sum_{z \in \bbZ_n(u;j)} h_{\rmB(z;\lfloor n-r_n\rfloor+1)} \,,
\end{equation}
where all the terms on the right-hand side are independent of each other. 

Now, on the one hand, notice that for all $n$ sufficiently large (depending only on $\eta_0$) we have $\partial \rmB(y(z);k(z)+1) \subseteq \rmB(x(z);\lfloor r_n \rfloor+2) \subseteq \rmB(z; \lfloor n-r_n\rfloor + \tfrac{1}{2})$ for all $z \in \bbZ_n(u;j)$. Hence, by using Lemma~\ref{l:103.1} with $\cU=\rmD$, $\rmZ= \bbZ_n(u;j)$, $\rmY := \{x(z) :\: z \in \bbZ_n(u;j)\}$, $k=\lfloor n-r_n\rfloor$ and $l=\lfloor r_n\rfloor +2$ we see that, conditional on $L_{t^A_n}$, the probability of the event
\begin{equation}
	\label{e:104.5a}
	\Big\{
	\max_{z \in \bbZ_n(u;j)} \big|\varphi_{\rmD_n, \rmV}\big(x(z)\big)\big| \leq r_n \log r_n  \ ,\ \  
	\max_{z \in \bbZ_n(u;j)} \max_{w \in \partial \rmB(y(z); k(z)+1)}
	\big|\varphi_{\rmD_n, \rmV}\big(w \big) - 
	\varphi_{\rmD_n, \rmV}\big(x(z)\big) 
	\big| \leq \frac{1}{2}  \Big\} 
\end{equation}
is at least $\frac{1}{2}$ for all $n$ sufficiently large (depending only on $\eta_0$ and $\rmD$).

On the other hand, since for each $z \in \bbZ_n(u;j)$ we have $\partial \rmB(y(z);k(z)+1) \subseteq \rmB(z;\lfloor n-r_n \rfloor +1)$ and the balls $\big(\rmB(z;\lfloor n-r_n\rfloor+1) : z \in \bbZ_n(u;j)\big)$ are disjoint, for each $z \in \bbZ_n(u;j)$ we have that
\begin{equation}
	\ol{h_n^2}\big(y(z);k(z)+1\big) = \ol{(\varphi_{\rmD_n, \rmV}+h_{\rmB(z;\lfloor n-r_n\rfloor+1)})^2}(y(z);k(z)+1)\,.
\end{equation} 
 It then follows from the independence of $L_{t_n^A}$ and $h_n$ in~\eqref{e:3.1} that, conditional on the field $L_{t_n^A}$, on the intersection of the event in~\eqref{e:104.5a} and 
\begin{equation}
	\Big\{\big|\bbZ_n(u;j)\big| > s \sqrt{n} \Big\} \,,
\end{equation}
$|\wt{\rmX}_n(u+1;j)|$ (and therefore also $|\rmX(u+1)|$) stochastically dominates a Binomial random variable with $\lceil s \sqrt{n}\rceil$ number of trials and success probability at least
\begin{equation}
\label{e:100.1a}
\min_{x,v,k,z,y}
\bbP \Big(\big|h_{\rmB(z;\lfloor n-r_n\rfloor+1)}(x) - v\big| \leq 1
\,,\,\, 
\sqrt{\ol{(h_{\rmB(z;\lfloor n-r_n\rfloor+1)}-v)^2}(y;k+1)} \leq \frac{1}{2}\sqrt{k^{1+2\eta}-1}\Big) \,,
\end{equation}
where the minimum is over all $|v| \leq r_n \log r_n$, $z \in \bbX_{\lfloor n-r_n\rfloor, \lfloor n-r_n\rfloor+2}(j)$, $x \in \rmB(z;\lfloor n-r_n\rfloor+1)$, $k \in [r, r_n]$ and $y \in \bbX_k \cap \rmD_n^{k+2}$ such that $x \in  \rmB(y;k) \cap  \rmB(z;\lfloor n-r_n\rfloor)$, and we used the inequality $(a+b)^2 \leq 4(a^2+b^2)$ to lower bound the success probability.

Now, by \eqref{e:3.19.2} we have that, for any $x$ and $z$ as above, $h_{\rmB(z;\lfloor n-r_n\rfloor+1)}(x)$ is a centered Gaussian with variance at $O(1)$ distance from $(n-r_n)/2 $. By using this in conjunction with Lemma~\ref{l:3.6}, we get that~\eqref{e:100.1a} is at least $\frac{c}{\sqrt{n}}$ for some $c > 0$ if $r$ is sufficiently large. Since by Chernov's bound the probability that a Binomial random variable with $\lceil s \sqrt{n}\rceil$ trials and success probability $c/\sqrt{n}$ is larger than $c s/2$ is at least $c_0(s \wedge 1) $ for all $n \geq 1$ and some fixed $c_0 > 0$ not depending on $s$, by conditioning on $L_{t^A_n}$ and using the independence of $\varphi_{\rmD_n,\rmV}$ and $(h_{\rmB(z;\lfloor n-r_n \rfloor +1)} : z \in \bbZ_n(u;j))$ together with the stochastic domination discussed above, we obtain 
\begin{equation}
	\bbP \Big(|\bbZ_n(u;j)| > s \sqrt{n} \Big)
	\leq \frac{2}{c_0} (s \wedge 1)^{-1} \bbP \Big(\big|\rmX_{n}(u+1)\big| > c s / 2 \Big) \,,
\end{equation}
which readily implies~\eqref{e:4.1b}. 
\end{proof}

\subsection{Resampling Lemma}

Recall that, given $k \geq 0$, we write $\cN_k=\big\{ \sqrt{\tfrac{1}{2}k^\gamma m} : m \in \N_0\}$ for the range of values of $\sqrt{\wh{N}_t(x;k)}$.

\begin{lem}
\label{l:4.1n} Given $k \geq 1$, let $n \geq k$ be sufficiently large so as to have that $\bbX_k \cap \rmD_n^{k+k^\gamma+1} \neq \emptyset$. In addition, let $M_n^k \subseteq \cN_k$ and $(\cA_{n,y,m}^k : y \in \bbX_k\,, m \in \cN_k)$ be non-empty and Borel measurable, where $\cA_{n,y,m}^k \subseteq \bbR_+^{\ol{\rmB(y;k+1)}}$ for each $y \in \bbX_k$ and $m \in \cN_k$. Finally, given $u \geq 0$, define
\begin{equation}
\label{e:4.15.a}
\begin{split}
	\bbZ_n^k(u) := \Big \{ z \in \bbW_n(u) :\: 
	 \exists y \in \bbX_k \cap \rmD_n^{k+k^\gamma+1}\,&,\,m \in M_n^k
	\text{ s.t. } 
	\rmW_n(u) \cap \rmB(z; \lfloor n-r_n\rfloor) \subseteq \rmB(y;k) \,,
	\\
	& L_{t^A_n}\big(\ol{\rmB(y;k+1)}\big) \in \cA_{n,y,m}^k	\text{ and }
			\sqrt{\wh{N}_{t_n^A}(y; k)}  = m	\Big \} \,.
\end{split}
\end{equation}
Then, for all $v \geq 0$, 
\begin{equation}
\label{e:4.2n}
\bbP \Big(\big|\bbZ_n^k(u)\big| > v \big|\bbW_n(u)\big|\Big)
\leq C  v^{-1} \rme^{4k^\gamma} p_n^{(k)}(u) \,,
\end{equation}
with $C \in (0,\infty)$ some universal constant (not depending on $u$, $n$ nor $k$) and
\begin{equation}
	\label{e:4.3n}
p_n^{(k)}(u) := \sup_{y,m}
	\left\|\frac{\bbP \Big(L_{t^A_n}\big(\ol{\rmB(y;k+1)}\big) \in \cA_{n,y,m}^k,\displaystyle{\min_{x \in \rmB(y;k)}}  L_{t_n^A}(x) \leq u \Big|\, 
 \sqrt{\wh{N}_{t_n^A}(y; k)} =  m\,;
	\cF\big(y;k\big) \Big)}{\bbP\Big( \displaystyle{\min_{x \in \rmB(y;k)}}L_{t_n^A}(x) \leq u \Big|\, 
 \sqrt{\wh{N}_{t_n^A}(y; k)} =  m\,;
	\cF\big(y;k\big) \Big)}\right\|_\infty
\end{equation} where the supremum is over all $y \in \bbX_k \cap \rmD_n^{k+k^\gamma+1}$ and $m \in M_n^k$.
\end{lem}

Our use of Lemma~\ref{l:4.1n} will be in conjunction with Lemma~\ref{l:3.2}, as captured in:
\begin{lem}[Resampling Lemma]
\label{l:4.2n} Given $k \geq 1$, let $n \geq k$, $M^k_n$ and $(\cA^k_{n,y,m} :\: y \in \bbX_k\,, m \in \cN_k)$ be as in Lemma~\ref{l:4.1n}. In addition, suppose that, for some $\eta \in (0,1/2-\gamma)$, if $k$ is large enough then
\begin{equation}
\label{e:4.4n}
\sup_{n,y,m}
\rme^{2\sqrt{2}(s \vee 0)+\frac{(s\vee 0)^2}{k}}
\Big\|\bbP \Big(L_{t^A_n}\big(\ol{\rmB(y;k+1)}\big) \in \cA^k_{n, y, m} \Big|\\
 \sqrt{\wh{N}_{t_n^A}(y; k)} =  m\,; \cF\big(y;k\big)\Big) \Big\|_\infty  
\leq \rme^{-k^{1/2-\eta}}\,,
\end{equation}
where $s:=m-\sqrt{2}k$ and the supremum is over all $n,y$ such that $y \in \bbX_k \cap \rmD_n^{k+k^\gamma+1}$ and $m \in M^k_n$. Then, given~$u \geq 0$, for all $\eta' \in (\eta,1/2)$,
\begin{equation}
\label{e:4.5n}
\lim_{r \to \infty} \limsup_{n \to \infty}
\bbP \Big( \exists k \in [r, r_n] :\: \big| \bbZ_n^k(u) \big| > 
\rme^{-k^{1/2-\eta'}} \sqrt{n} \Big) = 0 \,,
\end{equation}
where $\bbZ_n^k(u)$ is defined as in Lemma~\ref{l:4.1n}.

\end{lem}

\begin{proof}[Proof of Lemma~\ref{l:4.1n}]
Given $u \geq 0$, fix $k \geq 1$ and $n \geq k$ such that $\bbX_k \cap \rmD_n^{k+k^\gamma+1}\neq \emptyset$ holds and, for each $i=1,\dots,9$ and $j=1,\dots,J_{k,\lceil k+k^\gamma \rceil}$, define the sets
\begin{equation}
	\label{e:6.3}
	\begin{split}
		\bbY_n^k(u;i,j) := 
		\Big \{  z &\in \bbX_{\lfloor n-r_n\rfloor,\lfloor n-r_n\rfloor}(i)  :\, 
		\exists! \,y(z) \in \bbX_{k, \lceil k+k^\gamma\rceil}(j) \cap \rmD_n^{k+k^\gamma+1}\, \text{ s.t. } \\
		& \rmB\big(y(z);k\big) \cap \rmB\big(z;\lfloor n-r_n\rfloor \big) \cap \rmW_n(u) \neq \emptyset \text{ and } m(z) := \sqrt{\wh{N}_{t_n^A}(y(z);k)} \in M^k_n \Big\} 
	\end{split}
\end{equation}
and
\begin{equation}
	\bbZ_n^k(u;i,j) := \Big \{ z \in \bbY_n^k(u;i,j) :\: L_{t_n^A}\big(\ol{\rmB(y(z);k+1)}\big) \in \cA^k_{n,y(z), m(z)} \Big \}\,.
\end{equation}
Then, since for a fixed $j$ the balls $(\rmB(y;k) :y \in \bbX_{k,\lceil k+k^\gamma \rceil}(j))$ are all disjoint, we have that
\begin{equation}
	\bbZ_n^k (u)
	\subseteq \bigcup_{i=1}^9 \bigcup_{j=1}^{J_{k,\lceil k+k^\gamma\rceil}} \bbZ_n^k(u;i,j)\,.
\end{equation}
Furthermore, since in fact $(\rmB(y;k+k^\gamma) :y \in \bbX_{k,\lceil k+k^\gamma \rceil}(j))$ are disjoint for each $j$, it is not difficult to see that, conditional on $|\bbY^k_n(u;i,j)|$, the law of $|\bbZ^k_n(u;i,j)|$ is stochastically dominated by that of a Binomial random variable with $|\bbY_n^k(u;i,j)|$ number of trials and success probability given by $p_{n}^{(k)}(u)$ from~\eqref{e:4.3n}. Moreover, since $|\bbY^k_n(u;i,j)| \leq |\bbW_n(u)|$, it follows from Markov's inequality that, for each $i=1,\dots,9$ and $j=1,\dots, J_{k,\lceil k+k^\gamma \rceil}$,
\begin{equation}
	\bbP \Big(\big|\bbZ^k_n(u;i,j)\big| > \frac{v}{9J_{k,\lceil k+k^\gamma \rceil}} \big|\bbW_n(u)\big|\Big)
	\leq  v^{-1} 9J_{k,\lceil k+k^\gamma \rceil}p_{n}^{(k)}(u)\,.
\end{equation} Since $J_{k,\lceil k+k^\gamma \rceil}\leq \rme^{2k^\gamma +8}$, summing over all $i$ and $j$ and using the union bound then gives~\eqref{e:4.2n}.
\end{proof}

\begin{proof}[Proof of Lemma~\ref{l:4.2n}]
Given $u \geq 0$,  for each $k \geq 1$ and $n \geq k$ such that $\bbX_k \cap \rmD_n^{k+k^\gamma+1}\neq \emptyset$ holds, for any $y \in \bbX_k \cap \rmD_n^{k+k^\gamma+1}$ the conditional probability in the definition of $p_n^{(k)}$ in~\eqref{e:4.3n} can be bounded from above by
\begin{equation}\label{eq:bq}
\frac{
\bbP \Big(L_{t^A_n} \big(\ol{\rmB(y;k+1)}\big) \in \cA^k_{n, y, m} \Big|
 \sqrt{\wh{N}_{t_n^A}(y	; k)} =  m\,; \cF\big(y;k\big) \Big)}
{\bbP \Big( \displaystyle{\min_{x \in \rmB(y;k)}  L_{t_n^A}(x) \leq u}  \Big|\, 
	\sqrt{\wh{N}_{t_n^A}(y; k)} =  m\,; 
	\cF\big(y;k\big) \Big)} \,.
\end{equation}
Writing $m = \sqrt{2} k + s$, by Proposition~\ref{l:4.5} and the fact that the conditional probability in \eqref{e:4.11} is decreasing in the number of downcrossings, we find that the denominator in \eqref{eq:bq} is at least $\rme^{-2\sqrt{2}(s\vee 0) - \frac{(s\vee 0)^2}{k} - Ck^\gamma}$ so that, by~\eqref{e:4.4n}, we have $p_n^{(k)}(u) \leq \rme^{-k^{1/2-\eta}+Ck^\gamma}$ for all $k$ large enough and $n \geq k$ such that $\bbX_k \cap \rmD_n^{k+k^\gamma+1}\neq \emptyset$. By applying Lemma~\ref{l:4.1n} with $v := \delta \rme^{-k^{1/2-\eta'}}$ for $\delta > 0$, we then obtain
\begin{equation}\label{e:critbound}
	\bbP \Big( \big| \bbZ_n^k (u)\big| > \delta \rme^{-k^{1/2-\eta'}} \big|\bbW_n(u)\big| \Big)
	\leq C \delta^{-1} \rme^{-\frac12 k^{1/2-\eta}}\,,
\end{equation}
for all $k$ large enough and $n \geq k$ such that $\bbX_k \cap \rmD_n^{k+k^\gamma+1}\neq \emptyset$, provided that $\eta \in (0,1/2-\gamma)$.

At the same time, by the union bound we have that, for any $\delta > 0$, 
\begin{multline}
	\bbP \Big( \exists k \in [r, r_n] :\: \big| \bbZ_n^k(u) \big| > 
	\rme^{-k^{1/2-\eta'}} \sqrt{n} \Big) \\
\leq 
	\bbP \Big( \big|\bbW_n(u)\big| > \delta^{-1} \sqrt{n} \Big) +
\sum_{k \in [r,r_n]}
	\bbP \Big( \big| \bbZ_n^k(u) \big| > \delta \rme^{-k^{1/2-\eta'}} \big|\bbW_n(u)\big| \Big) \,.
\end{multline} Upon noticing that, if $n$ is large enough (depending only on $\eta_0$, $\gamma$ and $\rmD$) then $\bbX_k \cap \rmD_n^{k+k^\gamma+1} \neq \emptyset$ for all $k \in [1,r_n]$, 
by first taking $n \to \infty$, then $r \to \infty$ and finally $\delta \to 0$, Lemma~\ref{l:3.2} and~\eqref{e:critbound} combined yield~\eqref{e:4.5n}.
\end{proof}

As an example of the use of the Resampling Lemma, we obtain the following consequence, which shows that the harmonic average of local time and normalized number of downcrossings are comparable. To be more concrete, given $\eta \in (0,1/2)$, $m \geq 0$ and $u \geq 0$, for $k\geq 1$and $n\geq k$ sufficiently large so as to have $\bbX_k \cap \rmD_n^{k+k^\gamma+1} \neq \emptyset$ define
\begin{equation}
	\label{e:5.1ab}
	\begin{split}
		\bbT^{k,\eta,+}_{n,m}(u) := \Big \{  z \in  \bbW_n(u) :\:
		\exists & y \in \bbX_k \cap \rmD_n^{k+k^\gamma+1}
		\text{ s.t. } 
	\rmW_n(u) \cap \rmB(z; \lfloor n-r_n\rfloor ) \subseteq \rmB(y;k)\,, \\
		& \sqrt{\ol{L_{t_n^\rmA}}(y; k+1)} < m-k^{1/2-\eta} \text{ and }
		\sqrt{\wh{N}_{t_n^\rmA}(y; k)} \geq m \Big\}
	\end{split}
\end{equation} together with 
\begin{equation}
	\label{e:5.1ab2}
	\begin{split}
		\bbT^{k,\eta,-}_{n,m}(u) := \Big \{  z \in  \bbW_n(u) :\:
		\exists y \in \bbX_k & \cap \rmD_n^{k+k^\gamma+1}
		\ \text{ s.t. } 
	\rmW_n(u) \cap \rmB(z; \lfloor n-r_n\rfloor ) \subseteq \rmB(y;k)\,, \\
		& \sqrt{\ol{L_{t_n^\rmA}}(y; k+1)} > m+k^{1/2-\eta} \text{ and }\sqrt{\wh{N}_{t_n^\rmA}(y; k)} \leq m \Big\}
	\end{split}
\end{equation} and $\bbT^{k,\eta}_{n,m}(u):=\bbT^{k,\eta,+}_{n,m}(u) \cup\bbT^{k,\eta,-}_{n,m}(u)$. Then:

\begin{lem}
	\label{l:6.7n} If $\eta \in (0,1/2)$ is sufficiently small then, for all $\eta' \in (\eta,1/2)$, $u \geq 0$ and $(m_k)_{k \geq 1}$ such that $m_k - \sqrt{2} k \in [-k^{1/2+\eta}, k^{1/2+\eta}]$ for each $k \geq 1$, 
	\begin{equation}
		\label{e:3.13c}
		\lim_{r \to \infty} \limsup_{n \to \infty}
		\bbP \Big( \exists k \in [r, r_n] :\:  \big|\bbT^{k,\eta}_{n,m_k}(u)\big| > \rme^{-k^{1/2-\eta'}} \sqrt{n} \Big) = 0 \,.
	\end{equation}
\end{lem}

\begin{proof}[Proof of Lemma~\ref{l:6.7n}]
By the union bound, we need only prove the statement for $\bbT^{k,\eta,\pm}_{n,m_k}(u)$ separately in place of $\bbT^{k,\eta}_{n,m_k}(u)$. Starting with $\bbT^{k,+}_{n,m_k}(u)$, for $k \geq 1$, $n \geq k$ large enough so that $\bbX_k \cap \rmD_n^{k+k^\gamma+1} \neq \emptyset$, $y \in \bbX_k$ and $m \in \cN_k$, let
\begin{equation}
	\cA_{n,y,m}^{k,+}:= \Big \{ w \in \bbR_+^{\ol{\rmB(y;k+1)}} :\: \sqrt{\ol{w}(y;k+1)}  < m_k - k^{1/2-\eta} \Big\}\,.
\end{equation}
Then $\bbT^{k,+}_{n,m_k}(u) \subseteq \bbZ_n^k(u)$, where $\bbZ^k_n(u)$ is as in Lemma~\ref{l:4.2n} with $\cA_{n,y,m}^{k,+}$ in place of $\cA_{n,y,m}^{k}$ and $M_n^k := \cN_k \cap [m_k,\infty)$. 

To check~\eqref{e:4.4n}, for $n,y$ such that $y \in \bbX_k \cap \rmD_n^{k+k^\gamma+1}$ and $m \in M^k_n$, we write the probability therein as
\begin{equation}
\label{e:4.30a}
\bbP \Big(\sqrt{\ol{L_{t_n^\rmA}}(y;k+1)} < m_k-k^{1/2-\eta}
\,\bigg|\, \sqrt{\wh{N}_{t_n^\rmA}(x;k)} = m\,; 
\cF\big(x; k\big) \bigg)\,.
\end{equation}
Thanks to Proposition~\ref{l:4.4b}, if we write $m = m_k + s$ with $s \geq 0$ then, by the restriction on $m_k$, for all $k$ large enough (depending only on $\eta$) we can bound the last probability from above by 
\begin{equation}
\label{e:4.31}
\begin{split}
	\exp \bigg( -c \frac{\big((m_k+s)^2 - (m_k - k^{1/2-\eta})^2\big)^2}{k^\gamma (m_k+s)^2} \bigg)
	& \leq 
	\exp \bigg( -c \frac{s^4+m_k^2(k^{1-2\eta} + s^2)}{k^\gamma (m_k+s)^2} \bigg) \\
	& \leq 
	\exp \bigg(-c\bigg(k^{1-2\eta-\gamma}+\frac{s^2}{k^{\gamma}}\bigg)\bigg) \,.
\end{split}
\end{equation}
At the same time, using again the restriction on $m_k$, the exponential in~\eqref{e:4.4n} is at most
\begin{equation}	
\label{e:4.31.5}
\exp \bigg(2\sqrt{2}((m_k -\sqrt{2}k + s)\vee 0)+\frac{(m_k -\sqrt{2}k + s)^2}{k}\bigg)
\leq \exp \bigg(C \bigg( k^{1/2+\eta} + s + \frac{s^2}{k} \bigg) \bigg) \,.
\end{equation}
It is not difficult to see that, if $\eta > 0$ is taken small enough so as to have $1-2\eta-\gamma > 1/2+\eta$,  then the product of the two right-hand sides in~\eqref{e:4.31} and~\eqref{e:4.31.5} is at most $\rme^{-k^{1/2-\eta}}$ for all $s \geq 0$, the moment $k$ is large enough. An application of Lemma~\ref{l:4.2n} then gives~\eqref{e:3.13c} for $\bbT^{k,+}_{n,m_k}(u)$.

The proof for $\bbT^{k,-}_{n,m_k}(u)$ is analogous and amounts to verifying that
\begin{equation}
\label{e:4.30b}
\bbP \Big(\sqrt{\ol{L_{t_n^\rmA}}(y;k+1)} \geq m_k+k^{1/2-\eta}
\,\bigg|\, \sqrt{\wh{N}_{t_n^\rmA}(x;k)} = m_k-s\,; 
\cF\big(x; k\big) \bigg)\,,
\end{equation}
times
\begin{equation}	
\label{e:4.31.5b}
\exp \bigg(2\sqrt{2}((m_k -\sqrt{2}k - s) \vee 0)+\frac{((m_k -\sqrt{2}k - s) \vee 0)^2}{k}\bigg) \,
\end{equation}
is at most $\rme^{-k^{1/2-\eta}}$ for all $n$, $y$ and $s \geq 0$ as before, provided that $k$ is sufficiently large.

Since \eqref{e:4.30b} is decreasing in $s$, taking $s = 0$ yields an upper bound for~\eqref{e:4.30b} for all $s \geq 0$. Then, thanks again to Proposition~\ref{l:4.4b} and the restriction on $m_k$, the probability there is at most 
\begin{equation}
	\exp \bigg( -c \frac{\big((m_k + k^{1/2-\eta})^2 - m_k^2 \big)^2}{k^\gamma m_k^2} \bigg)
	\leq \rme^{-ck^{1-2\eta-\gamma}} \,.
\end{equation}
At the same time, display~\eqref{e:4.31.5b} is at most $\rme^{C k^{1/2+\eta}}$. For all $\eta$ small enough as before, this gives the desired upper bound on the product as soon as $k$ is sufficiently large, so that an application of Lemma~\ref{l:4.2n} now finishes the proof.
\end{proof}

\section{Sharp clustering of leaves with low local time}
\label{s:5}
In this section we prove a sharp clustering result for the set of vertices with $O(1)$ local time. To give a precise statement, we recall the definition of $r_n$ in~\eqref{e:2.3}, $\bbW_n^K(u)$, $\rmW_n^K(u)$ in~\eqref{e:2.10}, ~\eqref{e:2.11} and $\rmG_n(u)$ in~\eqref{e:103.3}. We will show:
\begin{thm}
\label{t:3.1}
For any $u \geq 0$,
\begin{equation}
\label{e:3.14}
\lim_{r \to \infty}
\limsup_{n \to \infty} \bbP \Big( \rmW_n^{[r, r_n]}(u) \cap  \rmG_n(u)  \neq \emptyset \Big) = 0 \,.
\end{equation}
In particular, for all such $u \geq 0$ and $\delta > 0$,
\begin{equation}\label{e:3.14b}
\lim_{r \to \infty}
\limsup_{n \to \infty} \bbP \Big( \big| \bbW_n^{[r, r_n]}(u) \big| > \delta \sqrt{n}  \Big) = 0 \,.
\end{equation}
\end{thm}
The theorem shows that $O(1)$-local time vertices which belong to clusters whose log-scale is in $[r,r_n]$ do not survive the isomporhism (in the sense defined in Subsection~\ref{sss:3.2.2}). Moreover (in fact, as a consequence), the number of such clusters is $o(\sqrt{n})$ with high probability. Together with the coarse clustering result in the following section which handles the scales in $[r_n, n-r_n]$, this will show the required clustering picture, which enable the proofs the main theorems of phase A and B.

In order to prove Theorem~\ref{t:3.1}, we will cover the set $\rmW_n^{[r, r_n]}(u)$ by three subsets and show that a statement analogous to \eqref{e:3.14} holds for each of them. The original claim will then follow by the union bound. To give a precise statement, given $\eta \in (0,1/2)$ and $u \geq 0$, for $n \geq 1$~and~$k \in [1, r_n]$ we define our first subset as 
\begin{equation}
\label{e:5.1}
\begin{split}
\bbU^{k,\eta}_n(u) := \Big \{ z \in \bbW^k_n(u) :\: 
\exists y \in \bbX_k \cap \rmD_n^{k+2}  
\ \text{ s.t. } 
& \rmW_n(u) \cap \rmB(z; \lfloor n-r_n\rfloor) \subseteq \rmB(y;k)\,,\,
\\
& \sqrt{\ol{L_{t_n^A}}(y; k+1)} > \sqrt{2}k + k^{1/2-\eta}
\Big \} \,.
\end{split}
\end{equation}
This is the set of all centers of clusters which can be contained in a ball of log-scale $k$ such that the harmonic average of the local time field  $L_{t^A_n}$ over the boundary of this ball is unusually high. We shall refer to such clusters as {\em up-repelled} clusters.

For our second subset we define,
\begin{equation}
\label{e:5.2}
\bbB^{k,\eta}_n(u) := \Big \{ z \in \bbW^k_n(u) :\: 
\big|\rmW_n(u) \cap \rmB(z; \lfloor n-r_n\rfloor) \big| > \rme^{k^{1/2-\eta}} \Big\} \,.
\end{equation}
This is the set of all centers of clusters of log-scale $k$ whose size is unusually large. We shall refer to such clusters as {\em big} clusters. For the union of all vertices which are included in such clusters, we let
\begin{equation}
\rmU^{k,\eta}_n(u) := \bigcup_{z \in \bbU_n^{k,\eta}(u)} \Big(\rmW_n(u) \cap \rmB(x; \lfloor n-r_n \rfloor ) \Big) 
\quad, \quad
\rmB^{k,\eta}_n(u) := \bigcup_{z \in \bbB_n^{k,\eta}(u)} \Big(\rmW_n(u) \cap \rmB(x; \lfloor n-r_n\rfloor ) \Big) \,.
\end{equation}
We extend all the above definitions to subsets $K \subseteq [1,r_n]$ in place of $k$ via union, e.g., by setting
$\rmU^{K,\eta}_{n}(u) = \bigcup_{k \in K} \rmU^{k, \eta}_n(u)$.

To prove Theorem~\ref{t:3.1}, we will show that neither the sets above nor their complement survive the isomorphism for all $k$ large enough, namely:
\begin{prop}
\label{p:3.1a}
If $\eta \in (0,1/2)$ is sufficiently small then, for all $u \geq 0$,
\begin{equation}
\label{e:3.14a}
\lim_{r \to \infty}
\limsup_{n \to \infty} \bbP \Big( \big( \rmU_n^{[r, r_n], \eta}(u) \cup \rmB_n^{[r, r_n], \eta}(u)\big) \cap \rmG_n(u)  \neq \emptyset \Big) = 0 \,.
\end{equation}
\end{prop}
\begin{prop}
\label{p:3.1c}
Fix $\eta \in (0,1/2)$. Then, for any $u \geq 0$,
\begin{equation}
\label{e:3.14c}
\lim_{r \to \infty}
\limsup_{n \to \infty} \bbP \Big(\Big( \rmW_n^{[r, r_n]}(u) \setminus \big(\rmU_n^{[r, r_n], \eta}(u) \cup  \rmB_n^{[r, r_n], \eta}(u)\big) \Big) \cap \rmG_n(u)  \neq \emptyset \Big) = 0 \,.
\end{equation}
\end{prop}

It is now a very short step to:
\begin{proof}[Proof of Theorem~\ref{t:3.1}]
The first part of the theorem follows easily from Propositions~\ref{p:3.1a} and~\ref{p:3.1c} by the union bound.
The second part then follows from the first by applying Lemma~\ref{l:3.2g} with $\rmY_n := \rmW_n^{[r, r_n]}(u)$ and
$\cA_k \equiv \bbR$ for all $k \geq r$.
\end{proof}

\subsection{Proof of Proposition~\ref{p:3.1a}}
To prove Proposition~\ref{p:3.1a} it is sufficient to show the following two lemmas:
\begin{lem}
\label{l:3.9.5a}
If $\eta \in (0,1/2)$ is sufficiently small then, for all $\eta' \in (\eta,1/2)$ and $u \geq 0$,
\begin{equation}
\lim_{r \to \infty} \limsup_{n \to \infty}
\bbP \Big( \exists k \in [r, r_n] :\: \big| \bbU^{k, \eta}_n(u)  \big| 
> \rme^{-k^{1/2-\eta'}} \sqrt{n} \Big) = 0 \,.
\end{equation}
\end{lem}
\begin{lem}
\label{l:3.9.5}
If $\eta \in (0,1/2)$ is sufficiently small then, for all $\eta' \in (\eta,1/2)$ and $u \geq 0$,
\begin{equation}
\label{e:3.47}
\lim_{r \to \infty} \limsup_{n \to \infty}
\bbP \Big( \exists k \in [r, r_n] :\: \big| \bbB^{k, \eta}_n(u) \big| 
> \rme^{-k^{1/2-\eta'}} \sqrt{n} \Big) = 0 \,.
\end{equation}
\end{lem}

Indeed, let us give the proof of Proposition~\ref{p:3.1a} assuming the two lemmas above.
\begin{proof}[Proof of Proposition~\ref{p:3.1a}]
Observe that if $x \in \rmU_n^{[r, r_n], \eta}(u) \cup \rmB_n^{[r, r_n],\eta}(u)$ then there exists $k \in [r,r_n]$, $z \in \bbU_n^{k, \eta}(u) \cup \bbB_n^{k,\eta}(u)$ and $y(z) \in \bbX_k$ such that $x \in \rmW_n(u) \cap \rmB(z;\lfloor n - r_n \rfloor) \subseteq \rmB(y(z);k)$. Moreover, if $x$ belongs also to $\rmG_n(u)$ then we must also have $|h_n(x)| \leq \sqrt{u}$. Hence, by conditioning on $L_{t^A_n}$, we may upper bound the probability in \eqref{e:3.14a} by 
\begin{equation}\label{eq:5.10a}
	\bbE\left(\sum_{k \in [r,r_n]} \sum_{z \in \bbU_n^{k, \eta}(u) \cup \bbB_n^{k,\eta}(u)} \bbP\Big( \min_{x \in \rmB(y(z);k)\cap \rmD_n} |h_n(x)| \leq \sqrt{u}\Big)\right)\,.
\end{equation} On the other hand, since $\rmB(y(z);k+1) \subseteq \rmD_n$ for all $n$ is sufficiently large (depending only on~$\eta_0$) because $x \in \rmD_n^\circ$, by decomposing $h_n$ as the independent sum of $\varphi_{\rmD_n, \rmB(y(z);k+1)}$ and $h_{\rmB(y(z);k+1)}$ we may bound
\begin{equation}\label{e:5.10}
\begin{split}
\bbP\Big( \min_{x \in \rmB(y(z);k)\cap \rmD_n} |h_n(x)| \leq & \sqrt{u}\Big)
\leq
\bbP \Big(\big|\varphi_{\rmD_n, \rmB(y(z);k+1)}(y)\big| < 2k^2 \Big) \\
&+ \bbP \Big(\sup_{x \in \rmB(y(z);k)} \big|\varphi_{\rmD_n, \rmB(y(z);k+1)}(x) - \varphi_{\rmD_n, \rmB(y(z);k+1)}(y(z))\big| > k^2 \Big)\\
&+ \bbP \Big(\sup_{x \in \rmB(y(z);k)} \big|h_{\rmB(y(z);k+1)}(x)| > k^2-\sqrt{u} \Big) \,.
\end{split}
\end{equation}
If $r$ is large enough, then the first part of Lemma~\ref{l:3.25} and a simple bound of the Gaussian density give $C k^2/\sqrt{n-k}$ as an upper bound for the first term on the right-hand side above whenever $n$ is sufficiently large (depending only on $\eta_0$). The 
second part of the same lemma gives $C\rme^{-ck^2}$ as a bound for the second term above. For the third term, we may use Proposition~\ref{p:3.1} to obtain the upper bound $C_u \rme^{-c k^{4/3}}$. Altogether, the right hand side of~\eqref{e:5.10} is at most $C_u (\frac{k^2}{\sqrt{n-k}}+\rme^{-ck^{4/3}})$. 
Combining this with \eqref{eq:5.10a}, for $\eta' \in (\eta,1/2)$ we can upper bound the probability in \eqref{e:3.14a} by 
\begin{equation}
\bbP \Big(
\exists k \in [r, r_n] :\: \big| \bbU^{k, \eta}_n(u) \cup \bbB_n^{[r, r_n],\eta}(u) \big| > 
\rme^{-k^{1/2-\eta'}} \sqrt{n} \Big) 
+ C_u \sum_{k \in [r,r_n]} \rme^{-k^{1/2-\eta'}} \Big(k^2 \sqrt{\frac{n}{n-k}} +\rme^{-ck^{4/3}}\Big)\,, 
\end{equation} for all $r$ and $n$ large enough. By Lemmas~\ref{l:3.9.5a} and Lemma~\ref{l:3.9.5}, we obtain that 
the first term tends to zero in the limits stated in \eqref{e:3.14a}. For the second term, we may bound the sum by $\wt{C}_u \rme^{-\tfrac{1}{2}r^{1/2-\eta'}}$ if~$r$ and $n$ are large enough (depending only on $\eta'$ and $\eta_0$). The result now immediately follows.
\end{proof}

\subsection{Up-repelled clusters} 
Thanks to Lemma~\ref{l:6.7n}, to show Lemma~\ref{l:3.9.5a} it is sufficient to prove an analogous statement with the harmonic average of the local time replaced by the number of downcrossings. To this end, given $\eta \in (0,1/2)$ and $u\geq 0$, for $n \geq 1$ and $k \in [1, r_n]$ we define
\begin{equation}
\label{e:5.1a}
\begin{split}
\bbV^{k,\eta}_n(u) := \Big \{ z \in \bbW^k_n(u) :\: 
\exists y \in \bbX_k \cap \rmD_n^{k+k^\gamma+1}
\ \text{ s.t. } 
& \rmW_n(u) \cap \rmB(z; \lfloor n-r_n\rfloor) \subseteq \rmB(y;k)  ,\,
\\
& \sqrt{\wh{N}_{t_n^A}(y; k)} > \sqrt{2}k + k^{1/2-\eta}
\Big \} 
\end{split}
\end{equation}
and extend this definition to subsets $K \subseteq[1,r_n]$ in place of $k$ via union in the usual way. Then, we have:
\begin{lem}
\label{l:3.9.5d}
If $\eta > 0$ is sufficiently small then, for all $\eta' \in (\eta,1/2)$ and $u \geq 0$,
\begin{equation}
\label{e:3.47a}
\lim_{r \to \infty} \limsup_{n \to \infty}
\bbP \Big( \exists k \in [r, r_n] :\: \big| \bbV^{k, \eta}_n(u)  \big| 
> \rme^{-k^{1/2-\eta'}} \sqrt{n} \Big) = 0 \,.
\end{equation}
\end{lem}
\begin{proof}
	The proof is an application of Lemma~\ref{l:4.2n}. For $k \geq 1$, $n \geq k$ and $\cN_k$ as defined in \eqref{eq:defnk},  let $M_n^k := \mathcal{N}_k \cap (\sqrt{2} k + k^{1/2-\eta}, \infty)$  
	and, for $y \in \bbX_k$ and $m \in \cN_k$, define also
	\begin{equation}
		\cA_{n,y,m}^k:= \Big \{ w \in \bbR_+^{\rmB(y;k+1)} :\: \exists x,x' \in \rmB(y;k) \text{ s.t. }
	\log \|x-x'\| \geq k-k^\gamma \,,\,w(x) \leq u \,,\, w(x') \leq u \Big\}	 \,.
	\end{equation}
	Then, since $\bbV_n^{k,\eta}(u) \subseteq \bbW_n^k(u)$, a straightforward computation shows that in fact $\bbV_n^{k,\eta}(u) \subseteq \bbZ_n^k(u)$ for all $k$ large enough (depending only on $\gamma$), with $\bbZ_n^k(u)$ as in Lemma~\ref{l:4.1n}. On the other hand, by the union bound, for any $k \geq 1$, $y \in \bbX_k \cap \rmD_n^{k+k^\gamma+1}$ and $m \in M_n^k$ we have
	\begin{multline}
		\Big\| \bbP \Big(L_{t^A_n} \big(\rmB(y;k+1)\big) \in \cA^k_{n, y, m} \Big|
		\sqrt{\wh{N}_{t_n^A}(x; k)} = m,\, \cF\big(y;k\big) \Big) \Big\|_\infty  \\
		\leq
		\sum_{x,x' } \bbP \Big(L_{t_n^A}(x) \leq u ,\, L_{t_n^A}(x') \leq u \,\Big|\, 
		\sqrt{\wh{N}_{t_n^A}(y; k)} = m ;\; 
		\cF\big(y;k\big)\Big)\,,
	\end{multline}
	where the sum is over all $x,x' \in \rmB(y;k)$ such that $\log |x-x'| \geq k-k^\gamma$.
	Thanks to Proposition~\ref{l:4.6}, for all such $y$ and $m$ we can bound the sum from above by $\rme^{-4\sqrt{2}s -\frac{s^2}{k} + Ck^\gamma}$ for all $k$ large enough, where $s:= m - \sqrt{2} k > k^{1/2-\eta}$. From this we obtain that~\eqref{e:4.4n} holds for all $k$ large enough if $\eta$ is taken sufficiently small, so that the result now follows from Lemma~\ref{l:4.2n}.
\end{proof}

\begin{proof}[Proof of Lemma~\ref{l:3.9.5a}] Fix $\eta > 0$ and $\eta' \in (\eta,1/2)$ and choose any  $\wt{\eta},\wt{\eta}' \in (\eta,\eta')$ such that $\wt{\eta} < \wt{\eta}'$. If $\eta$ is small enough so that~\eqref{e:3.47a} holds with $(\wt{\eta},\wt{\eta}')$ in place of $(\eta,\eta')$, then the result now~follows by the union bound and Lemma~\ref{l:6.7n} for $(\wt{\eta},\wt{\eta}')$ in place of $(\eta,\eta')$ and $m_k:=\sqrt{2}k+k^{1/2-\wt{\eta}}$. 
\end{proof}

\subsection{Big clusters}
Next we prove that the number of clusters of log-scale $k$ which are big decays with $k$. To do this, we first show that the number of downcrossings at log-scale $k$ in such clusters is unusually~low. 
To this end, given $\eta \in (0,1/2)$ and $u \geq 0$, for $n \geq 1$ and $k \in [1,r_n]$ we define in analog to~\eqref{e:5.1a} the set
\begin{equation}
\label{e:5.1b}
\begin{split}
	\bbE^{k,\eta}_n(u) := \Big \{ z \in \bbW_n^k(u) :\: 
	\exists y \in \bbX_k \cap \rmD_n^{k+k^\gamma+1}
	\ \text{ s.t. } 
	& \rmW_n(u) \cap \rmB(z; \lfloor n-r_n\rfloor) \subseteq \rmB(y;k) ,\,
	\\
	& \sqrt{\wh{N}_{t_n^A}(y; k)} \leq \sqrt{2}k - k^{1/2-\eta}
	\Big \} 
\end{split}
\end{equation}
and extend this definition to $K \subseteq[1,r_n]$ in place of $k$ via union in the usual way.
Then, similarly to Lemma~\ref{l:3.9.5d}, we have:
\begin{lem} \label{l:5.5} If $\eta > 0$ is sufficiently small then, for all $\eta' \in (\eta,1/2)$ and $u \geq 0$,
\begin{equation}
\label{e:3.48}
\lim_{r \to \infty} \limsup_{n \to \infty}
\bbP \Big( \exists k \in [r, r_n] :\: 
\big| \bbB^{k,\eta}_n(u)  \setminus \bbE^{k,\eta'}_n(u) \big| 
> \rme^{-k^{1/2-\eta'}} \sqrt{n}  \Big) = 0 \,.
\end{equation}
\end{lem}
\begin{proof}
We appeal again to the Resampling Lemma. For $k \geq 1$, $n \geq k$ and $\cN_k$ as defined in~\eqref{eq:defnk},  let $M_n^k := \mathcal{N}_k \cap  (\sqrt{2} k - k^{1/2-\eta'}, \sqrt{2} k + k^{1/2-\eta}]$  
	and, for $y \in \bbX_k$ and $m \in \cN_k$, define also
\begin{equation}
	\cA_{n,y,m}^k:= \Big \{ w \in \bbR_+^{\rmB(y;k+1)} :\:  \Big|\big\{x \in \rmB(y;k) :\: w(x) \leq u \big\}\Big| > \rme^{k^{1/2-\eta}} \Big\}\,.
\end{equation}
Then, if $n$ is sufficiently large (depending only on $\eta_0$), for any $k \in [1,r_n]$ the set $\bbB^{k,\eta}_n(u)  \setminus \bbE^{k,\eta'}_n(u)$ is contained in $\bbZ_n^k(u) \cup \bbV_n^{k,\eta}(u)$, where $\bbZ_n^k(u)$ is the set from Lemma~\ref{l:4.1n} and $\bbV_n^{k,\eta}(u)$ is the one from~\eqref{e:5.1a}. By Lemma~\ref{e:4.2n}, Lemma~\ref{l:3.9.5d} and the union bound, to obtain the result it will suffice to check~\eqref{e:4.4n} for these choices of $M_n^k$ and $\cA^k_{n,y,m}$. To this end, we write the probability therein as
\begin{equation}
	\label{e:4.19n}
	\bbP\Big( \big| \rmB(y;k) \cap \rmW_n(u) \big| > \rme^{k^{1/2-\eta}}\,  \Big|\, 
	\sqrt{\wh{N}_t(y; k)} =  m \,; 
	\cF\big(y;k\big) \Big)\,.
\end{equation}
Now, by Proposition~\ref{prop:Abulk2}, if $k$ is sufficiently large then for any $y \in \bbX_k \cap \rmD_n^{k+k^\gamma+1}$ and $m \in M_n^k$ the mean size of $\rmB(y;k) \cap \rmW_n(u)$ under the conditioning above 
is at most
\begin{equation}
	|\rmB(y;k)| \rme^{-2k-2\sqrt{2}s + C k^\gamma} \leq \rme^{-2\sqrt{2}s + \wt{C}k^\gamma} \,, 
\end{equation}
where $s := m - \sqrt{2} k \in (- k^{1/2-\eta'},k^{1/2-\eta}]$. Using the Markov Inequality and the restriction on $s$ we can then bound~\eqref{e:4.19n} from above by
$\rme^{2\sqrt{2}(k^{1/2-\eta'}-s\vee 0) - k^{1/2-\eta} + Ck^\gamma}$ for all $k$ large enough, which in turn yields~\eqref{e:4.4n} (for any $\eta'' \in (\eta,\eta')$ in place of $\eta$) if $\eta$ is taken small enough, so that the result now follows from Lemma~\ref{l:4.2n}.
\end{proof}

Next, we wish to show that the size of the set $\bbE^{k,\eta}_n(u)$ must also be an exponentially small multiple of $\sqrt{n}$. Thanks to Lemma~\ref{l:6.7n}, this can be shown first with the number of downcrossings replaced by the harmonic average of the local time. Accordingly, in analog to~\eqref{e:5.1b} and for the same range of parameters, we define
\begin{equation}
	\label{e:5.1c}
	\begin{split}
		\bbJ^{k,\eta}_n(u) := \Big \{ z \in \bbW_n^k(u) :\: 
		\exists y \in \bbX_k \cap \rmD_n^{k+2}
		\ \text{ s.t. } 
		& \rmW_n(u) \cap \rmB(z; \lfloor n-r_n \rfloor) \subseteq \rmB(y;k) ,\,
		\\
		& \sqrt{\ol{L_{t_n^A}}(y; k+1)} < \sqrt{2}k - k^{1/2-\eta}
		\Big \}\,. 
	\end{split}
\end{equation}
Then, we have:
\begin{lem}
\label{l:5.9} If $0 < \eta < \eta' < 1/2$ then, for all $u \geq 0$,
\begin{equation}
	\label{e:3.48b}
	\lim_{r \to \infty} \limsup_{n \to \infty}
	\bbP \Big( \exists k \in [r_, r_n] :\: 
	\big| \bbJ^{k,\eta}_n(u) \big|
	> \rme^{-k^{1/2-\eta'}} \sqrt{n}  \Big) = 0 \,.
\end{equation}
\end{lem}
\begin{proof}
	The proof is an application of the Thinning Lemma. Fix $r \geq 1$, $k_0 \geq r$ and, for $n,k\geq r$, define $\rmY_n \equiv \rmD_n$ and set
	$\cA_k:= [0, \sqrt{2}k - k^{1/2-\eta})$ if $k=k_0$ or $\cA_k=\emptyset$ otherwise. If $\cB_k$ is defined as in~\eqref{e:4.2a} for the same $\eta$ as $\cA_{k_0}$ then, given $\eta'' \in (\eta, \eta')$, if $r$ is large enough and $\eta$ small enough, we have that
	$\cB_{k} \subseteq[0, \sqrt{2}k - k^{1/2-\eta''})$ if $k=k_0$ and $\cB_k=\emptyset$ otherwise. In particular, by applying Lemma~\ref{l:3.2g} with $s:= \rme^{-k_0^{1/2-\eta'}}$, for $k_0 \in [r,r_n]$ we obtain that
	\begin{equation}
	\begin{split}
		\bbP \Big( \big| \bbJ^{k_0,\eta}_n(u) \big| &> \rme^{-k_0^{1/2-\eta'}} \sqrt{n}  \Big) 
		\leq C\rme^{k_0^{1/2-\eta'}} 
		\bbP (|X_n(u+1)| > 0)\\
		&\leq C\rme^{k_0^{1/2-\eta'}} \bbP\Big(\exists y \in \bbX_{k_0} \cap \rmD_n^{n-2\log n -1} :\:  \sqrt{\ol{f^2_n}(y;k_0+1)} \leq \sqrt{2} k_0 - k_0^{1/2-\eta''}  \Big) \,,
	\end{split}
	\end{equation} if $n$ is sufficiently large. Therefore, by fixing some $\eta''' \in (\eta'', \eta')$ and using Proposition~\ref{p:3.4} with $(\eta'',\eta''')$ in place of $(\eta,\eta')$ together with the union bound, the result now follows.
\end{proof}

We are now ready to prove Lemma~\ref{l:3.9.5}.

\begin{proof}[Proof of Lemma~\ref{l:3.9.5}] Given $\eta' \in (\eta,1/2)$, fix some $\eta_1,\eta_2,\eta_3$ such $\eta < \eta_1 < \eta_2 < \eta_3 < \eta'$ and observe that, if $r$ is sufficiently large, then the event in the statement of Lemma~\ref{l:3.9.5} is contained in the union of the events in the statements of Lemma~\ref{l:5.5} with $(\eta,\eta_1)$ in place of~$(\eta,\eta')$, Lemma~\ref{l:6.7n} with $(\eta_2,\eta_3)$ in place of $(\eta,\eta')$ and finally Lemma~\ref{l:5.9} with $(\eta_2,\eta_3)$ in place of $(\eta,\eta')$. The result then follows from the union bound and the three aforementioned lemmas, upon taking $\eta,\eta_1,\eta_2$ suitably small.
\end{proof}

\subsection{Proof of Proposition~\ref{p:3.1c}}

We now conclude the proof of Theorem~\ref{t:3.1} by giving the proof of Proposition~\ref{p:3.1c}.

\begin{proof}[Proof of Proposition~\ref{p:3.1c}]
	Thanks to Proposition~\ref{p:3.3} and the union bound, it suffices to show the probability
	\begin{equation}
		\label{e:5.8}
		\bbP \Big(\big( \rmW_n^{[r, r_n]}(u) \setminus \big(\rmU_n^{[r, r_n], \eta}(u) \cup  \rmB_n^{[r, r_n], \eta}(u)\big) \big) \cap \big(\rmG_n(u) \setminus \rmQ_n^{[r, r_n], \eta/4}(u)\big)  \neq \emptyset \Big) 
	\end{equation}
	tends to zero in the stated limits. Now, given $k \in [r,  r_n]$, if $x
	\in \rmW_n^k(u) \setminus \rmU_n^{k,\eta}(u)$ then by definition there exists some $y \in \bbX_k \cap \rmD_n^{k+2}$ such that 
	$x \in \rmB(y;k)$ and $\sqrt{\ol{L_{t_n^A}}(y; k+1)} \leq \sqrt{2}k + k^{1/2-\eta}$.  If $x$ is also in $\rmG_n(u) \setminus \rmQ_n^{[r, r_n], \eta/4}(u)$, then, thanks to~\eqref{e:3.1}, whenever $r$ is large enough we must have
	\begin{equation}
\label{e:5.25}
		\big|h_n(x)\big| \leq \sqrt{u} \,,\,\,
		\sqrt{\ol{h_n^2}(y;k+1)} >  k^{3/4-\eta/4} \,.
	\end{equation}
	By Lemma~\ref{l:103.2} and since $x \in \rmD_n^\circ$, the law of $h_n(x)$ is Gaussian with mean $0$ and variance at least $n/2-\log n - C$. Using Lemma~\ref{l:3.6} then shows that, if $r$ is large enough, the probability that the conditions in the last display are satisfied is at most
	\begin{equation}
		C \frac{\sqrt{u}}{\sqrt{n}} \rme^{-k^{1/2 - \frac{3}{4}\eta}} \,.
	\end{equation}
	At the same time, by definition of $\rmB_n^{k, \eta}(u)$ we have that $\big|\rmW_n^{k}(u) \setminus \rmB_n^{k, \eta}(u)\big| \leq |\bbW_n(u) \big| \rme^{k^{1/2-\eta}}$. Therefore, by the union bound, conditional on $L_{t_n^A}(\wh{\rmD}_n)$, whenever $r$ and $n$ are sufficiently large the probability in~\eqref{e:5.8} is at most
	\begin{equation}
		C \frac{\sqrt{u}}{\sqrt{n}} \big|\bbW_n(u)\big| 
		\sum_{k \in [r,r_n]} \rme^{k^{1/2-\eta}-k^{1/2-\frac{3}{4}\eta}}
		\leq
		C \frac{\sqrt{u}}{\sqrt{n}} \big|\bbW_n(u)\big|
		\rme^{-r^{1/2 - \eta}} \,.
	\end{equation}
	For any $\delta > 0$, on the event that $|\bbW_n(u)| \leq \delta^{-1} \sqrt{n}$ the right-hand side above goes to zero as $n \to \infty$ followed by $r \to \infty$. Since Lemma~\ref{l:3.2} shows that the latter event occurs with probability arbitrarily close to $1$ provided that $\delta$ is sufficiently small and $n$ is taken large enough, the result now follows by the union bound. 
\end{proof}

\section{Coarse clustering of vertices with low local time}
\label{s:6}
In this section we show the following (coarse) clustering result for the set of vertices with low local time:
\begin{prop}
For all $u > 0$, 
\label{p:6.1}
\begin{equation}
	\tfrac{1}{\sqrt{n}} \Big|\rmW^{[r_n, n-r_n]}_n(u) \Big|
	\overset{\bbP}{\underset{n \to \infty}\longrightarrow} 0 \,.
\end{equation}
In particular,
\begin{equation}
	\label{e:6.1b}
	\lim_{n \to \infty} \bbP \Big(\rmW^{[r_n, n-r_n]}_n(u) \cap \rmG_n(u) \neq \emptyset \Big) = 0 \,.
\end{equation}
\end{prop}
As in Theorem~\ref{t:3.1}, the proposition shows that clusters with large log-diameter do not survive the isomorphism. However, unlike in Theorem~\ref{t:3.1}, it also shows that the total number of vertices in such clusters and not just their number is $o(\sqrt{n})$ with high probability. In fact (and also unlike Theorem~\ref{t:3.1}) the former statement is a consequence of the latter.

To prove Proposition \ref{p:6.1} we will show that, up to a negligible fraction, all vertices in $\rmD_n^\circ$ with low local time have a downcrossings trajectory which exhibits a coarse form of entropic repulsion.  
More precisely, if for any $\eta,u,n\geq 0$ and $k \in [0,n]$ we let
\begin{equation}
	\rmO_n^{k,\eta}(u) := \Big \{ x \in \rmW_n(u) : 
	\sqrt{\wh{N}_{t_n^A}(x; k)} \leq \sqrt{2}k + n^{\eta} \Big\},
\end{equation}
with the definition extended to any set $K \subseteq[0,n]$ in place of $k$ via union over $k \in K$ as usual, then we will show that:
\begin{prop}
	\label{p:A.1} If $\gamma$ is chosen sufficiently small then there exists $\eta \in (\gamma,\tfrac{1}{2}-\eta_0)$ such that, for any $u \geq 0$ and $\delta > 0$,
	\begin{equation}
		\label{e:400.1}
		\lim_{n \to \infty} 
		\bbP \Big(\big|\rmO^{[r_n, n-r_n],\eta}_n(u) \big| > \delta \sqrt{n} \Big)= 0.
	\end{equation}
\end{prop}

Assuming Proposition~\ref{p:A.1}, whose proof will take up the majority of this section, let us show how to obtain Proposition~\ref{p:6.1}.
\begin{proof}[Proof of Proposition~\ref{p:6.1}]
Thanks to Proposition~\ref{p:A.1} and the union bound, it is sufficient to show that 
$n^{-1/2} \big|\rmW^{[r_n, n-r_n]}_n(u) \setminus \rmO^{[r_n, n-r_n],\eta}_n(u)\big|$ tends to zero in probability as $n \to \infty$.
By definition, for all $n$ large enough the mean size of the last set difference is at most
\begin{equation}
\label{e:6.5}
\sum_{x \in \rmD^\circ_n}
\sum_{k \in [r_n,n-r_n]} \,
\sum_{\substack{\log\|z-x\| \\\in (k-3, k)}} \,
\sum_{s = \lceil n^{\eta}\rceil}^{\infty}
\bbP \Big(x,z \in \rmW_n(u) ,\, 
\sqrt{\wh{N}_{t_n^A}(x; k)} \in \sqrt{2} k + [s-1, s) \Big) \,.
\end{equation}
Above we are using that if $x \in \rmW^{k}_n(u)$, then there exists $z \in \rmW_n(u)$ with $\log\|z-x\| \in (k-3,k)$.
Thanks to Proposition~\ref{prop:Abulk} and Proposition~\ref{l:4.6}, by conditioning on $\cF(x;k)$ we see that for all $n$ large enough this last probability is at most\begin{equation}
\exp \bigg\{ -\frac{\big(\sqrt{2}(n-k)-\frac{3}{4\sqrt{2}}\log n - s\big)^2}{n - k} - 4k - 4\sqrt{2}s + C k^\gamma
\bigg\} 
\leq \rme^{-2n-2k-2\sqrt{2}s + Cn^{\gamma}} \,. 
\end{equation}
Plugging this in~\eqref{e:6.5} and noticing that the first and third sums have at most $\rme^{2n}$ and $\rme^{2k}$ terms respectively, we obtain the upper bound
$\rme^{-2\sqrt{2}n^\eta + Cn^\gamma}$, which tends to zero as $n \to \infty$ if $\eta > \gamma$. We conclude the proof of the first part of the proposition via Markov's inequality.

For the second part, we can use the union bound to upper bound the probability therein conditional on $L_{t_n^A}(\rmD_n)$ by
\begin{equation}
	\sum_{x \in  \rmW^{[r_n, n-r_n]}_n(u)} \bbP(h_n^2(x) \leq u) \leq C_u \tfrac{1}{\sqrt{n}} \big|\rmW^{[r_n, n-r_n]}_n(u)\big| \,,
\end{equation}
where to bound the probabilities in the sum we have used that, for $x \in \rmD_n^\circ$,  $h_n(x)$ is a centered Gaussian with variance at least $n/2-\log n - C$ as shown in 
Lemma~\ref{l:103.2}.

 Then, for any $\delta > 0$, the probability in~\eqref{e:6.1b} is at most
\begin{equation}
	\bbP \Big(\tfrac{1}{\sqrt{n}} \Big|\rmW^{[r_n, n-r_n]}_n(u) \Big| > \delta\Big) + C_u \delta \,.
\end{equation}
Taking $n \to \infty$ and then $\delta \downarrow 0$, using the first part then gives the desired statement.
\end{proof}

\subsection{Coarse entropic repulsion}
The remaining of this section is devoted to proving Proposition~\ref{p:A.1}. To this end, we will need two auxiliary results.  
The first of these states that, with high probability, every vertex in $\rmD_n^{\circ}$ has a downcrossing trajectory which is not too low. More precisely, if for $\eta$ as in Proposition~\ref{p:A.1} we set
\begin{equation}
q_n:=n^{\eta},	
\end{equation} and, given $n \geq 1$ and $k \in [0,n]$, we define
\begin{equation}
	\label{e:6.1r}
	\rmR_n^{k,\eta} := \Big \{ x \in \rmD^\circ_n : 
	\sqrt{\wh{N}_{t_n^A}(x; k)} \leq \sqrt{2}k - q_n\Big\},
\end{equation} with the definition extended to any set $K \subseteq[0,n]$ in place of $k$ via union over $k \in K$ as usual, then 
one has the following:
\begin{lem}
	\label{l:6.3} For any $\eta \in (\gamma,\frac{1}{2}-\eta_0)$,  
	\begin{equation}\label{e:l.A.1}
		\lim_{n \to \infty} \bbP \Big( \rmR_n^{[q_n,n-q_n],\eta} \neq \emptyset \Big) = 0.
	\end{equation}
\end{lem}

The second auxiliary result tells us that, on the event that its downcrossing trajectory is~not too low, the probability that a given $x \in \rmD^{\circ}_n$ is a low local time vertex whose said trajectory is not suitably repelled \textit{upwards} is $o(\sqrt{n}/|\rmD^{\circ}_n|)$. The precise statement we shall need is contained in the following lemma:
\begin{lem}
	\label{l:6.4} If $\gamma$ is chosen sufficiently small then there exists some $\eta \in (\gamma,\frac{1}{2}-\eta_0)$ such that, for any $u \geq 0$, 
	\begin{equation}\label{e:l.A.2}
			\lim_{n \to \infty} \frac{|\rmD^{\circ}_n|}{\sqrt{n}} \sup_{x \in \rmD_n^{\circ}} 
			\bbP \Big( x \in \rmO^{[r_n, n-r_n], \eta}_n(u) \setminus \rmR_n^{[q_n,n-q_n],\eta}
			\Big) = 0.
	\end{equation}
\end{lem}

Before going into the proofs of these auxiliary results, let us see how to use Lemma~\ref{l:6.3} and Lemma~\ref{l:6.4} to conclude the proof of Proposition~\ref{p:A.1}. 

\begin{proof}[Proof of Proposition~\ref{p:A.1}] Choose $\gamma > 0$ small enough and $\eta \in (\gamma, \frac{1}{2}-\eta_0)$ so that  \eqref{e:l.A.1}--\eqref{e:l.A.2} both hold and, for each $n \in \N$ and $x \in \rmD^\circ_n$, consider the event
\begin{equation}
	B_n(x):= \big\{x \in \rmR_n^{[q_n,n-q_n],\eta}\big\}.
\end{equation} Furthermore, set $B_n:=\cup_{x \in \rmD^\circ_n} B_n(x)$ and observe that $B_n$ is the event from Lemma~\ref{l:6.3}. Then, by using the union bound and then Markov's inequality, we may bound the probability in \eqref{e:400.1} from above by
\begin{equation}
	\bbP(B_n) + \frac{1}{\delta\sqrt{n}}\bbE\Big(\big|\rmO^{[r_n, n-r_n],\eta}_n(u) \big| \,;\, B_n^\rmc\Big)
\end{equation} which, by linearity of the expectation, can be further bounded by
\begin{equation}\label{e:a2}
	\bbP(B_n) + \frac{1}{\delta\sqrt{n}}\sum_{x \in \rmD^\circ_n} \bbP\Big( \{ x \in \rmO^{[r_n, n-r_n],\eta}_n(u)\} \cap (B_n(x))^\rmc\Big).
\end{equation} Since the intersection $\{ x \in \rmO^{[r_n, n-r_n],\eta}_n(u)\} \cap (B_n(x))^\rmc$ corresponds exactly to the event in~\eqref{e:l.A.2}, Proposition~\ref{p:A.1} now follows at once from \eqref{e:a2} by Lemma~\ref{l:6.3} and Lemma~\ref{l:6.4}.\end{proof}

Thus, in order to complete the proof it only remains to prove Lemma~\ref{l:6.3} and Lemma~\ref{l:6.4}. For convenience, throughout the sequel we shall abbreviate, for $n \geq 1$, $k \in [0,n]$ and $\eta$ as in the statement of Proposition~\ref{p:A.1}, 
\begin{equation}\label{eq:a1}
\alpha^-_n(k):=\sqrt{2}k - q_n\hspace{1cm}\text{ and }\hspace{1cm}\alpha^+_n(k):=\sqrt{2}k + q_n.
\end{equation} 
\begin{proof}[Proof of Lemma~\ref{l:6.3}] For each $x \in \bbZ^2$ and $k \geq 1$, fix some $x_k \in \bbX_k$ such that $\|x-x'\| < \frac{1}{\sqrt{2}}\rme^k$.  
Then, given $0< \gamma < \eta' < \eta < \frac{1}{2}-\eta_0$, if $n$ is large enough (depending only on $\gamma,\eta_0,\eta$ and $\eta'$), for all $x \in \rmD_n^\circ$ and $k \in [q_n,n-q_n]$ we have
\begin{equation}
\rmB(x_k; k + \tfrac{1}{4}k^\gamma) \subseteq \rmB(x; k + \tfrac{1}{2}k^\gamma) \hspace{1cm} \text{ and }\hspace{1cm} \rmB(x;k+k^\gamma) \subseteq \rmB(x_k;k+2k^\gamma)
\end{equation} as well as
\begin{equation}
\rmB(x_k;k+2k^\gamma) \subseteq \rmB((x_k)_l; l) \hspace{1cm}\text{ and }\hspace{1cm}\ol{\rmB((x_k)_l;l+2l^\gamma)} \subseteq \rmB(x; n -2\log n) \subseteq \rmD_n,	
\end{equation} where for simplicity we abbreviate $l=l_n:=n-n^{\eta'}$ and $(x_k)_l \in \bbX_l$ satisfies $\| x_k - (x_k)_l\| \leq \frac{1}{\sqrt{2}}\rme^l$. 
In particular, if for $k \geq 1$ we write $\wh{N}'_t(x;k):=\wh{N}_t(x;k+\frac14k^\gamma,k+2k^\gamma)$, where $\wh{N}_t$ is as in \eqref{eq:defntz}, the union bound makes the probability in \eqref{e:l.A.1} less or equal than 
	\begin{equation}\label{eq:pfb}
	\sum_{y'} \bbP\Big( \sqrt{\wh{N}'_{t_n^A}(y';l)} < \beta_n^-(l)\Big) +  \sum_{k \in [q_n,n-q_n]} \sum_{x'} \bbP \Big( \sqrt{\wh{N}'_{t_n^A}(x';k)} \leq \alpha^-_n(k)\,,\sqrt{\wh{N}'_{t^A_n}(\{x'\}_l;l)} \geq \beta_n^-(l) \Big)
\end{equation} where $l=n-n^ {\eta'}$ as before, $\beta_n^-(l):=\sqrt{2}l - \frac{1}{2}q_n$, the leftmost sum is over all $y' \in \bbX_{l}$ such~that $\ol{\rmB(y';l+2l^\gamma)} \subseteq \rmD_n$ and the rightmost sum is over all $x' \in \bbX_{k}$ with $\rmB(x';k+2k^\gamma) \subseteq \rmB(x'_l ; l)$ and $\ol{\rmB(x'_l; l+2l^\gamma)} \subseteq \rmD_n$.

Now, by Proposition~\ref{prop:Abulk} but for $\wh{N}'_t$ instead of $\wh{N}_t$ (see the Appendix for a proof in this case), a straightforward calculation shows that, if $n$ is large enough (depending only on $\gamma,\eta$ and $\eta'$), for all $y' \in \bbX_{l}\cap \rmD_n$ as above we have
\begin{equation}\label{eq:Nlargebound}
	 \bbP\Big( \sqrt{\wh{N}'_{t_n^A}(y';l)} < \beta_n^-(l)\Big) \leq 2\rme^{-\frac{1}{8}n^{2\eta-\eta'}}
\end{equation} and for all $x' \in \bbX_k\cap \rmD_n$ as above we have
\begin{equation}
\bbP \Big( \sqrt{\wh{N}'_{t_n^A}(x';k)} \leq \alpha^-_n(k)\,,\sqrt{\wh{N}'_{t^A_n}(x'_l;l)} \geq \beta_n^-(l)\Big) \leq 2\rme^{-2(l-k)-\sqrt{2}q_n+Cl^\gamma }  \leq  2\rme^{-2(n-k)  - q_n}\,.		
\end{equation} Since $|\bbX_{l} \cap \rmD_n| \leq C_{\rmD} \rme^{2(n-l)}=C_{\rmD}\rme^{2n^{\eta'}}$ and $|\bbX_{k} \cap \rmD_n| \leq C_{\rmD} \rme^{2(n-k)}$ for some constant $C_\rmD > 0$ depending only on $\rmD$, this implies that the expression in \eqref{eq:pfb} vanishes as $n$ tends to infinity and from here the result now follows.
\end{proof}

Turning to the proof of Lemma~\ref{l:6.4}, we will first need to refine Proposition~\ref{l:4.4b} in the case $x$ is a low-local time vertex. To this end, given $k \geq 1$ and $\eta > 0$, we define
\begin{equation}
		\rmJ_n^{k,\eta}:=\Big\{ x \in \rmD^\circ_n : \Big|\sqrt{\ol{L_{t^A_n}}(x;k+1)} - \sqrt{\wh{N}_{t^A_n}(x;k)}\Big| > q_n\Big\}
	\end{equation} with the definition extended to any set $K \subseteq[0,n]$ in place of $k$ via union over $k \in K$ as usual. 
The refinement we shall need is contained in the following lemma:
	
\begin{lem}\label{l:A.11} For any $u \geq 0$ and $\eta \in (\gamma,\tfrac{1}{2}-\eta_0)$, 
\begin{equation}\label{eq:decayu}
\lim_{n \rightarrow \infty} \frac{\rme^{2n}}{\sqrt{n}} \sup_{x \in \rmD_n^\circ} \bbP\bigg( \sqrt{L_{t_n^A}(x)} \leq u\,,\,x \in \rmJ^{[q_n^2,n-q_n^2],\eta}_n \setminus \rmR_n^{[q_n,n-q_n],\eta}\bigg) = 0.
\end{equation}
\end{lem}

\begin{proof} By the union bound, it will be enough to show that, if $n$ is sufficiently large (depending only on $u$, $\gamma$, $\eta$ and $\eta_0$), for all $k \in [q_n^2,n-q_n^2]$ we have
\begin{equation}\label{eq:deck}
	\sup_{x \in \rmD_n^\circ} \bbP\bigg( \sqrt{L_{t_n^A}(x)} \leq u\,,\,x \in \rmJ^{k,\eta}_n \setminus \rmR_n^{[q_n,n-q_n],\eta}\bigg) \leq \rme^{-2n-q_n}.
\end{equation} To this end, we fix $k \in [q_n^2,n-q_n^2]$, define the scale $l:= \lfloor k - k^\gamma \rfloor$ and then split the event in~\eqref{eq:deck} according to the different values of $\sqrt{\wh{N}_{t^A_n}(x;k)}$ and $\sqrt{\wh{N}_{t^A_n}(x;l)}$. More precisely, if we abbreviate $\Lambda(x;k):=\big\{ x \in \rmJ^{k,\eta}_n \big\}$ and $\Theta(x):=\big\{ x \notin \rmR_n^{[q_n,n-q_n],\eta}\big\}$ for simplicity, it will suffice to show~that for all $n$ sufficiently large,
\begin{equation}\label{eq:deck2}
	\sup_{x \in \rmD_n^\circ} \sum_{w,\wt{w}}\bbP\bigg( \Big\{ \sqrt{L_{t_n^A}(x)} \leq u\,,\sqrt{\wh{N}_{t^A_n}(x;l)}=\wt{w}\,,\sqrt{\wh{N}_{t^A_n}(x;k)}=w\Big\} \cap \Lambda(x;k) \cap \Theta(x) \bigg) \leq \rme^{-2n-q_n},
\end{equation} where the sum is over all pairs $(w,\wt{w}) \in \cN_k \times \cN_l$, with $\cN_k$ defined as in \eqref{eq:defnk}. To show this, we will split the sum into different parts and estimate each part separately.

First of all, let us observe that, since $l,k \in [q_n,n-q_n]$ for all $n$ large enough by choice of $k$, we may restrict the sum in \eqref{eq:deck2} only to pairs $(w,\wt{w})$ such that $w > \alpha_n^-(k)$ and $\wt{w} > \alpha_n^-(l)$. Moreover, since $l +l^\gamma < k \leq n - q_n^2$ by definition of $l$, if $n$ is sufficiently large then for all $x \in \rmD_n^\circ$ we have $\rmB(x; l + l^\gamma) \subseteq \ol{\rmB(x;k+k^\gamma)} \subseteq \rmB(x;n-q_n)$ and $\rmB(x;n-q_n + n^\gamma) \subseteq \rmB(x; n-2\log n) \subseteq \rmD_n$. In particular, this allows us to use Propositions~\ref{l:4.4b}, \ref{prop:Abulk} and~\ref{prop:Abulk2} to bound the different terms of the sum in~\eqref{eq:deck2}.

Indeed, abbreviating $n':=n-q_n$ for simplicity, if $n$ is large enough then, on the one hand, by Proposition~\ref{prop:Abulk2} we have that
\begin{equation}\label{eq:blargecross}
\begin{split}
	\sup_{x \in \rmD_n^\circ} \bbP\Big( \sqrt{L_{t^A_n}(x)} \leq u  \,\Big|\,  \sqrt{\wh{N}_{t^A_n}(x;k)} > \alpha^-_n(n')-1\,;\cF(x;k) \Big) &\leq \rme^{-2k -2\sqrt{2}(\alpha_n^-(n')-1-\sqrt{2}k) +C_u n^\gamma} \\ &\leq \rme^{-2n -q_n^2},
\end{split}
\end{equation} where for the last bound we have used that $n-k \geq q^2_n$ and that $\eta > \gamma$. Since $n-l \geq q_n^2$ as well because $l \leq k$ by definition, in the same manner one can show that for all $n$ large enough, 
\begin{equation}\label{eq:blargecross2}
	\sup_{x \in \rmD_n^\circ} \bbP\Big( \sqrt{L_{t^A_n}(x)} \leq u  \,\Big|\,  \sqrt{\wh{N}_{t^A_n}(x;l)} > \alpha^-_n(n')-1\,;\cF(x;l) \Big) \leq \rme^{-2n -r_n}.
\end{equation} In particular, upon ignoring the event $\Lambda(x;k)$ in \eqref{eq:deck2}, \eqref{eq:blargecross}--\eqref{eq:blargecross2} combined together imply that the sum in \eqref{eq:deck2} over all pairs $(w,\wt{w})$ such that $\max\{w,\wt{w}\} > \alpha_n^-(n') -1$ is at most $2\rme^{-2n-q_n^2}$ uniformly in $x \in \rmD_n^\circ$.

On the other hand, if $n$ is large enough then, for any $w \in \cN_k$ with $\alpha_n^-(k) < w \leq \alpha_n^-(n')-1$, by first conditioning on $\cF(x;n')$ and then using Proposition~\ref{prop:Abulk} together with the fact that $\gamma < \eta$ we obtain that 
\begin{equation}\label{eq:bestimate1}
\sup_{x \in \rmD_n^\circ} \bbP\Big( \sqrt{\wh{N}_{t^A_n}(x;k)}  = w \,,\sqrt{\wh{N}_{t^A_n}(x;n')} > \alpha_n^-(n')\Big) \leq 2 \rme^{-2(n-k)+2\sqrt{2}s_w+Cq_n}	
\end{equation} with $s_w:= w - \sqrt{2}(k+k^\gamma)$ whereas, for any $\wt{w} \in \cN_l$ such that $\wt{w} > \alpha_n^-(l)$, Proposition~\ref{prop:Abulk2} yields 
\begin{equation}\label{eq:bestimate2}
\sup_{x \in \rmD_n^\circ} \bbP\Big( \sqrt{L_{t^A_n}(x)} \leq u \,\Big|\, \sqrt{\wh{N}_{t^A_n}(x;l)} = \wt{w}\,;\cF(x;l)\Big) \leq \rme^{-2l -2\sqrt{2}s_{\wt{w}}+C_u n^{\gamma}}, 		
\end{equation} with $s_{\wt{w}}:=\wt{w}-\sqrt{2}(l+\frac{1}{2}l^\gamma)$. Moreover, if $n$ is large enough then, for any constant $C_0 > 0$ and all $w,\wt{w}$ such that $\alpha_n^-(k) < w \leq \alpha_n^-(n')-1$, $\alpha_n^-(l) < \wt{w} \leq \alpha_n^-(n')-1$ and $|s_w - s_{\wt{w}}| > \sqrt{C_0k^{\gamma}q_n}$, by Proposition~\ref{prop:Abulk} again together~with the fact that $k+k^\gamma - (l+\tfrac{1}{2}l^\gamma) \leq 3k^\gamma$ we obtain that
\begin{equation}\label{eq:boundcross}
	\sup_{x \in \rmD_n^\circ} \bbP\Big( \sqrt{\wh{N}_{t^A_n}(x;l)} = \wt{w}\,\Big|\, \sqrt{\wh{N}_{t^A_n}(x;k)}  = w\,;\cF(x;k)\Big) \leq 2\rme^{-2(k-l) -2\sqrt{2}(s_{w}-s_{\wt{w}})-\frac{C_0}{3}q_n}.
\end{equation} In particular, if we choose $C_0$ sufficiently large then, by ignoring the event $\Lambda(x;k)$ in \eqref{eq:deck2} again and then conditioning first on~$\cF(x;l)$ and afterwards on $\cF(x;k)$, the last three estimates combined imply that the sum in \eqref{eq:deck2} over all pairs $(w,\wt{w})$ such that $\max\{w,\wt{w}\} \leq \alpha_n^-(n')-1$ and $|s_w-s_{\wt{w}}| > \sqrt{C_0 k^{\gamma}n^{\eta}}$ is at most
\begin{equation}\label{eq:b1}
	\sum_{\max\{w,\wt{w}\} \,\leq\, \alpha_n^-(n')} 4\rme^{-2n + Cq_n+C_u n^{\gamma}-\frac{C_0}{3}q_n} \leq \rme^{-2n -2q_n},
\end{equation} uniformly in $x \in \rmD_n^\circ$ for all $n$ large enough, since the number of all pairs $(w,\wt{w})$ satisfying that $\max\{w,\wt{w}\} \leq \alpha_n^-(n')$ is at most~$16n^4$.

Finally, by a straightforward computation using Proposition~\ref{l:4.4b} we obtain that, for any $w$ such that $\alpha_n^-(k) < w \leq \alpha_n^-(n')-1$, if $n$ is large enough then
\begin{equation} \label{eq:bestimate3}
	\sup_{x \in \rmD_n^\circ} \bbP \Big( \Lambda(x;k) \,\Big|\, \sqrt{\wh{N}_{t^A_n}(x;k)}  = w\,;\cF(x;k)\Big) \leq 2\rme^{-c\frac{n^{2\eta}}{k^\gamma}} \leq 2\rme^{-cn^{2\eta-\gamma}}.
\end{equation} Hence, by conditioning first on $\cF(x;l)$ and then on $\cF(x;k)$, \eqref{eq:bestimate1}, \eqref{eq:bestimate2} and \eqref{eq:bestimate3} yield that the sum in \eqref{eq:deck2} over all $(w,\wt{w})$ such that $\max\{w,\wt{w}\} \leq \alpha_n^-(n')-1$ and $|s_w - s_{\wt{w}}| \leq \sqrt{C_0k^{\gamma}n^{\eta}}$ is at most
\begin{equation}
\sum_{\max\{w,\wt{w}\} \leq \alpha_n^-(n')}4 \rme^{-2n +2(k-l) + 2\sqrt{2C_0 k^{\gamma}n^{\eta}} + Cn^{\eta}+C_u n^\gamma - cn^{2\eta-\gamma}} \leq \rme^{-2n -\frac{c}{2}n^{2\eta-\gamma}}
\end{equation} uniformly in $x \in \rmD_n^\circ$ for all $n$ sufficiently large, where to obtain the last bound we have used~that $2\eta - \gamma > \max\{\eta,\gamma\}$. Upon combining this estimate with \eqref{eq:b1} and the bound for pairs $(w,\wt{w})$ such that $\max\{w,\wt{w}\} > \alpha_n^-(n')-1$, we conclude that for all $n$ large enough the sum in \eqref{eq:deck2} is at most
\begin{equation}
	\rme^{-2n}(2\rme^{-q_n^2}+\rme^{-2q_n}+\rme^{-cn^{2\eta-\gamma}}) \leq \rme^{-2n-q_n	}
\end{equation} uniformly in $x \in \rmD_n^\circ$ and thus \eqref{eq:deck2} now immediately follows.
\end{proof}

We are now ready to conclude the proof of Lemma~\ref{l:6.4}.

\begin{proof}[Proof of Lemma~\ref{l:6.4}] Recalling the isomorphism, for $u \geq 0$, $\eta \in (0,\tfrac{1}{4}-\tfrac{\eta_0}{2})$, $n \in \N$ and $x \in \rmD_n^\circ$, let us define the sets	    \begin{align}
	    A_n(x):&= \Big\{ L_{t^A_n}(x) \leq u \,,\,\sqrt{\ol{L_{t_n^A}}(x; k+1)} > \sqrt{2}k - q_n^{3/2} \text{ for all } k \in [q_n^2,n-q_n^2]\,, \nonumber\\ 
		&\hspace{2.92cm}\sqrt{\ol{L_{t_n^A}}(x; k+1)} \leq \sqrt{2}k + q_n^{3/2} \text{ for some }k \in [r_n,n-r_n]\Big\}\,,\\ \nonumber\\
		B_n(x):&=\Big\{h_n^2(x) \leq u \,,\, \sqrt{\ol{h^2_n}(x;k+1)} \leq k^{\frac{1+\eta}{2}} \text{ for all } k \in [r_n, n-r_n] \Big\}\,, \\ \nonumber \\
		C_n(x):&=\Big\{ f_n^2(x) \leq 2u\,,\, \sqrt{\ol{f_n^2}(x; k+1)} > \sqrt{2}k - q_n^2 \text{ for all } k \in [q_n^2,n-q_n^2]\,, \nonumber\\ &\hspace{2.92cm}\sqrt{\ol{f_n^2}(x; k+1)} \leq \sqrt{2}k + q_n^2 \text{ for some }k \in  [r_n,n-r_n] \Big\}\,.
	 	\end{align}		

Then, by the isomorphism we have, for all $n$ large enough (depending only on $\eta$) and $x \in~\rmD_n^\circ$, the inclusion $A_n(x) \cap B_n(x) \subseteq C_n(x)$. On the other hand, since for any $x \in \rmD_n^\circ$ we have that $h_n(x)$ is a centered Gaussian random variable satisfying $\frac{n}{2}-\log n -C_\rmD \leq \E(h_n^2(x)) \leq \frac{n}{2}+C_\rmD$ for some $C_\rmD>0$ by Lemma~\ref{l:103.2}, a straightforward computation using Lemma~\ref{l:3.6} shows that $\inf_{x \in \rmD_n^\circ} \bbP(B_n(x)) \geq~\frac{C_u}{\sqrt{n}}$ for some $C_u>0$ and all $n$  large enough (depending only on $\eta$ and $\rmD$). Furthermore, it follows from Proposition~\ref{p:3.17} that if $\eta \in (0,\frac{1}{4}-\frac{\eta_0}{2})$ is chosen small enough then we have $\lim_{n \to \infty} \rme^{2n} \sup_{x \in \rmD_n^\circ} \bbP(C_n(x))=0$. Therefore, since $A_n(x)$ and $B_n(x)$ are independent by the independence of $L_{t^A_n}$ and $h_n$, we conclude that $\lim_{n \to \infty} \frac{\rme^{2n}}{\sqrt{n}}\sup_{x \in \rmD_n^\circ} \bbP(A_n(x))=0$ for this particular choice of $\eta$.
	
	Now, on the other hand, if $\Omega_n(x)$ denotes the event in~\eqref{e:l.A.2} then, since we have that $q_n^2 \leq r_n$ by choice of $\eta$, it is immediate to see~that 
	\begin{equation}
	\Omega_n(x) \setminus A_n(x) \subseteq \Big\{ L_{t_n^A}(x) \leq u\,,\,x \in \rmJ^{[q_n^2,n-q_n^2],\eta}_n \setminus \rmR_n^{[q_n,n-q_n],\eta}\Big\}
	\end{equation} so that, if $\gamma$ is taken sufficiently small so as to have $\gamma < \eta$ (where $\eta$ is taken as before) then  by Lemma~\ref{l:A.11} we obtain that $\lim_{n \rightarrow \infty} \frac{\rme^{2n}}{\sqrt{n}} \sup_{x \in \rmD_n^\varepsilon} \bbP( \Omega_n(x) \setminus A_n(x)) = 0$. Hence, since $|\rmD_n^\circ| \leq C_\rmD \rme^{2n}$ for some constant $C_\rmD > 0$ depending only on $\rmD$, the result now follows from these two limits by the union bound.
	\end{proof}

\section{Proofs of Theorem~\ref{t:2.1} (Phase A) and Theorem~\ref{t:2b} (Phase B)}
\label{s:7}

In this section we build on the results in the previous sections, mainly the coarse and sharp clustering, to prove Theorem~\ref{t:2.1} and Theorem~\ref{t:2b}.

\subsection{Proof of Theorem~\ref{t:2.1}}
While upper tightness of $|\bbW_n(0)|/\sqrt{n}$ is readily given by Lemma~\ref{l:3.2}, for tightness away from zero we will appeal to Theorem~\ref{t:3.1}. As a first step, we show that $|\bbW_n(u)|/\sqrt{n}$ is positive with probability tending to $1$ as $u \to \infty$ uniformly in $n$, namely:
\begin{lem} 
\label{l:6.1}
It holds that
\begin{equation}
\lim_{u \to \infty} \lim_{\delta \to 0} \limsup_{n \to \infty} \bbP \big(\big|\bbW_n(u)\big| \leq \delta \sqrt{n} \Big) = 0 \,. 
\end{equation}
\end{lem}
\begin{proof}
Given $u \geq 0$, since $\rmG_n(u) \subseteq\rmW_n(u) = \rmW_n^{[0,n-r_n]}(u)$,  for all $r,n \geq 1$ we have
\begin{multline}
\label{e:6.2}
\bbP \Big( \rmG_n(u) \cap \rmW_n^{[0,r]}(u) = \emptyset\Big) \leq \bbP \Big( \rmG_n(u) = \emptyset \Big) +  
 \bbP \Big(\rmG_n(u) \cap \rmW_n^{[r, r_n]}(u) \neq \emptyset\Big) \\
+
 \bbP \Big(\rmG_n(u) \cap \rmW_n^{[r_n,n-r_n]}(u) \neq \emptyset\Big) \,.
\end{multline}
The first probability on the right-hand side is at least one minus
\begin{equation}\label{eq:l.7.1.1}
 \bbP\big(\exists x \in  \rmW_n^{[0,r]}(u) :\: h_n^2(x) \leq u \big) \,.
\end{equation}
Since, for any $x \in \rmD_n^\circ$ by Lemma~\ref{l:103.2} we have that $h_n(x)$ is a centered Gaussian random variable with variance at least $\frac{n}{2} - \log n -C$, it follows that for any such $x$ the probability that $h^2_n(x) \leq u$ is at most $C u/\sqrt{n}$. Moreover, since by definition $\rmW_n^{[0,r]}(u) \leq C \rme^{2r} \bbW_n^{[0,r]}(u)$, conditional on 
$L_{t_n^A}$ the probability in~\eqref{eq:l.7.1.1} is at most 
\begin{equation}
	C_{u,r} \frac{1}{\sqrt{n}}\big|\bbW_n^{[0,r]}(u)\big| \,,
\end{equation}
for some $C_{u,r} \in (0,\infty)$.  Since this last display will be smaller than $1/2$ as soon as $\big|\bbW_n^{[0,r]}(u)\big| < (2C_{u,r})^{-1} \sqrt{n}$, a simple computation yields that
\begin{equation}
\bbP \big( \rmG_n(u) \cap \rmW_n^{[0,r]}(u)  = \emptyset \big) \geq 
\tfrac12 \bbP \Big(\big|\bbW_n^{[0,r]}(u)\big| < (2C_{u,r})^{-1} \sqrt{n} \Big) \,.
\end{equation}
Combined with~\eqref{e:6.2}, this gives for any $r \geq 1$ and  $\delta < (2C_{u,r})^{-1}$, 
\begin{multline}\label{eq:l.7.1.1.2}
\bbP \Big(\big|\bbW_n(u)\big| < \delta \sqrt{n} \Big) 
\leq 2 \bbP \Big( \rmG_n(u) = \emptyset \Big)   + 
 2 \bbP \Big(\rmG_n(u) \cap \rmW_n^{[r, r_n]}(u) \neq \emptyset\Big) 
\\ + 2 \bbP \Big(\rmG_n(u) \cap \rmW_n^{[r_n,n-r_n]}(u) \neq \emptyset\Big) \,.
\end{multline}
The right-hand side in \eqref{eq:l.7.1.1.2} can now be made arbitrarily small by choosing first $u$ large enough, then $r$ large enough and finally $n$ large enough, thanks to Proposition~\ref{p:3.2},  Theorem~\ref{t:3.1} and Proposition~\ref{p:6.1}.
\end{proof}

Next, we boost the previous result to tightness above $0$ of $\frac{1}{\sqrt{n}}\bbW_n(0)$. To this end, we first show that the downcrossings trajectory of vertices in $\rmW_n(u)$ cannot be too repelled. This is done as usual by an application of the Thinning and Resampling lemmas.  Accordingly, given $u \geq 0$ and $\eta \in (0,1/2)$, for $n \geq 1$, $k \in [1, r_n]$ define
\begin{multline}
	\label{e:6.1}
	\bbM_n^{k, \eta}(u) :=  \Big \{ z \in \bbW_n(u) :\: 
	\exists y \in \bbX_k \cap \rmD_n^{k+k^\gamma+1}
	\ \text{ s.t. } 
	\rmW_n(u) \cap \rmB(z; \lfloor n-r_n\rfloor)  \subseteq  \rmB(y;k),\,\\
	\sqrt{\wh{N}_{t_n^A}(y; k)} > \sqrt{2} k + k^{1/2+\eta} \Big\} \,.
\end{multline} with the definition extended to any set $K \subseteq[0,r_n]$ in place of $k$ via union over $k \in K$ as usual.
We shall then prove:
\begin{lem}
	\label{l:4.2} Given $u \geq 0$ and $\eta \in (0,1/2)$, for any $\delta > 0$ we have
	\begin{equation}
		\label{e:4.12}
		\lim_{r \to \infty} \limsup_{n \to \infty} \bbP \Big(
			\big|\bbM^{[r, r_n], \eta}_n(u)\big| > \delta \sqrt{n} \Big) = 0\,.
	\end{equation}
\end{lem}	

As usual, we first prove an analogous statement with the number of downcrossings replaced by the harmonic average of the local time, namely with the set $\bbM_n^{k,\eta}(u)$ replaced by
\begin{multline}
	\bbN_n^{k, \eta}(u) :=  
	\Big \{ z \in \bbW_n(u) :\: 
	\exists y \in \bbX_k \cap \rmD_n^{k+2}
	\ \text{ s.t. } 
	\rmW_n(u) \cap \rmB(z; \lfloor n-r_n\rfloor) \subseteq \rmB(y;k)  ,\,\\
	\sqrt{\ol{L_{t_n^A}}(y; k+1)} > \sqrt{2} k + k^{1/2+\eta} \Big\} \,.
\end{multline} Our first task will be then to show:
	
\begin{lem}
	\label{l:4.2a}
	Given $u \geq 0$ and $\eta \in (0,1/2)$, for any $\delta > 0$ we have
	\begin{equation}
		\label{e:4.12a}
		\lim_{r \to \infty} \limsup_{n \to \infty} \bbP \Big(
		 \big|\bbN^{[r,r_n], \eta}_n(u)\big| > \delta \sqrt{n} \Big) = 0\,.
	\end{equation} 
\end{lem}	

\begin{proof}
Apply the Thinning Lemma (Lemma~\ref{l:3.2g}) with $\rmY_n:=\rmD_n$ and
$\cA_{k} := (\sqrt{2}k + k^{1/2+\eta}, \infty)$ for $k \geq r$. 
Noticing that, for this choice of $\cA_k$, we have $\cB_k = \cA_k$ for all $k \geq r$ and $\bbN_n^{[r,r_n],\eta}(u) \subseteq \bbZ_n(u)$, where $\bbZ_n(u)$ is the set from Lemma~\ref{l:3.2g}, if we take $r$ large enough then Lemma~\ref{l:3.2g} gives
\begin{equation}
	\bbP \Big(  \big| \bbN_n^{[r,r_n],\eta}(u) \big| > \delta \sqrt{n}  \Big) 
	\leq C(1+\delta^{-1}) 
	\bbP \Big( \big| \rmQ_n^{[r,r_n],\eta}(u+1) \big| > c \delta \Big) 
\end{equation}
for all $\delta > 0$ and $n$ large enough.  The result then follows at once from Proposition~\ref{p:3.3}.
\end{proof}
It is now a short step towards:
\begin{proof}[Proof of Lemma~\ref{l:4.2}] For $r$ large enough, the event in~\eqref{e:4.12} is included in the union of the event in~\eqref{e:4.12a} with $(\frac{\eta}{2},\frac{\delta}{2})$ in place of the pair $(\eta,\delta)$ and the one in~\eqref{e:3.13c}. The result then follows from Lemma~\ref{l:4.2a} and Lemma~\ref{l:6.7n} together with the union bound.
\end{proof}

We are finally ready for:
\begin{lem}
\label{l:7.4}
	It holds that
	\begin{equation}\label{eq:l.7.4}
		\lim_{\delta \to 0} \limsup_{n \to \infty} \bbP \big(\big|\bbW_n(0)\big| \leq \delta \sqrt{n} \Big) = 0 \,. 
	\end{equation}
\end{lem}
\begin{proof}
Fix $\eta > 0$. Then, thanks to Lemma~\ref{l:6.1}, Lemma~\ref{l:4.2}, Theorem~\ref{t:3.1} and Proposition~\ref{p:6.1}, for any $\epsilon > 0$ we may find $u$ large enough, then $\delta'>0$ small enough and finally $r$ large enough so that, for all $n$ large enough, 
\begin{equation}
\label{e:6.8}
\bbP \Big(\big|\bbW^{[0,r]}_n(u) \setminus \bbM^{\lfloor r \rfloor +2,\eta}_n(u)\big| > \delta' \sqrt{n}\Big) > 1-\epsilon \,.
\end{equation} We will now show that, with overwhelming probability as $n \rightarrow \infty$, a positive fraction of vertices in $\bbW^{[0,r]}_n(u) \setminus \bbM^{\lfloor r \rfloor +2,\eta}_n(u)$ belongs to $\bbW_n(0)$, which will readily imply that the event in \eqref{eq:l.7.4} has vanishing probability.
We will show this by employing a variation of the Resampling Lemma. To this end, given $r \geq 1$, we set $r_* := \lfloor r \rfloor+ 
2$ and for each $n \geq r$ define
\begin{equation}
\begin{split}
\bbY_n(u) := \Big \{ z \in \bbW^{[0,r]}_n(u)  : \:  \exists y \in \bbX_{r_*} \cap \rmD_n^{r_*+r_*^\gamma+1} \text{ s.t. }& \rmW_n(u) \cap \rmB(z,\lfloor n -r_n\rfloor ) \subseteq \rmB(y;r_*) \\  & \sqrt{\wh{N}_{t_n^A}(y; r_*)} \leq \sqrt{2} r_* + r_*^{1/2+\eta}\Big\} \,
\end{split}
\end{equation}and, for each $i=1,\dots,9$ and $j=1,\dots,J_{r_*,\lceil r_* +r_*^\gamma \rceil}$, 
\begin{equation}
\begin{split}
	\bbY_n(u;i,j):=\Big\{ z \in &\, \bbX_{\lfloor n-r_n\rfloor,\lfloor n-r_n\rfloor}(i) : \exists! \,y(z) \in \bbX_{r_*,\lceil r_* +r_*^\gamma \rceil}(j)\cap \rmD_n^{r_*+r_*^\gamma+1}\text{ s.t. } \\ & \rmB(y(z);r_*) \cap \rmB(z; \lfloor n - r_n \rfloor) \cap \rmW_n(u) \neq \emptyset\,, \sqrt{\wh{N}_{t_n^A}(y; r_*)} \leq \sqrt{2} r_* + r_*^{1/2+\eta}\Big\}\,.
\end{split} 
\end{equation} Then, by choice of $r_*$ and since the balls $(\rmB(y;r_*) ; y \in \bbX_{r_*,\lceil r_* + r_*^\gamma\rceil})$ are disjoint for any fixed $j$, it is not hard to check that for all $n$ large enough (depending on $r$) we have the inclusions
\begin{equation}
\label{e:6.9}
\bbW^{[0,r]}_n(u) \setminus \bbM^{r_*,\eta}_n(u) \subseteq \bbY_n(u) \subseteq \bigcup_{i=1}^9 \bigcup_{j=1}^{J_{r_*,\lceil r_* +r_*^\gamma \rceil}}\bbY_n(u;i,j)\,. 
\end{equation} Furthermore, if for each $i,j$ as above we define
\begin{equation}
	\bbZ_n(u;i,j):=\Big\{ z \in \bbY_n(u;i,j) :  L_{t^A_n}(x)=0 \text{ for all }x \in \rmB(y(z); r_*)\Big\}\,,
\end{equation} then by construction we obtain that $\bbZ_n(u;i,j) \subseteq \bbW_n(0)$. Therefore, by arguing as in the proof of Lemma~\ref{l:4.1n} we see that, conditional on $|\bbY_n(i,j)|$, for all $n$ large enough (depending only on $r$), 
the law of $|\bbW_n(0)|$ stochastically dominates that of a Binomial with $|\bbY_{n}(i,j)|$ number of trials and success probability
\begin{equation}
	p_{r_*} := \inf_{n,y} \left[ \text{ess-inf }\, \bbP\Big(L_{t_n^A}(x) = 0 \text{ for all }x \in \rmB(y;r_*) \,\Big|\, \sqrt{\wh{N}_{t_n^A}(y;r_*)} \leq \sqrt{2}r_* + r_*^{1/2+\eta} \,; \cF(y; r_*)\Big)\right] ,
\end{equation} where the infimum is over all $n \geq r$ such that $\bbX_{r*} \cap \rmD_n^{r_*+r_*^\gamma+1} \neq \emptyset$ and $y \in \bbX_{r*} \cap \rmD_n^{r_*+r_*^\gamma+1} \neq \emptyset$. 
Note that $p_{r_*} > 0$ if $r$ is large enough by the Markov property of the random walk. In particular, by Chebyshev's inequality this Binomial will be at least $|\bbY_{n}(i,j)| p_{r_*} /2 - \log n$ with (conditional) probability at least $1-1/\log n$. Taking expectation, it follows 
that $|\bbW_n(0)| > |\bbY_{n}(i,j)| p_k /2 - \log n$ with probability tending to $1$ as $n \to \infty$, so that the union bound over all $i$ and $j$ and~\eqref{e:6.9} together yield 
\begin{equation}
\label{e:6.11}
\bbP \Big(\big|\bbW_n(0)\big| > \frac{ p_{r_*}}{18J_{r_*,\lceil r_* +r_*^\gamma \rceil}}\big|\bbW^{[0,r]}_n(u) \setminus \bbM^{r_*,\eta}_n(u) \big| - \log n \Big) > 1-\epsilon 
\end{equation}
for all $n$ large enough.
Thus, by combining~\eqref{e:6.8} with~\eqref{e:6.11}, for any $\delta < \frac{p_{r_*}}{18J_{r_*,\lceil r_* +r_*^\gamma \rceil}}\delta'$ we have
\begin{equation}
	\bbP \big(\big|\bbW_n(0)\big| \leq \delta \sqrt{n} \big) < 2 \epsilon \,
\end{equation} for all $n$ large enough which, since $\epsilon$ can be taken arbitrarily small, completes the proof.
\end{proof}

For the third statement in Theorem~\ref{t:2.1} we need to show that, by the time phase A is finished, there will be no uncovered points in~$\rmD_n \setminus \rmD_n^\circ$. This is captured in the following lemma.
\begin{lem}
	\label{l:7.5}
	\begin{equation}
		\label{e:7.21}
		\lim_{n \to \infty} \bbP \big(\rmF_{n,t_n^A}(0) \setminus \rmD_n^\circ \neq \emptyset \big) = 0 \,.
	\end{equation}
\end{lem}

\begin{proof}
	The mean size of the set inside the probability in~\eqref{e:7.21} is at most
	\begin{equation}
		C\,\text{Len}(\partial \rmD) \, \rme^{2n-2\log n} \rme^{-\frac{t_n^A}{n}} = o(1) \,,		
	\end{equation}
	by Lemma~\ref{lem:nhprob}, where $\text{Len}(\partial \rmD)$ is the Euclidean length of the boundary and $C=C(\rmD) \in (0,\infty)$ is some universal constant. The result now follows by Markov's inequality.
\end{proof} 

Finally, we arrive at:
\begin{proof}[Proof of Theorem~\ref{t:2.1}]
The first statement follows from Lemma~\ref{l:3.2} and Lemma~\ref{l:7.4}. The second is precisely the statement of Proposition~\ref{p:6.1}. The third is a direct consequence of Lemma~\ref{l:7.5}.
\end{proof}

\subsection{Proof of Theorem~\ref{t:2b}}
For the proof of Theorem~\ref{t:2b}, we shall need a few new definitions and two preliminary lemmas. Extending the definition in~\eqref{e:1.4}, the cover time of $\rmA$ will be denoted by $\check{\tau}_\rmA := \max_{x \in \rmA} \tau_{x}$. Then:
\begin{lem}
\label{l:8.2}
For all $n$ large enough, $x \in \rmD_n^\circ$ and $y \in \partial \rmB(x; 2 r_n+1)$, 
\begin{equation}
\bbP_y \big(\tau_{\partial} < \ck{\tau}_{\rmB(x;r_n)} \big) \leq n^{-\eta_0/3}\,.
\end{equation}
\end{lem}
\begin{proof}
For any $\eta > 0$, the probability in the statement of the lemma is at most
\begin{equation}
\label{e:8.2}
\bbP_y \big(\tau_{\partial} < \ck{\tau}_{\rmB(x;r_n)} \,\big|\, 
\bfN_{\tau_{\partial \rmD_n}}(x; 2r_n, 2r_n+1) \geq r_n^{2+\eta} \big)
+ 
\bbP_y \big(\bfN_{\tau_{\partial \rmD_n}}(x; 2r_n, 2r_n+1) < r_n^{2+\eta} \big)\,.
\end{equation}
Since $x \in \rmD_n^\circ$, it follows from Lemma~\ref{l:103.2} and Lemma~\ref{lem:g} that
$\bbP_{y'}\big(\tau_{\partial}< \tau_{\rmB(x;2r_n)}\big)$ is at most $3/n$ uniformly over all $y' \in \partial \rmB(x;2r_n+1)$ for all $n$ large enough. Consequently, 
under $\bbP_y$, the law of $\bfN_{\tau_{\partial}}(x; 2r_n, 2r_n+1)$ stochastically dominates the Geometric distribution (starting from~$0$) with success probability $3/n$ and, by a standard estimate, the second probability in~\eqref{e:8.2} is at most $n^{-\eta_0/2}$ as soon as $\eta$ is small enough so that $r_n^{2+\eta} < n^{1-\eta_0}$.

At the same time, by~\eqref{e:3.2},~\eqref{e:rel} and Lemma~\ref{l:103.2}, for any $w \in \partial \rmB(x;2r_n)$ and $z \in \rmB(x;r_n)$, 
\begin{multline}
\bbP_w(\tau_z < \tau_{\partial \rmB(x;2r_n+1)}) = \frac{G_{\rmB(x;2r_n+1)}(w,z)}{G_{\rmB(x; 2r_n+1)}(z,z)} \\ 
\geq \frac{\log \rmd(w, \partial \rmB(x;2 r_n+1)) - \log \|w - z\| + O(\|w-z\|^{-2}) }{2 r_n+C}
\geq cr_n^{-1}
\end{multline}
for some $c > 0$ and all $n$ large enough. Hence, since $\ol{\rmB(x; 2r_n+1)} \subseteq \rmD_n$ for all $n$ large enough because $x \in \rmD_n^\circ$, the union bound now gives that the first term in \eqref{e:8.2} is at most
\begin{equation}
C\rme^{2r_n} \big(1-cr_n^{-1}\big)^{r_n^{2+\eta}} = \rme^{2r_n-r_n^{1+\eta/2}} \,,
\end{equation}
which goes to $0$ as $n \to \infty$ stretched exponentially fast for any fixed $\eta > 0$.
\end{proof}

Next, recall the definition of a $(r_n,n-r_n)$-clustered set introduced right before the statement of Theorem~\ref{t:2b}. If $\rmA$ is a $(r_n,n-r_n)$-clustered set, any set of points $\wh{\rmA}$ of size $\chi_n(\rmA)$ such that 
$\rmA \subseteq\cup_{x \in \wh{\rmA}} \rmB(x;r_n)$ and $\log \|x - x'\| > \lfloor n-r_n\rfloor - 3$ for any pair of different $x, x' \in \wh{\rmA}$ will be called a {\em skeleton} of $\rmA$. It is simple to check that at least one such skeleton exists if $n$ is large enough. Then:
\begin{lem}
\label{l:8.3}
Let $\rmA \subseteq\rmD_n^\circ$ be a $(r_n,n-r_n)$-clustered set such that $\chi_n(\rmA) < n^{1/2+\eta_0/4}$ and take any skeleton $\wh{\rmA}$ of $\rmA$ and $x \in \wh{\rmA}$. Then, as $n \to \infty$, 
\begin{equation}
\label{e:8.1a}
\bbP_\partial \Big(\tau_{\partial \rmB(x;2r_n+1)} < \ol{\tau}_\partial \wedge \min_{y \in \wh{\rmA} \setminus \{x\}}
\tau_{\partial \rmB(y;2r_n+1)} \quad \text{and} \quad 
\ck{\tau}_{\rmB(x;r_n)} < \ol{\tau}_\partial \Big)
= \frac{2\pi}{\deg(\partial) n}\Big(1+o(n^{-\eta_0/4})\Big) \,,
\end{equation}
with $o(n^{-\eta_0/4})n^{\eta_0/4} \to 0$ uniformly over all such sets $\rmA$ and their skeletons $\wh{\rmA}$.
\end{lem}

\begin{proof}By a slight modification of the estimate in~\eqref{eq:form2.1a} to account for the fact that $x$ may not belong to $\rmD_n^\circ$ (but it nonetheless satisfies $\log \rmd (x, \rmD_n^c) \geq n-2\log n + o(1))$, we have that 
\begin{equation}\label{eq:pbound1}
\bbP_\partial \Big(\tau_{\partial \rmB(x;2r_n+1)} < \ol{\tau}_\partial\Big) = \frac{2\pi}{\deg(\partial) n} (1+o(n^{-\eta_0/4})).
\end{equation}
The probability in the statement of the lemma is clearly smaller than 
 the probability in~\eqref{eq:pbound1}. The difference between the two is at most
\begin{multline}\label{eq:multline}
\sum_{y \in \wh{\rmA} \setminus \{x\}} \bbP_\partial \big( \tau_{\partial \rmB(y;2 r_n + 1)} < \ol{\tau}_\partial \big)
\bbP_y \big( \tau_{\partial \rmB(x;2 r_n + 1)} < \ol{\tau}_\partial \big) \\
+ \bbP_\partial \big(\tau_{\partial \rmB(x;2 r_n + 1)} < \ol{\tau}_\partial \big)
\max_{y \in \partial_i \rmB(x;2 r_n+1)} \bbP_{y} \big(\tau_\partial < \ck{\tau}_{\rmB(x;r_n)}\big)\,.
\end{multline}
By (a slight modification of) Lemma~\ref{lem:g} and Lemma~\ref{l:8.2}, the first term in~\eqref{eq:multline} above is at most $C \chi_n(\rmA) \deg(\partial)^{-1} n^{-2} r_n$, while the second term is at most $C \deg(\partial)^{-1}n^{-1} n^{-\eta_0/3}$. Both terms are of smaller order than the right-hand side of~\eqref{eq:pbound1}, in light of the condition on $\chi_n(\rmA)$.
\end{proof}

We are now ready for:

\begin{proof}[Proof of Theorem~\ref{t:2b}]
Starting with the first statement of the theorem, we write $\chi_n:=\chi_n(\rmA)$ and enumerate the points in $\wh{\rmA}$ as $x_1, \dots, x_{\chi_n}$. Throughout the following, we will assume that $n$ is large enough so that all closed balls $\ol{\rmB(x_k,2r_n+1)}$ for $k=1,\dots,\chi_n$ are disjoint and contained in $\rmD_n$. In addition, by monotonicity of $A \mapsto \ck{\tau}_\rmA$ with respect to set inclusion, we may and will assume without loss of generality that $\chi_n < n^{1/2+\eta_0/4}$. We then decompose the trajectory of the random walk up to $\partial$-time $t_n^B + sn$ into a number $E_n$ of excursions. Clearly, $E_n$ has a Poissonian law with rate $\lambda_n := (2\pi)^{-1} \deg(\partial)(t_n^B + sn)$, while the law of the walk in each excursion is that of a random walk starting from $\partial$ and killed upon returning to~$\partial$. For any $\rmV \subseteq\rmD_n$, define $\hat{\tau}_{\rmV}$ as the first $\partial$-time all the points in $\rmV$ have been visited if, in every excursion, after hitting one of the boundaries $\partial \rmB(x_k;2r_n+1)$ with $k \in [1,\chi_n]$ we ignore visits to any points in $\rmA\setminus \rmB(x_k;2r_n+1)$. Let also $U_i$ be the set of indices $k$ such that the random walk hits $\partial \rmB(x_k;2r_n+1)$ during excursion $i$ and $V_i$ be the set of indices $k$ such that during excursion $i$ the random walk hits $\partial \rmB(x_k;2r_n+1)$ before any $ \partial \rmB(x_j;2r_n+1)$ for $j \neq k$ and then visits all points in $\rmB(x_k;r_n) \cap \rmA$. Observe that $U_i$ may be empty and that $V_i$ is either empty or contains one index. 

The probability in the first statement of the lemma can be reformulated as $\bbP \big(\ck{\tau}_{\rmA} \leq t_n^B + sn \big)$.
Clearly $\hat{\tau}_{\rmA} \geq \ck{\tau}_{\rmA}$ so that
\begin{equation}
\big \{\hat{\tau}_{\rmA} \leq t_n^{B} + sn \big \} \subseteq \big \{\ck{\tau}_{\rmA} \leq t_n^{B} + sn \big \} \,. 
\end{equation}
On the other hand, we have that 
\begin{equation}
\label{e:8.14}
\{\ck{\tau}_{\rmA} \leq t_n^{\rmB} + sn \big \}  \setminus \big\{\hat{\tau}_{\rmA} \leq t_n^{\rmB} + sn\} \subseteq \{ \hat{\tau}_{\rmA} \neq\ck{\tau}_{\rmA}\} 
\subseteq F  \cup \Big(\bigcup_{m \geq 2} G_m \Big) \,,
\end{equation}
where
\begin{equation}
F:=\Bigl\{E_n < \ul{\lambda}_n \text{ or } E_n > \ol{\lambda}_n \Big\} \,,
\end{equation}
\begin{equation}
G_m:=\Big\{\exists i\le \ol{\lambda}_n \colon |U_i|=m \text{ and }
U_i\smallsetminus \bigcup \Big(V_j :\: j\le \ul{\lambda}_n, \, j\ne i \Big) \ne\emptyset\Big\} \,
\end{equation}
and we set $\ul{\lambda}_n := (1-\eta) \lambda_n$ and $\ol{\lambda}_n := (1+\eta) \lambda_n$ for some $\eta > 0$ to be specified later.

By Chernov's bound we have $\bbP (F) \leq \rme^{-c \lambda_n} \leq \rme^{-\rme^{\ol{c} n}}$ for some constants $c,\ol{c} > 0$ depending on $\eta$. At the same time,
by the union bound, Lemma~\ref{lem:g} and Lemma~\ref{l:8.3}, for all $n$ large enough and any $m \geq 2$ we have
\begin{multline}
\bbP \big(G_m\big)\leq 
C \ol{\lambda}_n \frac{\chi_n}{\deg(\partial) n} \Big(\frac{2\chi_n r_n}{n}\Big)^{m-1} m 
\bigg(1-\frac{2\pi}{\deg(\partial)n}(1+o(n^{-\eta_0/2})) \bigg)^{\ul{\lambda}_n-1} 
\leq C_s m 2^{m} n^{-\frac{\eta_0}{2} + \eta} \,,
\end{multline}
which tends to zero as $n \to \infty$ if $\eta$ is small enough. 
On the other hand, for any fixed $m_0 \geq 2$,
\begin{equation}
\bbP \Big(\bigcup_{m \geq m_0} G_m \Big)\le 
C \ol{\lambda}_n \frac{\chi_n}{\deg(\partial) n} \Big(\frac{2\chi_n r_n}{n}\Big)^{m_0-1}  
\leq C_s 2^{m_0} n^{1/2 - \frac{3}{4}\eta_0 (m_0 - 2)} \,,
\end{equation}
which tends to zero as $n \to \infty$ as soon as $m_0$ is taken large enough. Hence, by the union bound, this shows that the probability of the leftmost event in~\eqref{e:8.14} tends to $0$ with $n$, uniformly in~all $(r_n,n-r_n)$-clustered sets $\rmA$ with $\chi_n(\rmA) < n^{1/2+\eta_0/4}$.

It therefore remains to estimate for all $n$ large enough the probability of 
\begin{equation}
\label{e:8.21}
\big \{\hat{\tau}_{\rmA} \leq t_n^{\rmB} + sn\} = \bigcap_{k=1}^{\chi_n} \big\{\hat{\tau}_{\rmA \cap \rmB(x_k; r_n)} \leq t_n^{\rmB} + sn \big\} \,.
\end{equation}
To this end we notice that, by definition, the events
\begin{equation}
\Big\{\tau_{\partial \rmB(x_k;2r_n+1)} < \ol{\tau}_\partial \wedge \min_{y \in \wh{\rmA} \setminus \{x_k\}}
\tau_{\partial \rmB(y;2r_n+1)} \quad \text{and} \quad 
\ck{\tau}_{\rmB(x;r_n)} < \ol{\tau}_\partial \Big\}	\,,
\end{equation}
are disjoint for $k=1, \dots, \chi_n$. Hence, by standard Poisson thinning, the events in the intersection in~\eqref{e:8.21} are all independent and, in light of Lemma~\ref{l:8.3}, have probabilities
\begin{equation}
1 - \exp \Big\{-\lambda_n \frac{2\pi}{\deg(\partial)n}(1+o(n^{-\eta_0/2})) \Big\} \!= \!1-\frac{1}{\sqrt{n}}\rme^{-s} \big(1 + o_s(n^{-\eta_0/4})\big) 
\!=\! \exp \Big(-\frac{\rme^{-s}}{\sqrt{n}}\big(1+o_s(n^{-\eta_0/4})\big)\Big).
\end{equation}
Taking the product over all $k$ and recalling that $\chi_n < n^{1/2+\eta_0/4}$ now yields the desired statement.

Finally, the second statement follows easily from Lemma~\ref{lem:nhprob} using the union bound.
\end{proof}

\section{Proof of Proposition~\ref{p:2.2} (Time Reparametrization)}
\label{s:8}
In this section we prove Proposition~\ref{p:2.2}. For the proof we will need estimates on the first and second moments of the time length of an excursion away from $\partial$, namely the (real) time the walk spends after leaving and before returning to $\partial$. Recalling that $\tau_{\rmA}$ and $\ol{\tau}_{\rmA}$ are the hitting-time-of and return-time-to $\rmA \subseteq \wh{\rmD}_n$ of our continuous random walk on $\wh{\rmD}_n$, the excursion time length can be written as 
\begin{equation}
\label{e:2.12}
	\theta := \ol{\tau}_{\partial} - \tau_{\rmD_n} \,,
\end{equation}
assuming, of course, that $X_0 = \partial$.
\begin{lem}
\label{l:2.3}
Let $\theta$ be defined as in~\eqref{e:2.12} and suppose that $X_0 = \partial$. Then,
\begin{equation}
\label{e:2.13a}
\bbE \theta = 2\pi\frac{|\rmD_n|}{\deg (\partial)} \quad, \qquad
\bbE \theta^2 \leq C \frac{|\rmD_n|^2}{\deg(\partial)} \,,
\end{equation}
for some $C < \infty$.
\end{lem}
\begin{proof}
Enumerate the excursion times of the walk by $\theta_1, \theta_2, \dots$ and the successive times spent at $\partial$ before each excursion by $\cE_1, \cE_2, \dots$. Let also $N_t = \min \{l \geq 0 :\: \sum_{k=1}^{l+1} \cE_k > t\}$, i.e. $N_t$ is the number of excursions away from $\partial$ until $\partial$-time $t$. By definition we have
\begin{equation}
S_t := \sum_{x \in \rmD_n} L_t(x) = \sum_{k=1}^{N_t} \theta_k \,.
\end{equation}

Since $\cE_1, \cE_2, \dots$ are i.i.d. Exponentials with rate $\frac{1}{2\pi}\deg(\partial)$, the process $(N_t)_{t \geq 0}$ is Poisson with the same rate. Moreover, since $\theta_1, \theta_2, \dots$ are also i.i.d. with law as that of $\theta$ under $\bbP$, and they are also independent of the holding times $\mathcal{E}_k$, the process $(S_t)_{t \geq 0}$ is Compound Poisson and thus by standard arguments
\begin{equation}
\label{e:2.15}
\bbE S_t = \bbE N_t \, \bbE \theta  = \frac{1}{2\pi}\deg (\partial)  t\, \bbE \theta 
\quad , \qquad
\Var S_t = \bbE N_t \, \bbE \theta^2  = \frac{1}{2\pi}\deg (\partial) t\, \bbE \theta^2 \,.
\end{equation}

On the other hand, from Theorem~\ref{t:103.1} we get 
\begin{equation}
\bbE L_t(x) + \bbE h_n^2(x) = \bbE h'_n(x)^2 + 2\sqrt{t} \bbE h'_n(x) + t 
\end{equation}
and
\begin{multline}
\Cov \big(L_t(x), L_t(y)\big) 
+ \Cov \big(h_n(x)^2, h_n^2(y)\big)  
= \Cov \big(h'_n(x)^2, h'_n(y)^2\big) \\ 
+ 2\sqrt{t} \Cov \big(h'_n(x)^2, h'_n(y)\big) 
+ 2\sqrt{t} \Cov \big(h'_n(x), h'_n(y)^2\big) 
+ 4t \Cov \big(h'_n(x), h'_n(y)\big) \,.
\end{multline}
Since multinomials in $h'$ whose degree is odd have zero mean by the symmetry under sign change of $h'$, we get
\begin{equation}
	\bbE L_t(x) = t
	\quad , \qquad
	\Cov \big(L_t(x), L_t(y) \big)) = 2t G_{\rmD_n}(x,y) \,.
\end{equation}
Consequently,
\begin{equation}
\label{e:2.18}
\bbE S_t = |\rmD_n| t
\quad, \qquad
\Var S_t = 2t \sum_{x,y \in \rmD_n} G_{\rmD_n}(x,y) \leq C t |\rmD_n|^2 \,, 
\end{equation} where, for the last inequality, we can use the bound
\begin{equation}
\sum_{y \in \rmD_n} G_{\rmD_n}(x,y) = \bbE_x \tau_\partial \leq C |\rmD_n| \,
\end{equation} which follows from \cite[Proposition~6.2.6]{LL}.
By equating~\eqref{e:2.18} with~\eqref{e:2.15}, we now recover~\eqref{e:2.13a}.
\end{proof}

The proof of Proposition~\ref{p:2.2} is straightforward, given Lemma~\ref{l:2.3}.
\begin{proof}[Proof of Proposition~\ref{p:2.2}]
As in the proof of Lemma~\ref{l:2.3}, we may write
\begin{equation}
S_t := \sum_{x \in \rmD_n} L_t(x) = \sum_{k=1}^{N_t} \theta_k \,,
\end{equation}
where $N_t$ is a Poisson process with rate $\frac{1}{2\pi}\deg (\partial)$ and $\theta_k$ for $k \geq 1$ are i.i.d. having the law of the time length of an excursion away from $\partial$. It then follows by definition that
\begin{equation}
\bfL_t^{-1}(\partial) = t + S_t \,,
\end{equation}
and also that 
\begin{equation}
\label{e:2.23}
t+S_t^\circ \leq \bfL_{t-}^{-1}(\partial) \leq t + S_t\,,
\end{equation}
where $S_t^\circ := \sum_{k=1}^{N_t-1} \theta_k$. Henceforth we denote the quantity on the leftmost-hand side in~\eqref{e:2.23} by $\wt{\bfL}_t^{-1}(\partial)$.

In light of Lemma~\ref{l:2.3} and standard arguments, whenever $t \leq |\wh{\rmD}_n|$ we have that 
\begin{equation}\label{e:est1}
\bbE\, \bfL_t^{-1}(\partial) = t\big(|\rmD_n|+1) = |\wh{\rmD}_n|\big(t+O(1)\big)
\end{equation}
\begin{equation}\label{e:est2}
\bbE\, \wt{\bfL}_{t-}^{-1}(\partial) = t\Big(|\rmD_n| + 1 - 2\pi \frac{|\rmD_n|}{|\deg(\partial)|}\Big) = |\wh{\rmD}_n|\big(t+O(\sqrt{t})\big)
\end{equation}
and
\begin{equation}\label{e:est3}
\Var\, \wt{\bfL}_t^{-1} (\partial) \leq
\Var\, \bfL_t^{-1}(\partial) =\frac{1}{2\pi}t |\deg(\partial)| \bbE \theta^2 \leq 
 C t|\wh{\rmD}_n|^2 
\ , \quad
\end{equation}
Above $C > 0$ and $\theta$ has the law of the time length of one excursion away form $\partial$.

It then follows by Chebyshev's inequality and since $\wt{\bfL}_{t}^{-1}(\partial) \leq \bfL_{t-}^{-1}(\partial) \leq \bfL_{t}^{-1}(\partial)$ that the collection:
\begin{equation}
\frac{\bfL^{-1}_t(\partial)/|\wh{\rmD}_n| - t}{\sqrt{t}}
\quad, \qquad
\frac{\bfL^{-1}_{t-}(\partial)/|\wh{\rmD}_n| - t}{\sqrt{t}}
\end{equation}
for all $n \geq 0$ and $0 \leq t \leq |\wh{\rmD}_n|$ is tight. Using that $|\sqrt{a+b}-\sqrt{a}| \leq |b|/\sqrt{a}$ for all $a \geq 0$ and $b \geq -a$, this readily implies the statement of the proposition.
\end{proof}

\appendix
\section{Appendix: Proofs of preliminary statements} 
\label{s:A}
This appendix includes proofs for the various preliminary statements from Section~\ref{s:3}. 

\subsection{Random walk preliminaries} \label{s:A1}

Throughout this entire section, $Z=(Z_j)_{j \in \N_0}$ will denote a simple symmetric random walk on~$\Z^2$ to be used as an auxiliary process during the proofs. To keep matters simple, we will adopt the same notation used for the original walk $X$. That is, given any $x \in \Z^2$, we will write $\bbP_x$ and $\bbE_x$ for the underlying probability and expectation whenever $Z_0 \equiv x$. Similarly, given any $A \subseteq \Z^2$, we will write $\tau_A$ for the hitting time of the set $A$ by $Z$, 
\begin{equation}
\tau_A:=\inf \{ j \geq 0 : Z_j \in A\}\,,
\end{equation} and $\ol{\tau}_A$ for the return time to $x$ by $Z$, 
\begin{equation}\label{eq:defret}
	\ol{\tau}_A:=\inf\{ j \geq 1 : Z_j \in A\}\,.
\end{equation} However, we shall prefer to write $\tau_x$ and $\ol{\tau}_x$ whenever $A=\{x\}$ is a singleton. Finally, for $A \subseteq \bbZ^2$ nonempty, we write $\partial A$ and $\partial_i A$ respectively for the \textit{outer} and \textit{inner boundaries} of $A$,
\begin{equation}
	\partial A := \{ y \in \Z^2 \setminus A : \exists x \in A \text{ s.t. }\| y-x\|=1\} \quad,\quad \partial_i A := \{ x \in A :  \exists y \notin A \text{ s.t. }\| y-x\|=1\}\,.
\end{equation}

\subsubsection{Discrete Harmonic Analysis}
\begin{proof}[Proof of Lemma~\ref{l:103.2}]
	The first statement is a direct consequence of~\eqref{e:rel} and~\eqref{e:3.2} and
	\cite[Lemma~3.2]{BLS} proves~\eqref{e:3.19.2}. To show the remaining estimates, write
	\begin{equation}
		\big|G_{\cD_n}(x,y) - G_{\cD_n}(x,x) - a(x,y)\big| \leq \sum_{z \in \partial \cD_n} \big|a(z-y)-a(z-x)\big| \Pi_{\cD_n}(x,z) 
	\end{equation}
	and use the bound 
	\begin{equation}
		\big|a(z-y)-a(z-x)\big| \leq \frac{\|x-y\|}{\rmd(\{x,y\}, \cD_n^\rmc)} + C(\rmd(\{x,y\}, (\cD_n)^\rmc))^{-2}.
	\end{equation} Since $|\log \rmd(x,(\cD_n)^\rmc) - \log \rmd(\{x,y\}, (\cD_n)^\rmc)| \leq \frac{\|x-y\|}{\rmd(\{x,y\}, (\cD_n)^\rmc)}$ as well, the remaining estimates now follow from~\eqref{e:3.19.2} and the triangle inequality.
\end{proof}

\begin{proof}[Proof of Lemma~\ref{lem:g}] By conditioning on $X_{\tau_{\rmD_n}}$, the first location visited by the random walk $X$ once it jumps away from the boundary, we obtain
	\begin{equation}\label{eq:nhprob}
	\mathbb{P}_\partial ( \tau_x < \ol{\tau}_\partial) =\frac{1}{\deg(\partial)} \sum_{z \in \partial_i \rmD_n}\bbP_z(\tau_x < \ol{\tau}_\partial)=\frac{1}{\deg(\partial)}\sum_{z \in \partial_i \rmD_n} \bbP_z(Z_{\tau_{\partial \rmD_n \cup \{x\}}} = x)\,,
	\end{equation} with $Z=(Z_j)_{j \in \bbN_0}$ the auxiliary random walk on $\bbZ^2$ introduced in the beginning of the section.
		
	Now, given $z \in \partial_i \rmD_n$, by reverting the paths of the random walk $Z$ going from $z$ to $x$ without exiting $\rmD_n$ and then adding to each reversed path a final step connecting them to $\partial \rmD_n$, we obtain
	\begin{equation}\label{eq:reveq}
	\bbP_z(Z_{\tau_{\partial \rmD_n \cup \{x\}}} = x)=4\bbP_x( Z_{\tau_{\partial \rmD_n}-1} = z\,,\,\tau_{\partial \rmD_n}<\tau_x)\,,
    \end{equation} where the factor $4$ accounts for this final step of the reversed path connecting $z$ to an element in $\partial \rmD_n$ (notice that there might be more than one such element in $\partial \rmD_n$ to which we can connect~$z$, and that this is accounted for in \eqref{eq:reveq}). Plugging this into \eqref{eq:nhprob} yields
\begin{equation}
	\mathbb{P}_\partial ( \tau_x < \ol{\tau}_\partial)= \frac{4}{\deg(\partial)}\bbP_x(Z_{\tau_{\partial \rmD_n}-1} \in \partial_i \rmD_n\,,\,\tau_{\partial \rmD_n}<\tau_x)= \frac{4}{\deg(\partial)}\bbP_x(\tau_{\partial \rmD_n}<\tau_x)\,.
\end{equation} But a straightforward Markov chain calculation shows that, with our choice of transition rates, we have
\begin{equation}\label{eq:form4}
	G_{\rmD_n}(x,x)=\frac{\pi}{2}\frac{1}{\bbP_x(\tau_{\partial \rmD_n}<\tau_x)}\,,
\end{equation} so \eqref{e:h1} now follows.  

To check \eqref{eq:form2.1a}, we notice that, by the strong Markov property of the walk, we have
\begin{equation}
	\bbP_\partial(\tau_x < \ol{\tau}_\partial)= \E_\partial(1_{\{\tau_{\rmB(x;r)} < \ol{\tau}_\partial\}} \bbP_{X_{\tau_{\rmB(x;r)}}}(\tau_x < \tau_\partial))
\end{equation} and also, for any $y \in \partial_i \rmB(x;r)$, 
\begin{equation}\label{e:repgreen}
	\bbP_y( \tau_x < \tau_\partial) = \frac{G_{\rmD_n}(y,x)}{G_{\rmD_n}(x,x)}.
\end{equation} Thus, since by \eqref{e:3.2}, \eqref{e:rel} and our choice of $r$ we have that, uniformly over $y \in \partial_i \rmB(x;r)$,
\begin{equation}\label{eq:repgreen2}
\begin{split}
	G_{\rmD_n}(y,x)-G_{\rmD_n}(x,x) &= \sum_{z \in \partial \rmD_n} [a(z-y)-a(x-y)-a(z-x)]\Pi_{\rmD_n}(x,z)\\
	&=-r-\gamma^*+O(2^{-(n-2\log n -r)})+O(2^{-r})\,,
\end{split}
\end{equation} by combining the last three displays we conclude that 
\begin{equation}
\bbP_\partial(\tau_x < \ol{\tau}_\partial) =  \frac{G_{\rmD_n}(x,x)-r-\gamma^*+O(2^{-(r\wedge (n-2\log n - r))})}{G_{\rmD_n}(x,x)}\bbP_\partial (\tau_{\rmB(x;r)} < \ol{\tau}_\partial)\,,
\end{equation} from where \eqref{eq:form2.1a} now follows upon recalling \eqref{e:h1} and observing that $G_{\rmD_n}(x,x)-r-\gamma^* \geq 1$ for all $n$ large enough by~\eqref{e:3.19.2} and our choice of $x$ and $r$.

Finally, \eqref{eq:form2.1b} follows by a very similar argument, replacing $\partial$ by $y \in \rmB(x;n-3\log n)\setminus \rmB(x;r)$ and using the analogues of \eqref{e:repgreen}--\eqref{eq:repgreen2} for this $y$ instead of \eqref{e:h1}. We omit the details.
\end{proof}

\subsubsection{Local time and downcrossings}

Our purpose now is to prove the results on the local time and downcrossings from Section~\ref{ss:locdownprelim}.  As a matter of fact, to improve the clarity of the arguments during the proofs, we shall prove a slightly more general version of the statements presented in the aforementioned section. Indeed, recall from \eqref{e:2.2} that 
\begin{equation}\label{A:12a}
N_t(x;k):=	N_t(x;k+c_1k^\gamma,k+c_2k^\gamma) \quad,\quad \wh{N}_t(x;k):=(c_2-c_1)k^\gamma N_t(x;k)\,,
\end{equation}where the values of $c_1,c_2$ were respectively fixed as $c_1=\frac{1}{2}$ and $c_2=1$. Our goal in this section will be to prove the results in Section~\ref{ss:locdownprelim} but for the general case in which $c_2 > c_1 > 0$ are any pair of fixed constants, not necessarily $\frac{1}{2}$ and $1$. Therefore, with a slight abuse of notation, throughout this subsection we shall assume that $N_t(x;k)$ and $\wh{N}_t(x;k)$ are defined as in \eqref{A:12a} for some arbitrary $c_2>c_1>0$.

For convenience, in the following for any $x \in \bbZ^2$ and $r>0$ we will abbreviate $\rmB_r(x):=\rmB(x;r)$ as well as $\rmB_k^-(x):=\rmB_{k+c_1k^\gamma}(x)$ and $\rmB_k^+(x):=\rmB_{k+c_2k^\gamma}(x)$ for any $k \geq 0$, often suppressing $x$ from the notation and writing instead $\rmB_r$ or $\rmB_k^\pm$ whenever $x$ is fixed and clear from the context. 
In addition, for $x$ and $k$ as above, excursions of the walk from $\partial_i \rmB_k^-(x)$ to $\partial \rmB_k^+(x)$ will be called $(x;k)$-\textit{excursions} and excursions from $\partial \rmB_k^+(x)$ to~$\partial_i \rmB_k^-(x)$ will be called $(x;k)$-\textit{downcrossings}. Furthermore, given any pair $y \in \partial_i \rmB_k^-(x)$ and $y' \in \partial \rmB_k^+(x)$, we~shall write $\bbP_{y,y'}^{(x;k)}$ for the law of any $(x;k)$-excursion conditioned to start at $y$ and end at $y'$, as well as $\E_{y,y'}^{(x;k)}$ to denote expectation with respect to $\bbP_{y,y'}^{(x;k)}$. As before, we will often suppress the pair $(x;k)$ from the notation and write only $\bbP_{y,y'}$ or $\E_{y,y'}$ whenever no ambiguity arises from doing so. 

In order to prove the (more general form of the) results from Section~\ref{ss:locdownprelim}, we will need to rely on several auxiliary results, which we present next. 
The first of these is a series of standard ``gambler's ruin''-type estimates for $2$-dimensional random walks. Recall the auxiliary walk $Z$ presented at the beginning of Section \ref{s:A1}.

\begin{lem} \label{lem:ll1} The random walk $Z=(Z_j)_{j \in \N}$ satisfies the following estimates:
	\begin{enumerate}
		\item [a)]  Given $x \in \Z^2$ and $0 < r < R$, for any $y \in \rmB(x;R)\setminus \rmB(x;r)$ we have that
		\begin{equation}
			\bbP_y( Z_{\tau_{\rmB(x;R)^\rmc \cup \rmB(x;r)}} \in \rmB(x;r)) = \frac{R- \log \|x-y\| +O(\rme^{-r})}{R-r}.  
		\end{equation}
		\item [b)] Given $x \in \Z^2$ and $R>0$, for any $y \in \rmB(x;R) \setminus \{x\}$ we have that
		\begin{equation}\label{eq:LLest2}
			\bbP_y( Z_{\tau_{\rmB(x;R)^\rmc \cup \{x\}}} = x) = \left(\frac{R-\log\|x-y\|+O(\|x-y\|^{-1})}{R}\right)(1+O(R^{-1})).  
		\end{equation} 
	\end{enumerate}
\end{lem}

\begin{proof} This follows immediately from Proposition~6.4.1 and Proposition~6.4.3 in \cite{LL}. 
\end{proof}

The second auxiliary result concerns asymptotics for the Poisson kernel $\Pi_{\rmB(x;r)}$ for $r$ large.

\begin{lem}\label{lem:kernel} Given $x \in \bbZ^2$, for all $k$ large enough (depending only on $c_1$, $c_2$ and $\gamma$) we have
\begin{equation}\label{eq:qpk0}
	\frac{\Pi_{\rmB(x;k+c_2k^\gamma)}(y,z)}{ \Pi_{\rmB(x;k+c_2k^\gamma)}(x,z)}=1+O(\rme^{-(c_2-c_1)k^\gamma}) 
\end{equation} uniformly over all $y \in \rmB(x;k+c_1k^\gamma)$ and $z \in \partial \rmB(x;k+c_2k^\gamma)$. 
\end{lem}

\begin{proof} By \cite[Theorem~6.3.8]{LL} there exists $C> 0$ such that we have that
\begin{equation}\label{eq:qpk1}
1-C\rme^{-(k+c_2k^\gamma)} \leq \frac{\Pi_{\rmB(x;k+c_2k^\gamma)}(v,z)}{\Pi_{\rmB(x;k+c_2k^\gamma)}(u,z)} \leq 1+C\rme^{-(k+c_2k^\gamma)}
\end{equation} uniformly over all pairs of nearest neighbors $u,v \in \rmB(x;k+c_1k^\gamma)$ and all $z \in \partial \rmB(x;k+c_2k^\gamma)$. Indeed, this follows from applying \cite[Equation~(6.19)]{LL} to the map $f_{u,z}(\cdot):=\Pi_{\rmB(x;k+c_2k^\gamma)}(\cdot+u,z)$ which is harmonic on $\rmB(0;k+c_2k^\gamma-1)$ if $k$ is taken sufficiently large (depending only on $c_1$, $c_2$ and $\gamma$) so as to have that $\rmB(u;k+c_2k^\gamma-1) \subseteq \rmB(x;k+c_2k^\gamma)$. 
Hence, since any $y \in \rmB(x;k+c_2k^\gamma)$ can be joined with~$x$ by a path having at most $\rme^{k+c_1k^\gamma}$ steps, by iterating \eqref{eq:qpk1} over all pairs of consecutive neighbors on such path, a straightforward computation then yields \eqref{eq:qpk0} for all $k$ large enough.
\end{proof}

The third auxiliary result contains tail estimates for Binomial/Poisson random sums of i.i.d. Geometric/Exponential random variables.  

\begin{lem} 
For $p,q \in (0,1]$, let $(A_j)_{j \in \N}$ be a sequence of $\cB(p)$ random variables and $(G_j)_{j \in \N}$ a sequence of $\cG(q)$ random variables, all mutually independent. Then, for any $n \geq 1$ and $v \leq \frac{np}{q}$,
\begin{equation} \label{A:bin-geo}
	\bbP\left( \sum_{j=1}^n A_jG_j \leq v \right) \leq  \rme^{-(\sqrt{np}-\sqrt{vq})^2}.
\end{equation}  
Similarly, given $p \in (0,1]$, $\lambda > 0$, let $(A_j)_{j \in \N}$ be a sequence of $\cB(p)$ random variables and~$(E_j)_{j \in \N}$ a sequence of $\cE(\lambda)$ random variables, all mutually independent. Then, for any $n \geq 1$ and $v \leq \frac{np}{\lambda}$, 
\begin{equation} \label{A:bin-exp}
		\bbP\left( \sum_{j=1}^n A_jE_j \leq v \right) \leq \rme^{-(\sqrt{np}-\sqrt{v\lambda})^2}.
\end{equation} In addition, given $q \in (0,1]$ and $\mu > 0$, let $(G_j)_{j \in \N}$ a sequence of $\cG(q)$ random variables and $N$ be a $\mathcal{P}(\mu)$ random variable, all mutually independent. Then, for any $v \leq \frac{\mu}{q}$,
\begin{equation}\label{A:poi-geo}
	\bbP\left( \sum_{j=1}^N G_j \leq v \right) \leq \rme^{-(\sqrt{\mu}-\sqrt{vq})^2}.
\end{equation}
Finally, given $\lambda, \mu > 0$, let $(E_j)_{j \in \N}$ be a sequence of $\cE(\lambda)$ random variables and let $N$ be a $\mathcal{P}(\mu)$ random variable, all mutually independent. Then, for any $v \leq \frac{\mu}{\lambda}$, 
\begin{equation} \label{A:poi-exp}
		\bbP\left( \sum_{j=1}^N E_j \leq v \right) \leq \rme^{-(\sqrt{\mu}-\sqrt{v\lambda})^2}.
\end{equation}
\end{lem}

\begin{rem}\label{rem:rs} Observe that, if $(A_j)_{j \in \N}$ is a sequence of $\cB(p)$ random variables and $(X_j)_{j \in \N}$ is a sequence of identically distributed random variables, all mutually independent, then 
\begin{equation}
\sum_{j=1}^n A_j X_j \overset{d}{=}\sum_{j=1}^{A_1+\dots+A_n} X_j,
\end{equation} where $A_1+\dots+A_n \sim \cB(n,p)$. Thus, the sums in both \eqref{A:bin-geo} and \eqref{A:bin-exp} can truly be regarded as Binomial random sums of Geometric and Exponential i.i.d. random variables, respectively.
\end{rem}

\begin{proof} A proof of \eqref{A:bin-geo} can be found in \cite[Lemma 4.6]{BK}, so we will only prove the other bounds. We start with \eqref{A:bin-exp}. To this end, let us fix $v \leq \frac{np}{\lambda}$ and notice that by Remark~\ref{rem:rs} we have 
\begin{equation}
	\bbP\left( \sum_{j=1}^n A_j E_j \leq v\right) = \bbP\left( \sum_{j=1}^{A_1+\dots+A_n} E_j \leq v\right). 
\end{equation} Viewing the $(E_j)_{j \in \N}$ as the sequence of interarrival times of a Poisson process of parameter~$\lambda$, we can further write
\begin{equation}\label{eq:view}
	\bbP\left( \sum_{j=1}^{A_1+\dots+A_n} E_j \leq v\right)=\bbP(A_1+\dots+A_n \leq Y), 
\end{equation} where $Y \sim \cP(\lambda v)$ is independent of the $A_j$. Now, by the exponential Tchebychev inequality and independence, for any $\theta \geq 0$ we obtain
\begin{align*}
\bbP(A_1+\dots+A_n \leq Y) &\leq \E(\rme^{\theta Y})\E(\rme^{-\theta(A_1+\dots+A_n)}) \\
&=\rme^{\lambda v(\rme^\theta -1)} (1+p(\rme^{-\theta}-1))^n \leq \exp\{\lambda v(\rme^\theta-1) + np(\rme^{-\theta}-1)\},	
\end{align*} where in the line we have used that $A_1+\dots+A_n \sim \text{Binomial}(n,p)$ and that $1+x \leq \rme^x$ for all~$x$. Taking $\theta:=\frac{1}{2}\log \frac{np}{\lambda v} \geq 0$ now immediately yields the bound in \eqref{A:bin-exp}.

Finally, to check~\eqref{A:poi-geo}--\eqref{A:poi-exp}, on the hand we notice that by virtue of independence,
\begin{equation}
	\bbP\left( \sum_{j=1}^N G_j \leq v \right) = \bbP( N \leq I_1+\dots+I_{[v]}),
\end{equation} with $(I_j)_{j \in \bbN}$ an i.i.d. sequence of $\cB(q)$ random variables independent of $N$, and on the other~hand we observe that by the same argument yielding~\eqref{eq:view}, 
\begin{equation}
	\bbP\left( \sum_{j=1}^N E_j \leq v \right) = \bbP(N \leq Y), 
\end{equation} where $Y \sim \cP(\lambda v)$ is independent of $N$. In light of this, the proofs of both bounds now continue as for \eqref{A:bin-exp}, we omit the~details.
\end{proof}

Finally, the last auxiliary result is a weak tail estimate for sums of random variables which are independent but not necessarily identically distributed.

\begin{lem}\label{lem:A.5} Let $(\xi_j)_{j=1}^m$ be independent nonnegative random variables such that 
\begin{equation}
\sup_{1 \leq j \leq m}\E(\rme^{\theta_0\xi_j}) \leq \alpha_1 \qquad \text{ and }\qquad \inf_{1\leq j\leq m}\E(\xi_j) \geq \alpha_2
\end{equation} for some constants $\theta_0,\alpha_1,\alpha_2 \in (0,\infty)$. Then there exist constants $c,\rho>0$ (depending only on $\theta_0,\alpha_1$ and $\alpha_2$) such that, for all $0 \leq z \leq \rho m$,
\begin{equation}\label{eq:boundind}
	\bbP\left(\sum_{j=1}^m \xi_j > \sum_{j=1}^m \E(\xi_j)+ z\right) \leq \rme^{-c\frac{z^2}{m}}
\end{equation} and, for all $z \geq 0$, 
\begin{equation}\label{eq:boundind_II}
	\bbP\left(\sum_{j=1}^m \xi_j < \sum_{j=1}^m \E(\xi_j)- z\right) \leq \rme^{-c\frac{z^2}{m}}
\end{equation}
\end{lem}

\begin{proof} We start by showing \eqref{eq:boundind}. Let us abbreviate $S_1:=\sum_{j=1}^m \E(\xi_j)$ and $S_2:=\sum_{j=1}^m \E(\xi_j^2)$ (notice that both quantities are well-defined and finite under our current hypotheses). Then, by the exponential Tchebychev inequality, the independence of the $\xi_j$ yields that for all $\theta \geq 0$,
\begin{equation}
	\bbP\left(\sum_{j=1}^m \xi_j > \sum_{j=1}^m \E(\xi_j) +z\right)=\bbP\left( \sum_{j=1}^m \xi_j > S_1 + z\right) \leq \rme^{-\theta(S_1 +z)} \prod_{j=1}^m \E(\rme^{\theta \xi_j}).
\end{equation}
Since $\sup_{1\leq j\leq m}\E(\rme^{\theta_0\xi_j}) \leq \alpha_1$, for all $\theta < \frac{1}{2}\theta_0$ and $j=1,\dots,m$ we have the expansion
\begin{equation}
\E(\rme^{\theta \xi_j}) = 1 + \E(\xi_j)\theta + \E(\xi_j^2)\frac{\theta^2}{2} + O_{\theta_0,\alpha_1}(\theta^3),
\end{equation} where the error term $O_{\theta_0,\alpha_1}(\theta^3)$ can be bounded uniformly in $j$ by $C_{\theta_0,\alpha_1}\theta^3$, for some $C_{\theta_0,\alpha_1}>0$ depending only on $\theta_0$ and $\alpha_1$. Using that $1+x \leq \rme^x$ for all $x$, we conclude that for all $0 \leq \theta < \frac{1}{2}\theta_0$
\begin{equation}
	\bbP\left( \sum_{j=1}^m \xi_j > S_1 + z\right) \leq \rme^{-\theta z + \frac{S_2}{2}\theta^2 + C_{\theta_0,\alpha_1}m\theta^3}.
\end{equation} Since $S_2 \geq m \alpha_2^2$ by Jensen's inequality,  choosing $\theta=\frac{z}{S_2}$ (which, under our current assumptions, satisfies $0 \leq \theta < \frac{cm}{m\alpha_2^2} = \frac{c}{\alpha_2^2} < \frac{1}{2}\theta_0$ is $c$ is taken sufficiently small) gives the bound
\begin{equation}\label{eq:boundind2}
		\bbP\left( \sum_{j=1}^m \xi_j > S_1 + z\right) \leq \rme^{-(\frac{1}{2} - \frac{1}{\alpha_2^4}C_{\theta_0,\alpha_1}\rho) \frac{z^2}{S_2}}.
\end{equation} Since $\sup_{1 \leq j \leq m}\E(\rme^{\theta_0\xi_j}) \leq \alpha_1$ implies that $S_2 \leq C_{\theta_0,\alpha_1}'m$ for some constant $C_{\theta_0,\alpha_1}'>0$ depending only on $\theta_0$ and $\alpha_1$, by choosing $\rho$ sufficiently small so that $\frac{1}{\alpha_2^4}C_{\theta_0,\alpha_1}\rho< \frac{1}{4}$, \eqref{eq:boundind2} now yields \eqref{eq:boundind} with $c:=(4C'_{\theta_0,\alpha_1})^{-1}$.

The proof of \eqref{eq:boundind_II} is similar: for all $z,\theta \geq 0$ we have
\begin{equation}
\bbP\left(\sum_{j=1}^m \xi_j < \sum_{j=1}^m \E(\xi_j) -z\right)=\bbP\left( \sum_{j=1}^m \xi_j < S_1 - z\right) \leq \rme^{\theta(S_1 -z)} \prod_{j=1}^m \E(\rme^{-\theta \xi_j}).	
\end{equation} but the difference now is that, since each $\xi_j$ is nonnegative, by Taylor's expansion for all $\theta \geq 0$ we have the estimate
\begin{equation}
\E(\rme^{-\theta \xi_j}) \leq 1 - \E(\xi_j)\theta + \E(\xi_j^2)\frac{\theta^2}{2},
\end{equation} which in turn gives the bound
\begin{equation}
\bbP\left(\sum_{j=1}^m \xi_j < \sum_{j=1}^m \E(\xi_j) -z\right) \leq \rme^{-\theta z+ \frac{S_2}{2}\theta^2}.
\end{equation} From here the proof now continues as for \eqref{eq:boundind}. 
\end{proof}

We are now ready to prove the results from Subsection~\ref{ss:locdownprelim}.

\begin{proof}[Proof of Lemma~\ref{lem:nhprob}]

To prove \eqref{eq:form2} we observe that, since the number $N_t$ of excursions of $X$ away from $\partial$ until the local time accumulated at $\partial$ is $t$ has Poisson distribution of parameter $\frac{\deg(\partial)}{2\pi}t$, it follows from the strong Markov property of $X$ at the different return times to $\partial$ that
	\begin{equation}
	\bbP_\partial( L_t(x) = 0)= \E_\partial\left[ (\bbP_\partial( \ol{\tau}_\partial < \tau_x))^{N_t}\right] = \exp\Big\{ - \frac{\deg(\partial)}{2\pi} t \bbP_\partial(\ol{H}_x < H_\partial)\Big\}\,.
	\end{equation} Thus, by  \eqref{e:h1} we immediately conclude that
\begin{equation}
\bbP_\partial( L_t(x) = 0) = \rme^{-\frac{t}{G_{\rmD_n}(x,x)}}	
\end{equation} and now the last inequality in \eqref{eq:form2} follows at once from \eqref{e:3.19.1}. 

It remains to show \eqref{eq:form3}. To this end, we first notice that, for any $x \in \rmD_n$ and $t>0$ we~have the equality in distribution 
\begin{equation}
L_t(x) \overset{d}{=} \sum_{i=1}^{N_t(x)} E_j,
\end{equation} where $N_t(x)$ is the number of excursions in which $x$ was visited until the local time accumulated at the boundary $\partial$ is $t$ and each $E_j$ denotes the total amount of time spent at $x$ by the walk during the $j$-th such excursion. But, on the one hand, the random variables $(E_j)_{j \in \N}$ are i.i.d. with an exponential distribution of parameter $\frac{1}{G_{\rmD_n}(x,x)}$ by \eqref{eq:form4} and, on the other hand, $N_t(x)$ has Poisson distribution of parameter $\frac{t}{G_{\rmD_n}(x,x)}$ by \eqref{e:h1} and is also independent of  the $(E_j)_{j \in \N}$ since the holding times are independent of the path of each excursion. In light of all this,~\eqref{eq:form3} is now an immediate consequence of \eqref{A:poi-exp} and \eqref{e:3.19.1}.
\end{proof}

Next, we have:
\begin{proof}[Proof of Proposition~\ref{l:4.4b}]
	 We prove only \eqref{e:3.19b}, as \eqref{e:3.19c} can be proved using a similar argument.

Fix $x$ and $n$ as in the statement and let us henceforth drop $x$ from the notation in $\rmB_r(x)$. Also, let us assume that $k$ is sufficiently large (depending only on $c_1$) so as to have $\ol{\rmB_{k+1}} \subseteq \rmB_k^- \setminus \partial_i \rmB_k^-$. Now, let $Y_j^{\text{in}} \in \partial_i \rmB_k^-$ and $Y_j^{\text{out}} \in \partial \rmB_k^+$ be respectively the entry and exit points corresponding to the $j$-th $(x;k)$-excursion of the walk (from $\partial_i \rmB_k^-$ back to $\partial \rmB_k^+$). By the strong Markov property of the walk, conditional on $\cF(x;k)$, on the event $\{ N_t(x;k) = m\}$ we have $\ol{L_t}(x;k+1)=\sum_{j=1}^{m} \ol{\xi}_j$, where the random variables $(\ol{\xi}_j)_{j=1}^{m}$ are independent and each $\ol{\xi}_j$ is given by harmonic average of the local time accumulated at the sites of $\partial \rmB_{k+1}$ during the $j$-th $(x;k)$-excursion, i.e.
\begin{equation}
\ol{\xi}_j= \sum_{w \in \partial \rmB_{k+1}} \ol{\xi}_j(w)\Pi_{\rmB_{k+1}}(x,w)\,,	
\end{equation} for $\ol{\xi}_j(w)$ the total local time spent at $w$ during the $j$-th $(x;k)$-excursion (which is conditioned to start at $Y_i^{\text{in}}$ and end at $Y_i^{\text{out}}$). In particular, since the random variables $\ol{\xi}_j$ are non-negative, if we set $\xi_j:=\frac{1}{(c_2-c_1)k^\gamma}\ol{\xi}_j$, $\wh{m}:=(c_2-c_1)k^\gamma m$ and $\wt{z}:=\frac{z}{(c_2-c_1)k^\gamma}$, on the event $\{ \wh{N}_t(x;k)\leq \wh{m}\}$ we have
\begin{equation}
		\label{e:3.19e}
		\bbP \big(\ol{L_t}(x;k+1) > \wh{m} + z
		\,\big|\,\cF(x;k) \big) 
		\leq \bbP\Bigg( \sum_{j=1}^m \xi_j > m + \wt{z} \,\Bigg|\,\cF(x;k) \Bigg).
	\end{equation}  
Therefore, if we show that for all $k$ large enough we almost surely have that 
\begin{equation}\label{e:3.19g}
	\E(\xi_j\,|\,\cF(x;k)) = 1 + O(\rme^{-(c_2-c_1)k^\gamma}+\rme^{-c_1k^\gamma})\hspace{1cm}\text{ and }\hspace{1cm}\E(\rme^{\theta_0 \xi_j}\,|\,\cF(x;k)) \leq \alpha_2
\end{equation} uniformly over $1 \leq j \leq m$ for some fixed $\theta_0,\alpha_2 > 0$ independent of $k$ then, using that $z \geq \wh{m} \rme^{-\delta k^\gamma}$, for all $k$ large enough we can bound the right-hand side of \eqref{e:3.19e} from above by
\begin{equation}\label{e:3.19f}
\bbP\Bigg( \sum_{j=1}^m \xi_j > \sum_{j=1}^m \E(\xi_j\,|\,\cF(x;k)) +\wt{z}/2 \,\Bigg|\,\cF(x;k) \Bigg). 
\end{equation} provided that $\delta$ is taken small enough (depending only on $c_1$ and $c_2$). In particular, if in~addition we take $\delta$ smaller than $2\rho$, where $\rho$ is the constant from Lemma~\ref{lem:A.5}, then we have that for any $\wh{m} \rme^{-\delta k^\gamma} \leq z \leq \delta \wh{m}$ the probability in \eqref{e:3.19f} is smaller than $\rme^{-c\frac{\wt{z}^2}{m}}$ for some $c>0$ and therefore \eqref{e:3.19b} now follows. Hence, it only remains to show \eqref{e:3.19g}.

To this end, let $\bbP_{y,y'}$ be the law of an $(x;k)$-excursion conditioned to start at $y$ and end at $y'$ and take $\xi$ to be the harmonic average of the local time accumulated at the sites of $\partial \rmB_{k+1}$ during this excursion normalized by $(c_2-c_1)k^\gamma$, i.e. $\xi\overset{d}{=}\xi_j$ whenever $Y^{\text{in}}_j =y$ and $Y^{\text{out}}_j =y'$. Then, proving \eqref{e:3.19g} amounts to showing that for all $k$ large enough we have 
\begin{equation}
		\label{e:3.23b}
		\bbE_{y,y'}(\xi) = 1 + O(\rme^{-(c_2-c_1)k^\gamma}+\rme^{-c_1k^\gamma})\hspace{1cm}\text{ and }\hspace{1cm}\bbE_{y,y'}(\rme^{\theta_0 \xi}) \leq \alpha_2
	\end{equation} uniformly over all possible choices of $y \in \partial_i\rmB_k^-$ and $y'\in \partial \rmB_k^+$, where $\theta_0,\alpha_2>0$ are some fixed constants independent of $k$. 
	
	We begin by dealing with the leftmost expectation in \eqref{e:3.23b}. To this end, let us define $G_{\rmB_{k}^+,y'}$ to be the Green function associated with an $(x;k)$-excursion, i.e. $G_{\rmB_{k}^+,y'}(y,w) := \bbE_{y,y'} (\bfL_{T_{\rmB_k^+}}(w))$. Since for any $w \in \rmB_k^+$ we have $\bbP_{w}(Z_{T_{\rmB_k^+}} = y' \,|\, T_{\rmB_k^+} < H_z ) = \Pi_{\rmB_k^+}(w,y')$ because any excursion starting from $w$ renews itself every time it revisits $w$, it is not hard to check that
	\begin{equation}
		G_{\rmB_k^+,y'}(y,w) = G_{\rmB_k^+}(y,w) \frac{\Pi_{\rmB_k^+}(w,y')}{\Pi_{\rmB_k^+}(y,y')},
	\end{equation} where $G_{\rmB_k^+}$ is our usual Green function and $\Pi$ is the Poisson kernel, as introduced in Section~\ref{s:dga}. Moreover, since $|w-w'|=\mathrm{e}^R(1+O(\rme^{-(R-r)}))$ holds uniformly over all $w \in \partial \rmB_R \cup \partial_i \rmB_R$, $w' \in \partial \rmB_r$ and any choice of $R>r\geq 1$, a standard computation using Lemma~\ref{l:103.2} shows that, uniformly over all $w \in \partial \rmB_{k+1}$ and $y \in \partial_i \rmB_k^-$, 
\begin{equation}\label{eq:boundG}
	G_{\rmB_k^+}(y,w)=(c_2-c_1)k^\gamma + O(\rme^{-c_1 k^\gamma}).
\end{equation} On the other hand, by Lemma~\ref{lem:kernel} we have that
	\begin{equation}\label{eq:boundkernel}
		\frac{\Pi_{\rmB_k^+}(z,y')}{\Pi_{\rmB_k^+}(y,y')} = 1 + O(e^{-(c_2-c_1)k^\gamma})
	\end{equation}
	holds uniformly over all $z \in \partial \rmB_{k+1}$, $y \in \partial_i \rmB_k^-$ and $y' \in \partial \rmB_k^+$. In view of these last three displays, we~see that the leftmost expectation in~\eqref{e:3.23b} is 
	\begin{equation}\label{eq:expm}
		\frac{1}{(c_2-c_1)k^\gamma} \sum_{z \in \partial \rmB_{k+1}} G_{\rmB_k^+,y'}(y,z)\Pi_{\rmB_k}(x,z) = 1 + O(e^{-(c_2-c_1)k^\gamma}+\rme^{-c_1k^\gamma}).
	\end{equation}

	Turning to the second expectation in~\eqref{e:3.23b}, by Kac's Moment Formula (see, e.g..~\cite[Section 3.19]{williams1979diffusions}) and \eqref{eq:boundkernel} we have, for all $k$ sufficiently large (depending only on $c_1$, $c_2$ and $\gamma$), $p \in \N$ and $y$, $y'$ as above,
	\begin{equation}\label{eq:expm1}
		\begin{split}
			\E_{y,y'} \Bigg( \bigg(\int_0^{T_{\rmB_k^+}} \sum_{z \in \partial \rmB_{k+1}} & 1_{z}(X_s)\Pi_{\rmB_{k+1}}(x,z) \rmd s \bigg)^p \Bigg)\\
			 & = 
			p! \sum_{\substack{y_1, \dots, y_p \\ y_0=y}} \prod_{j=1}^p G_{\rmB_{k}^+,y'}(y_{j-1}, y_j) \Pi_{\rmB_{k+1}}(x,y_j)1_{\partial \rmB_{k+1}}(y_j) \\
			& \leq p! 2^p \sum_{\substack{y_1, \dots, y_p \\ y_0=y}} \prod_{j=1}^p G_{\rmB_k^+}(y_{j-1}, y_j) \Pi_{\rmB_{k+1}}(x,y_j)1_{\partial \rmB_{k+1}}(y_j).
		\end{split}
	\end{equation}
	Now, on the one hand,  \eqref{eq:boundG} implies that for all $k$ sufficiently large we have
	\begin{equation}\label{eq:expm2}	
		G_{\rmB_k^+}(y_0, y_1) \leq 2(c_2-c_1)k^\gamma.
	\end{equation} uniformly over all $y_0 \in \partial \rmB_k^-$ and $y_1 \in \partial \rmB_{k+1}$. On the other hand, by doing a similar computation to the one yielding \eqref{eq:boundG}, we obtain that for all $y_{j-1},y_j \in \partial \rmB_{k+1}$ and all $k$ sufficiently large,
	\begin{align}
	G_{\rmB_k^+}(y_{j-1}, y_j) &\leq ((c_2+1)k^\gamma + O(1))1_{(\rmB(y_{j-1}; k-k^\gamma))^\rmc}(y_j) + (k+c_2k^\gamma + O(1))1_{\rmB(y_{j-1}; k-k^\gamma)}(y_j) \nonumber\\ 
		&\leq 2(c_2+1)(k^\gamma + k1_{\rmB(y_{j-1}; k-k^\gamma)}(y_j)).   \label{eq:boundGpair}              
	\end{align} Since $\Pi_{\rmB_{k+1}}(x,y_j)=O(\rme^{-k})$ uniformly over $y_j \in \partial \rmB_{k+1}$ by \cite[Lemma~6.3.7]{LL} and $|\partial \rmB_r| =O(\rme^r)$, \eqref{eq:boundGpair} implies that for all $k$ sufficiently large
	\begin{equation}\label{eq:expm3}
		\sum_{y_j} G_{\rmB_k^+}(y_{j-1}, y_j)\Pi_{\rmB_{k+1}}(x,y_j)1_{\partial \rmB_{k+1}}(y_j)\leq C_0k^\gamma+O(k\rme^{-k^\gamma}) \leq 2C_0k^\gamma                                                                                                                                                                                                                                                                                                                                                                                                                                                                                                                                                                                                                                                                                                                                                                                                                                                                                                                                                                                                                                                                                                                                                                                                                                                                                                                                                                                                                                                                                                                                                                                                                                                                                                                                                                                                                                                                                                                                                                                                                                                                                                                                                                                                                                                                                                                                                                                                                                                                                                                                                                                                                                                                                                                                                                                                                                                                                                                                                                                                                                                                                                                                                                                                                                                                                                                                                                                                                                                                                                                                                                                                                                                                                                                                                                                                                                                         
	\end{equation} uniformly over $y_{j-1} \in \partial \rmB_{k+1}$ and $j=1,\dots,p$, for some constant $C_0>0$ depending only on~$c_2$. Therefore, in light of \eqref{eq:expm1}, \eqref{eq:expm2} and \eqref{eq:expm3}, for all $k$ sufficiently large we obtain the bound 
	\begin{equation}
		\E_{y,y'}(\xi^p)=\frac{1}{((c_2-c_1)k^\gamma)^p}\E_{y,y'} \Bigg( \bigg(\int_0^{T_{\rmB_k^+}} \sum_{z \in \partial \rmB_{k+1}} 1_{z}(X_s)\Pi_{\rmB_{k+1}}(x,z) \rmd s \bigg)^p \Bigg)\leq p! C^p
	\end{equation} uniformly over all choices of entry/exit points~$y,y'$, where $C>0$ is some constant depending only on $c_1$ and $c_2$. In particular, interchanging sum and expectation via the Fubini-Tonelli~theorem, for $\theta_0 < C^{-1}$ the second expectation in~\eqref{e:3.23b} becomes 	\begin{equation}
		\sum_{p=0}^\infty \frac{\theta_0^p}{p!} \E_{y,y'}(\xi^p) \leq \sum_{p=0}^\infty (C\theta_0)^p = \frac{1}{1-C\theta_0}=:\alpha_2
	\end{equation} for all $k$ sufficiently large and all choices of $y,y'$ as above. Together with~\eqref{eq:expm}, this shows~\eqref{e:3.23b} and thus concludes the proof.
\end{proof}

To prove Proposition~\ref{prop:Abulk}, we shall need Lemma~\ref{lem:rep1} below which shows that, given~$\cF(y;l)$, $N_t(x;k)$ is stochastically comparable with a Binomial sum of i.i.d. Geometric random variables. In the sequel, the notation $\overset{d}{\leq}$ stands for stochastic domination. 
\begin{lem}\label{lem:rep1} If $l$ is sufficiently large (depending only on $c_1$, $c_2$) then, given any $n \geq l \geq k \geq 0$ and $x,y \in \rmD_n$ such that $\rmB(x; k+c_2k^\gamma) \subseteq \rmB(y;l)$ and $\ol{\rmB(y;l+c_2l^\gamma)} \subseteq \rmD_n$, for all $m \geq 1$ and~$t>0$ we have that, conditional on $\cF(y;l)$, on the event $\{N_t(y;l) =m\}$ it holds that
	\begin{equation}\label{eq:A.6.1} 
\sum_{j=1}^{m} A_j^-G_j^- \overset{d}{\leq}	N_t(x;k)\overset{d}{\leq}\sum_{j=1}^{m} A_j^+G_j^+\,,
	\end{equation} where the random variables $A_j^\pm $ and $G_j^\pm$ are all mutually independent and respectively distributed as 
		\begin{equation}\label{eq:sc1}
			A_j^\pm \sim \cB\Big(\frac{(c_2-c_1)l^\gamma +O^\pm(\rme^{-(k+c_1k^\gamma)}+\rme^{-c_1l^\gamma} + l^\gamma \rme^{-(c_2-c_1)l^\gamma} )}{l+c_2l^\gamma-(k +c_1 k^\gamma)}\Big)
	\end{equation} and
\begin{equation}\label{eq:sc1b}
	G_j^\pm \sim \cG\Big(\frac{(c_2-c_1)k^\gamma+O^\pm(\rme^{-(k+c_1k^\gamma)}+\rme^{-c_2l^\gamma})}{l+c_2l^\gamma-(k+c_1k^\gamma)}\Big).
		\end{equation} In addition, for each fixed $\eta \in (\gamma,1)$, if $n$ is large enough (depending only on $c_1, c_2, \eta$ and $\gamma$) then, given any $k \in [n^\eta,n-n^\eta]$ and $x \in \rmD_n^\circ$ such that $\ol{\rmB(x;k+c_2 k^\gamma)} \subseteq \rmD_n$, for all $t > 0$ we have
\begin{equation}\label{eq:A.6.2}
\sum_{j=1}^{N_t} G_j^- \overset{d}{\leq}	N_t(x;k)\overset{d}{\leq}\sum_{j=1}^{N_t} G_j^+\,,
\end{equation} where the random variables $N_t$ and $G_j^\pm$ are all mutually independent and respectively distributed as 
		\begin{equation}\label{eq:sc1e}
		N_t \sim \cP\bigg(\frac{1 + O\big(\rme^{-\frac{1}{2}n^\eta}\big )}{G_{\rmD_n}(x,x) - (k+c_1k^\gamma) - \gamma^*}t \bigg)
	\end{equation} and
\begin{equation}\label{eq:sc1d}
	G_j^\pm \sim \cG\bigg(\frac{(c_2-c_1)k^\gamma+O^\pm(\rme^{-\frac{1}{2}n^\eta})}{G_{\rmD_n}(x,x)-(k+c_1k^\gamma)-\gamma^*}\bigg).
		\end{equation} 
\end{lem}

\begin{proof} Let $V_j$ be the number of $(x;k)$-downcrossings during the $j$-th $(y;l)$-excursion of the walk. Then, on the one~hand, by definition we have that, on the event $\{N_t(y;l)=m\}$,
	\begin{equation}
		N_t(x;k)=\sum_{j=1}^{m} V_j.	
	\end{equation} On the other hand, the random variables $(V_j)_{j \in \N}$ are independent conditionally on $\cF(y;l)$ since, by the strong Markov property, these excursions are independent given their entry/exit points. Moreover, if $\bbP_{y_1,y_2}$ denotes the law of an $(y;l)$-excursion conditioned to start at $y_1$ and end at~$y_2$ and $V$ denotes the number of $(x;k)$-downcrossings during this excursion, for any $u \in \bbN$ we have
\begin{equation}
	\bbP(V_j \geq u \,|\, \cF(y;l)) = \bbP_{Y_j^{\text{in}},Y_j^{\text{out}}}(V \geq u), 
\end{equation} where $Y_j^{\text{in}}$ and $Y_j^{\text{out}}$ are respectively the entry and exit points of the $j$-th $(y;l)$-excursion of $X$. In light of all these observations, to prove \eqref{eq:A.6.1} it will be enough to show that for each $u \in \bbN$ we have 
	\begin{equation} \label{eq:sc2}
	p^-(1-q^-)^{u-1} \leq \bbP_{y_1,y_2}(V \geq u) \leq p^+(1-q^+)^{u-1} 
\end{equation} uniformly over all $y_1 \in \partial_i \rmB_l^-(y)$ and $y_2 \in \partial \rmB_l^+(y)$, where $p^-,p^+,q^-,q^+$ are some fixed parameters satisfying
\begin{equation}
	p^\pm=\frac{(c_2-c_1)l^\gamma +O(\rme^{-(k+c_1k^\gamma)}+\rme^{-c_1l^\gamma}+l^\gamma\rme^{-(c_2-c_1)l^\gamma})}{l+c_2l^\gamma-(k +c_1 k^\gamma)}
	\end{equation} and
	\begin{equation}
	q^\pm=\frac{(c_2-c_1)k^\gamma+O(\rme^{-(k+c_1k^\gamma)}+\rme^{-c_2l^\gamma})}{l+c_2l^\gamma-(k+c_1k^\gamma)}.
\end{equation} 
We shall only show the leftmost inequality in \eqref{eq:sc2}, since the other one can be proved similarly.
Henceforth, let us simplify the notation by omitting $x$ and $y$ and writing instead $\rmB_k^\pm:=\rmB_k^\pm(x)$ and $\rmB_l^\pm:=\rmB_l^\pm(y)$.  In order to show \eqref{eq:sc2} we first notice that, under our current assumptions, if $l$ is large enough then we have $\ol{\rmB_k^+} \subseteq\rmB_l^-$, so that $\partial_i \rmB_l^-$ and $\rmB_k^+$ become disjoint and therefore we obtain that $\partial_i \rmB_l^- \cup \partial \rmB_k^+\subseteq\rmB_l^+ \setminus \rmB_k^-$. Thus, if we consider the auxiliary random walk $Z$ on~$\Z^2$ introduced at the beginning of Section~\ref{s:A1} and define two sequences of stopping times $(S_j)_{j \in \N_0}$ and $(T_j)_{j \in \N}$ recursively by first setting $S_0:=0$ and then for $j \geq 1$,
	\begin{equation}\label{eq:stoppingtimes}
		\begin{array}{rl}
			T_j:=&\hspace{-0.2cm}\inf\{ k \geq S_{j-1} : Z_k \notin \rmB^+_l \setminus \rmB^-_k\}\\ \\
			S_j:=&\hspace{-0.2cm}\inf\{ k > T_{j} : Z_k \in \partial \rmB^+_k \},
		\end{array}
	\end{equation} for all $l$ large enough we may write, for any $u \in \bbN$ and all entry/exit points $y_1,y_2$ as above,
	\begin{equation}
		\bbP_{y_1,y_2} (V \geq u) = \frac{1}{\bbP_{y_1}(Z_{\tau_{\partial \rmB^+_l}}=y_2)}\bbP_{y_1} ( Z_{T_j} \in \rmB^-_k \,\,\forall j=1,\dots,m\,,\,Z_{\tau_{\partial \rmB^+_l}}=y_2).
	\end{equation} Then, using the strong Markov property successively at times $S_u,\dots,S_1$, we find that
	\begin{equation}\label{eq:scbound1}
		\bbP_{y_1,y_2}(V\geq u) \geq \bigg[\inf_{z \in \partial_i \rmB_l^-} \bbP_z(Z_{T_1} \in \rmB^-_k)\bigg] \bigg[\inf_{z' \in \partial \rmB^+_k}\bbP_{z'}(Z_{T_1} \in \rmB^-_k)\bigg]^{u-1}\Bigg[\frac{\inf_{w \in  \rmB_l^-}\bbP_w(Z_{\tau_{\partial \rmB^+_l}}=y_2)}{\sup_{w \in  \rmB_l^-}\bbP_w(Z_{\tau_{\partial \rmB^+_l}}=y_2)}\Bigg].
	\end{equation}
Therefore, by multiplying together the first and third factors from this lower bound in \eqref{eq:scbound1}, we see that in order to obtain \eqref{eq:sc2} it will be enough to establish the following three estimates (for all $l$ large enough): first that, uniformly over all $z \in \partial_i \rmB_l^-$,
	\begin{equation}\label{eq:sd2a}
		\bbP_z(Z_{T_1} \in \rmB^-_k)=\frac{ (c_2-c_1)l^\gamma + O(\rme^{-(k+c_1k^\gamma)}+\rme^{-c_1l^\gamma})}{l+c_2l^\gamma-(k +c_1 k^\gamma)};
	\end{equation} then that, uniformly over $z' \in \partial \rmB^+_k$, 
	\begin{equation}\label{eq:sd3a}
		\bbP_{z'}(Z_{T_1} \notin \rmB^-_k)=\frac{(c_2-c_1)k^\gamma+O(\rme^{-(k+c_1k^\gamma)}+\rme^{-c_2l^\gamma})}{l+c_2l^\gamma-(k+c_1k^\gamma)}; 
	\end{equation} and finally that, uniformly over all $w,w' \in \rmB_l^-$, 
\begin{equation}\label{eq:sd5a}
	\frac{\bbP_w(Z_{\tau_{\partial \rmB^+_l}}=y_2)}{\bbP_{w'}(Z_{\tau_{\partial \rmB^+_l}}=y_2)}=1+O(\rme^{-(c_2-c_1)l^\gamma}).
\end{equation} But \eqref{eq:sd5a} is a straightforward consequence of Lemma~\ref{lem:kernel}, so that we only need to check both \eqref{eq:sd2a} and \eqref{eq:sd3a}. To this end, notice that, since $x \in \rmB(y;l)$, for $l$ large enough we may choose some radii $\rho^+_l,\rho^-_l>0$ such that $\rho^\pm_l = l+c_2l^\gamma \pm O(\rme^{-c_2l^\gamma})$ and $\rmB_l^-\subseteq \rmB(x;\rho^-_l) \subseteq \rmB_l^+ \subseteq \rmB(x;\rho^+_l)$ (recall that the balls $\rmB_l^\pm$ are centered at $y$). In particular, since $T_1 = \tau_{(\rmB^+_l)^\rmc \cup \rmB^-_k}$, for any $w \in \rmB_l^-$ we have that 
\begin{equation}\label{eq:Arep1.comp}
	\bbP_w( Z_{\tau_{\rmB(x;\rho^-_l)^\rmc \cup \rmB_k^-}} \in \rmB_k^-) \leq \bbP_w (Z_{T_1} \in \rmB_k^-) \leq \bbP_w ( Z_{\tau_{\rmB(x;\rho^+_l)^\rmc \cup \rmB_k^-}} \in \rmB_k^-).
\end{equation} In light of these bounds, upon noticing that $\log |x-z| = l+c_1l^\gamma +O(\rme^{-c_1l^\gamma})$ for all $z \in \partial_i \rmB_l^-$, \eqref{eq:sd2a} and \eqref{eq:sd3a} now follow at once from Lemma~\ref{lem:ll1} applied to $\tau_{\rmB(x;\rho_l^-)^\rmc \cup \rmB_k^-}$ and $\tau_{\rmB(x;\rho_l^+)^\rmc \cup \rmB_k^-}$. This concludes the proof of \eqref{eq:A.6.1}.

The proof of \eqref{eq:A.6.2} is similar. Indeed, if $\ol{N}_t$ denotes the number of excursions of $X$ away from $\partial$ until it accumulates local time $t$ at $\partial$ and $\ol{V}_j$ denotes the number of $(x;k)$-downcrossings during the $j$-th such excursion, then we have
\begin{equation}
	N_t(x;k)=\sum_{j=1}^{\ol{N}_t} \ol{V}_j.
\end{equation} Observe that $\ol{N}_t \sim \cP(\frac{\deg(\partial)}{2\pi}t)$ by definition and also that $\ol{N}_t$ is independent of $(\ol{V}_j)_{j \in \N}$, since each excursion is independent of the holding times at $\partial$. Furthermore, by the strong Markov~property, the random variables $(\ol{V}_j)_{j \in \N}$ are i.i.d. and also, by carrying out a similar argument to the one yielding \eqref{eq:sc2} but using Lemma~\ref{lem:g} instead of Lemma~\ref{lem:ll1}, each $\ol{V}_j$ satisfies
\begin{equation}
	A_j G_j^- \overset{d}{\leq} \ol{V}_j \overset{d}{\leq} A_j G_j^+\,
\end{equation} where the random variables $A_j$ and $G_j^\pm$ are all mutually independent (and independent of $\ol{N}_t$) and respectively distributed as
\begin{equation}
	A_j \sim \cB\bigg( \frac{2\pi}{\deg(\partial)} 
	\frac{1 + O\big(\rme^{-(k+c_1k^\gamma) \wedge(n-2\log n-k-c_1k^\gamma)}\big)}{G_{\rmD_n}(x,x) - (k+c_1k^\gamma) - \gamma^*}\bigg)
\end{equation} 
and
\begin{equation}
G_j^\pm = \cG\bigg(\frac{ (c_2-c_1)k^\gamma + O^\pm\big(\rme^{-(k+c_1 k^\gamma) \wedge (n-2 \log n - k-c_2k^\gamma)}\big)}{G_{\rmD_n}(x,x) - (k+c_1k^\gamma) - \gamma^*}\bigg)\,.
\end{equation}
Since $(k+c_1 k^\gamma) \wedge (n-2 \log n - k-c_2k^\gamma) \geq \frac{1}{2}n^\eta$ for all $k \in [n^\eta,n-n^\eta]$ if $n$ is sufficiently large, \eqref{eq:A.6.2} now follows upon noticing the equalities in distribution
\begin{equation}
	\sum_{j=1}^{\ol{N}_t} A_j G_j^- \overset{d}{=} \sum_{j=1}^{N_t} G_j^- \quad \text{ , }\quad\sum_{j=1}^{\ol{N}_t} A_j G_j^+ \overset{d}{=} \sum_{j=1}^{N_t} G_j^+
\end{equation} for $N_t$ as in \eqref{eq:sc1e} taken independent of the $G_j^\pm$.
\end{proof}

We are now ready to prove Proposition~\ref{prop:Abulk}.

\begin{proof}[Proof of Proposition~\ref{prop:Abulk}] 
We begin by showing \eqref{eq:proofAbulk2}. By Lemma~\ref{lem:rep1} we have that
\begin{equation}
\bbP\Big( \sqrt{\wh{N}_t(x;k)} \leq v \,\Big|\,\sqrt{\widehat{N}_t(y;l)} \geq u\,; \cF(y;l) \Big) \leq \bbP\left( \sum_{j=1}^{\lceil \wt{u}_l\rceil}  A_j^- G_j^- \leq  \wt{v}_k\right), 
\end{equation} where $A_j^-\sim \cB(p^-)$ and $G_j^-\sim \cG(q^-)$ are as in \eqref{eq:sc1}--\eqref{eq:sc1b}, $\wt{u}_l:=\frac{u^2}{(c_2-c_1)l^\gamma}$ and $\wt{v}_k:=\frac{v^2}{(c_2-c_1)k^\gamma}$. Under our assumptions on $l,u$ and $v$, if $\delta$ is small enough (depending only on $c_1$, $c_2$ and $\gamma$), then for all $k$ sufficiently large we have that $\wt{v}_k < \frac{\wt{u}_lp^-}{q^-}$ and also that
\begin{equation}
	(\sqrt{\wt{u}_lp^-} - \sqrt{\wt{v}_kq^-})^2 = \frac{(u-v)^2}{l+c_2l^\gamma - (k+c_1k^\gamma)} + O(\rme^{-\ol{c}k^\gamma})
\end{equation} for some constant $\ol{c}>0$, so that~\eqref{A:bin-geo} then yields the bound
\begin{equation}
	 \bbP\left( \sum_{j=1}^{\lceil \wt{u}_l\rceil}  A_j^- G_j^- \leq \wt{v}_k\right) \leq \rme^{-(\sqrt{\wt{u}_lp^-} - \sqrt{\wt{v}_kq^-})^2} \leq 2\exp\Big\{ - \frac{(u-v)^2}{l+c_2l^\gamma - (k+c_1k^\gamma)}\Big\}
\end{equation} for all $k$ sufficiently large and \eqref{eq:proofAbulk2} now follows.  

The proof of \eqref{eq:downesta} is similar. Indeed, by Lemma~\ref{lem:rep1} we have that
\begin{equation}\label{eq:boundsq2}
	\bbP\Big( \sqrt{\wh{N}_t(x;k)} \leq v \Big) \leq \bbP\left( \sum_{j=1}^{N_t}  G_j^- \leq  \wt{v}_k\right), 
\end{equation} with $N_t\sim \cP(\lambda)$ and $G_j^-\sim \cG(q^-)$ as in \eqref{eq:sc1e}--\eqref{eq:sc1d} and $\wt{v}_k$ as above. Under our assumptions on $t$ and~$v$, for all $n$ sufficiently large we have that $\wt{v}_k < \frac{\lambda}{q^-}$ and, in addition, that
\begin{equation}\label{eq:boundsq}
	(\sqrt{\lambda}-\sqrt{\wt{v}_kq^-})^2=\frac{(\sqrt{t}-v)^2 +O(\rme^{-\frac{1}{4}n^{\eta}})}{G_{\rmD_n}(x,x)-(k+c_1k^\gamma) - \gamma^*}\,.
\end{equation} Furthermore, Lemma~\ref{l:103.2} implies that, if $n$ is sufficiently large (depending only on $\eta$, $\gamma$ and $c_1$), for all $x \in \rmD_n^\circ$ and $k \in [n^\eta,n-n^\eta]$ we have 
\begin{equation}
	1 \leq G_{\rmD_n}(x,x)-(k+c_1k^\gamma) - \gamma^* \leq n+C-(k+c_1k^\gamma)
\end{equation} for some constant $C=C(\rmD)>0$, so that by~\eqref{A:poi-geo} and~\eqref{eq:boundsq} we obtain the bound 
\begin{equation}
	 \bbP\left( \sum_{j=1}^{N_t}  G_j^- \leq \wt{v}_k\right) \leq \rme^{-(\sqrt{\lambda} - \sqrt{\wt{v}_kq^-})^2} \leq 2\exp\Big\{ - \frac{(\sqrt{t}-v)^2}{n+C-(k+c_1k^\gamma)}\Big\}
\end{equation} for all $n$ sufficiently large and, in light of \eqref{eq:boundsq2}, from this \eqref{eq:downesta} immediately follows.
\end{proof}

As it was the case for Proposition \ref{prop:Abulk}, the key to proving Proposition \ref{prop:Abulk2} is the following lemma which states that, given $\cF(y;l)$, $L_t(x)$ is stochastically comparable with a Binomial sum of i.i.d. Exponential random variables.
\begin{lem}\label{lem:rep2} If $l$ is sufficiently large (depending only on $c_1$, $c_2$ and $\gamma$) then, given any $n \geq l$ and $x,y \in \rmD_n$ such that $x \in  \rmB(y;l)$ and $\ol{\rmB(y;l+c_2l^\gamma)} \subseteq \rmD_n$, for all $m \geq 1$ and $t>0$ we have that, conditional on $\cF(y;l)$, on the event $\{N_t(y;l) =m\}$ it holds that
	\begin{equation}
		\sum_{i=1}^{m} A_i^-E_i^- \overset{d}{\leq}	L_t(x)\overset{d}{\leq}\sum_{i=1}^{m} A_i^+E_i^+
	\end{equation} where the random variables $A_i^\pm $ and $E_i^\pm$ are all mutually independent and respectively distributed as 
	\begin{equation}\label{eq:sc1c}
		A_i^\pm \sim \cB\big((c_2-c_1)l^{\gamma-1} +O^\pm( l^{2(\gamma-1)})\big) \hspace{0.5cm}\text{ and }\hspace{0.5cm}
		E_i^\pm \sim \cE\big(l^{-1}+O^\pm(l^{\gamma-2})\big).
	\end{equation}
\end{lem}

\begin{proof} 
Let us abbreviate again $\rmB_l^+:= \rmB_l^+(y)$ and $\rmB_l^-:= \rmB_l^-(y)$. Then, by the same argument used for the proof of Lemma~\ref{lem:rep1}, we see that to prove the lemma it will be enough to show that for each $t>0$ we have 
	\begin{equation} \label{eq:sc2-b}
		p^-\rme^{-\lambda^- t} \leq \bbP_{y_1,y_2}(W > t) \leq p^+\rme^{-\lambda^+t} 
	\end{equation} uniformly over all $y_1 \in \partial_i \rmB^-_l$ and $y_2 \in \partial \rmB^+_l$ where $\bbP_{y_1,y_2}$ denotes the law of any $(y;l)$-excursion of the walk conditioned to start at $y_1$ and end at $y_2$, $W$ is the local time accumulated at $x$ during this $(y;l)$-excursion and $p^\pm,\lambda^\pm$ are some fixed parameters satisfying
	\begin{equation}
	p^\pm = (c_2-c_1)l^{\gamma-1} +O^\pm( l^{2(\gamma-1)}) \hspace{0.5cm}\text{ and }\hspace{0.5cm}\lambda^\pm = l^{-1}+O^\pm(l^{\gamma-2}).
	\end{equation}
To this end, observe that for any $t>0$ and all entry/exit points $y_1,y_2$ as above we have
\begin{equation}
	\bbP_{y_1,y_2}(W > t) = \frac{1}{\bbP_{y_1}(Z_{\tau_{\partial \rmB^+_l}} = y_2)}\bbP_{y_1}(Z_{\tau_{\partial \rmB_l^+ \cup \{x\}}} = x \,,\,W>t\,,\,Z_{\tau_{\partial \rmB^+_l}}=y_2) 
\end{equation} so that by the strong Markov property at $\tau_{\partial \rmB_l^+ \cup \{x\}}$ we obtain
\begin{equation}\label{eq:aw4}
	\bbP_{y_1,y_2}(W > t) = \frac{\bbP_x(Z_{\tau_{\partial \rmB^+_l}} = y_2)}{\bbP_{y_1}(Z_{\tau_{\partial \rmB^+_l}} = y_2)}\bbP_{y_1}(Z_{\tau_{\partial \rmB_l^+ \cup \{x\}}} = x)\bbP_{x,y_2}(W>t).
\end{equation}
By Lemma~\ref{lem:kernel} we have that, uniformly over all $x \in \rmB(y;l)$, $y_1 \in \partial_i \rmB_l^-$ and $y_2 \in \partial \rmB_l^+$,
\begin{equation}\label{eq:aw1}
\frac{\bbP_x(Z_{\tau_{\partial \rmB^+_l}} = y_2)}{\bbP_{y_1}(Z_{\tau_{\partial \rmB^+_l}} = y_2)}=1+O(\rme^{-(c_2-c_1)l^\gamma}).
\end{equation} On the other hand, 
using the same type of comparison argument leading to~\eqref{eq:Arep1.comp} in combination with \eqref{eq:LLest2} from Lemma~\ref{lem:ll1} yields that
\begin{equation}\label{eq:aw2}
\bbP_{y_1}(Z_{\tau_{\partial \rmB_l^+ \cup \{x\}}} = x) = \frac{(c_2-c_1)l^\gamma + O(\rme^{-c_1l^\gamma})}{l+c_2l^\gamma}(1+O(l^{-1}))=(c_2-c_1)l^{\gamma-1} + O(l^{2(\gamma-1)}).
\end{equation} Finally, since the path of each excursion is independent of the holding times at each of the sites and, moreover, any excursion starting from $x$ renews itself every time it revisits $x$, under $\bbP_{x,y'}$ the random variable $W$ has exponential distribution with parameter 
\begin{equation}\label{eq:aw3}
p_{x,y_2}:=\frac{2}{\pi}\bbP_{x,y_2}(\tau_{\partial \rmB_l^+} < \tau_x)=	\frac{2}{\pi}\bbP_{x}(\tau_{\partial \rmB_l^+} < \tau_x) = \frac{1}{G_{\rmB_l^+}(x,x)}=\frac{1+O(l^{-1})}{l+c_2l^\gamma}=l^{-1}+O(l^{\gamma-2}),
\end{equation} where the second to last equality follows from \eqref{e:3.19.2} in Lemma~\ref{l:103.2}. Therefore, in light of \eqref{eq:aw1},  \eqref{eq:aw2} and \eqref{eq:aw3}, \eqref{eq:sc2-b} is now a consequence of the decomposition in \eqref{eq:aw4}.
\end{proof}

We can now give the proof of Proposition \ref{prop:Abulk2}.

\begin{proof}[Proof of Proposition~\ref{prop:Abulk2}] By Lemma~\ref{lem:rep2} we have that 
\begin{equation}
\bbP\Big( \sqrt{L_t(x)} \leq v \,\Big|\, \sqrt{\widehat{N}_t(y;l)} \geq \sqrt{2}l +s\,;\,\cF(y;l) \Big) \leq \bbP\left( \sum_{j=1}^{\lceil \wt{s}_l\rceil}  A_j^- E_j^- \leq  v^2\right), 
\end{equation} where $A_j^-\sim \cB(p^-)$ and $E_j^-\sim \cE(\lambda^-)$ are as in Lemma~\ref{lem:rep2} and $\wt{s}_l:=\frac{(\sqrt{2}l+s)^2}{(c_2-c_1)l^\gamma}$. Furthermore, since $s \geq -l$, it is not hard to see that for all $l$ sufficiently large (depending only on $c_1$, $c_2$ and~$\gamma$) we have that $v^2 < \frac{\wt{s}_lp^-}{\lambda^-}$ and also that
\begin{equation}
	-(\sqrt{\wt{s}_lp^-} - \sqrt{v^2 \lambda^-})^2 \leq -2l -2\sqrt{2}s - \frac{s^2}{l} + \wh{C}_v l^\gamma\Big(1+\frac{s^2}{l^2}\Big),
\end{equation} for some large enough constant $\wh{C}_v > 0$. In view of~\eqref{A:bin-exp}, this gives the bound
\begin{equation}
	 \bbP\left( \sum_{j=1}^{\lceil \wt{s}_l\rceil}  A_j^- E_j^- \leq v^2\right) \leq \rme^{-(\sqrt{\wt{s}_lp^-} - \sqrt{v^2\lambda^-})^2} \leq \exp\Big\{ -2l -2\sqrt{2}s -\frac{s^2}{l} +\wh{C}_vl^\gamma \Big(1+\frac{s^2}{l^2}\Big)\Big\}
\end{equation} for all $l$ sufficiently large, and so \eqref{eq:proofAbulk3} now follows.

To show \eqref{eq:proofAbulk3b}, notice that by Lemma~\ref{lem:rep2} we have that 
\begin{equation}
	\begin{split}
	\bbP\Big( \sqrt{L_t(x)} \leq v \,\Big|\, \sqrt{\widehat{N}_t(y;l)} \leq \sqrt{2}l +s\,;\,\cF(y;l) \Big) & \geq \bbP\left( \sum_{j=1}^{\lfloor \wt{s}_l \rfloor} A_j^+ E_j^+ = 0 \right) \\
	& = (1-p^+)^{\lfloor \wt{s}_l \rfloor} \geq \rme^{\wt{s}_l \log(1-p^+)} \,,
\end{split}
\end{equation} for $A_j^+\sim \cB(p^+)$ and $E_j^+\sim \cE(\lambda^+)$ as in Lemma~\ref{lem:rep2} and $\wt{s}_l$ as above. Then, since $p^+=O(l^{\gamma-1})$, we obtain that for all $l$ sufficiently large
\begin{equation}
	\wt{s}_l \log(1-p^+)\geq -\Big(2l+2\sqrt{2}s +\frac{s^2}{l}\Big)(1+ \wt{C}_1 l^{\gamma-1}) \geq -2l-2\sqrt{2}s -\frac{s^2}{l} -\wt{C}_2l^\gamma\Big(1+\frac{s^2}{l^2}\Big)
\end{equation} for some appropriately large constants $\wt{C}_1,\wt{C}_2 >0$, from where \eqref{eq:proofAbulk3b} now readily follows. 
\end{proof}

The proof of Proposition \ref{l:4.6} relies on the following result giving a stochastic lower bound for the local time spent on the set $\{x,z\}$ in the vein of those in Lemmas \ref{lem:rep1} and \ref{lem:rep2}.

\begin{lem}\label{lem:rep3} For any $\delta >0$, if $l$ is sufficiently large (depending only on $\delta$, $c_1$, $c_2$ and $\gamma$) then, given $n \geq 1$ and $x,y,z \in \rmD_n$ such that $x,z \in  \rmB(y;l)$, $\log |x-z| \geq l -\delta l^\gamma$ and $\ol{\rmB(y;l+c_2l^\gamma)} \subseteq \rmD_n$, for all $m \geq 1$ and $t>0$ we have that, conditional on $\cF(y;l)$, on the event $\{N_t(y;l) =m\}$ it~holds that
	\begin{equation}
		L_t(x)+L_t(y) \overset{d}{\geq}\sum_{i=1}^{m} A_iE_i
	\end{equation} where the random variables $A_i$ and $E_i$ are all mutually independent and respectively distributed as 
	\begin{equation}
		A_i \sim \cB\big(2(c_2-c_1)l^{\gamma-1} +O_\delta( l^{2(\gamma-1)})\big) \hspace{0.5cm}\text{ and }\hspace{0.5cm}
		E_i \sim \cE\big(l^{-1}+O(l^{\gamma-2})\big).
	\end{equation}
\end{lem}

\begin{proof} By the same argument used to prove Lemma~\ref{lem:rep2} we see that it will suffice to show that, in the notation used in the proof of the aforementioned lemma, for each $t>0$ we have 
	\begin{equation} 
		\bbP_{y_1,y_2}(W_{x,z} > t) \geq p \rme^{-\lambda t} 
	\end{equation} uniformly over all $y_1 \in \partial_i \rmB^-_l$ and $y_2 \in \partial \rmB^+_l$, where $\bbP_{y_1,y_2}$ denotes the law of any $(y;l)$-excursion of the walk conditioned to start at $y_1$ and end at $y_2$, $W_{x,z}$ is the local time accumulated at $\{x,z\}$ during this $(y;l)$-excursion and $p,\lambda$ are some fixed parameters satisfying
	\begin{equation}\label{eq:paramrep3}
	p =2(c_2-c_1)l^{\gamma-1} +O_c( l^{2(\gamma-1)}) \hspace{0.5cm}\text{ and }\hspace{0.5cm}\lambda = l^{-1}+O(l^{\gamma-2}).
	\end{equation}
 Note that, as in the proof of Lemma~\ref{lem:rep2}, for any $t>0$ and all entry/exit points $y_1,y_2$ as above we have
\begin{equation}
	\bbP_{y_1,y_2}(W > t) = \frac{1}{\bbP_{y_1}(Z_{\tau_{\partial \rmB^+_l}} = y_2)}\bbP_{y_1}(Z_{\tau_{\partial \rmB^+_l \cup \{x.z\}}} \in \{x,z\} \,,\,W_{x,z}>t\,,\,Z_{\tau_{\partial \rmB^+_l}}=y_2) 
\end{equation} so that, using the strong Markov property at $\tau_{\partial \rmB^+_l \cup \{x.z\}}$, a simple computation yields that
\begin{equation}\label{eq:aw4b}
	\bbP_{y_1,y_2}(W_{x,z} > t) \geq \frac{\inf_{w \in \rmB(y;l)}\bbP_w(Z_{\tau_{\partial \rmB^+_l}} = y_2)}{\bbP_{y_1}(Z_{\tau_{\partial \rmB^+_l}} = y_2)}\bbP_{y_1}(Z_{\tau_{\partial \rmB^+_l \cup \{x.z\}}} \in \{x,z\})\inf_{w \in \rmB(y;l)}\bbP_{w,y_2}(W_{x,z}>t).
\end{equation}
As before, by Lemma~\ref{lem:kernel} we have that, uniformly over all $y_1 \in \partial_i \rmB_l^-$ and $y_2 \in \partial\rmB_l^+$,
\begin{equation}\label{eq:aw1b}
\frac{\inf_{w \in \rmB(y;l)}\bbP_w(Z_{\tau_{\partial \rmB^+_l}} = y_2)}{\bbP_{y_1}(Z_{\tau_{\partial \rmB^+_l}} = y_2)}=1+O(\rme^{-(c_2-c_1)l^\gamma}).
\end{equation} On the other hand, if $l$ is large enough (depending only on $c_1$) so as to have $\rmB(y;l) \cap \partial_i \rmB_l^- = \emptyset$ then $y_1 \notin \{x.z\}$ and thus by the inclusion-exclusion principle and \eqref{eq:defret} we can write
\begin{equation}\label{eq:awe1}
\begin{split}
\bbP_{y_1}(Z_{\tau_{\partial \rmB^+_l \cup \{x.z\}}} \in \{x,z\})&=	\bbP_{y_1}( \tau_x \wedge \tau_z < \tau_{\partial \rmB_l^+})\\
&=\bbP_{y_1}(\tau_x < \tau_{\partial \rmB_l^+}) + \bbP_{y_1}(\tau_z < \tau_{\partial \rmB_l^+})-\bbP_{y_1}(\tau_x \vee \tau_z < \tau_{\partial \rmB_l^+}).
\end{split}
\end{equation}
Now, as in the proof of Lemma~\ref{lem:rep2} (see \eqref{eq:aw2}) we have that for $w \in \{x,z\}$,
\begin{equation}\label{eq:aw2b}
\bbP_{y_1}(\tau_w < \tau_{\partial \rmB_l^+}) =
\bbP_{y_1}(Z_{\tau_{\partial \rmB_l^+ \cup \{w\}}} = w) =(c_2-c_1)l^{\gamma-1} + O(l^{2(\gamma-1)}).
\end{equation} In addition, by the strong Markov property at $\tau_x \wedge \tau_z$ we have that
\begin{equation} \label{eq:awe2}
\begin{split}
	\bbP_{y_1}(\tau_x \vee \tau_z < \tau_{\partial \rmB_l^+}) &= \bbP_{x}(\tau_z < \tau_{\partial \rmB_l^+})\bbP_{y_1}(\tau_x < \tau_z \wedge \tau_{\partial \rmB_l^+})+ \bbP_{z}(\tau_x < \tau_{\partial \rmB_l^+})\bbP_{y_1}(\tau_z < \tau_x \wedge \tau_{\partial \rmB_l^+})\\
	& \leq \bbP_{x}(\tau_z < \tau_{\partial \rmB_l^+})\bbP_{y_1}(\tau_x < \tau_{\partial \rmB_l^+})+ \bbP_{z}(\tau_x < \tau_{\partial \rmB_l^+})\bbP_{y_1}(\tau_z < \tau_{\partial \rmB_l^+}).
\end{split}	
\end{equation} Recalling that $\log|x-z| \geq l -\delta l^\gamma$, the same type of argument leading to~\eqref{eq:aw2} shows that
\begin{equation}\label{eq:aw3b}
	\bbP_{x}(\tau_z < \tau_{\partial \rmB_l^+}) =\bbP_x( Z_{\tau_{\partial \rmB_l^+ \cup \{z\}}} =z)\leq \frac{ (c_2+c)l^\gamma + O(l^{\gamma-1})}{l+c_2l^\gamma} = O_\delta(l^{\gamma-1})
\end{equation} and by symmetry that the same bound holds for $\bbP_{z}(\tau_x < \tau_{\partial \rmB_l^+})$. In light of the decompositions in \eqref{eq:awe1} and \eqref{eq:awe2}, a straightforward computation combining \eqref{eq:aw2b} and \eqref{eq:aw3b} now yields the estimate
\begin{equation}\label{eq:aw5b}
\bbP_{y_1}(Z_{\tau_{\partial \rmB^+_l \cup \{x.z\}}} \in \{x,z\}) \geq 	2(c_2-c_1)l^{\gamma-1} + O_\delta(l^{2(\gamma-1)})
\end{equation} uniformly over all $y_1 \in \partial_i \rmB_l^-$ (and $x,z$ as in the statement of the lemma).

Finally, unlike in the proof of Lemma~\ref{lem:rep2} where the random variable $W$ was  exponential, here $W_{x,z}$ is not exponentially distributed but rather stochastically bounded from below~by~an exponential random variable of parameter $\frac{2}{\pi}q$, where
\begin{equation}
q:=\max\{ \bbP_{x,y_2}(\tau_{\partial \rmB_l^+} < \tau_x \wedge \tau_z),\bbP_{z ,y_2}(\tau_{\partial \rmB_l^+} < \tau_x \wedge \tau_z)\}.
\end{equation} Indeed, starting from $x$ (respectively $z$), the walk will be held at this site for an exponential random time of parameter $\frac{2}{\pi}$ and then jump away to one of its nearest neighbors, with probability $\bbP_{x,y_2}(\tau_{\partial \rmB_l^+} < \tau_x \wedge \tau_z)$ (resp. $\bbP_{z,y_2}(\tau_{\partial \rmB_l^+} < \tau_x \wedge \tau_z)$) of exiting from $\rmB_l^+$ at once before revisiting the set $\{x,z\}$. 
Should the walk choose to revisit $\{x,z\}$ before exiting from $\rmB_l^+$, it will start~afresh once it reaches any of these two sites and then reiterate the procedure described above, until a final iteration in which it chooses to exit $\rmB_l^+$. It follows from this that, under either $\bbP_{x,y_2}$~or~ $\bbP_{z,y_2}$, we have 
\begin{equation}
W_{x,y} \overset{d}{=} \sum_{j=1}^G E_j,	
\end{equation} where $(E_j)_{j \in \bbN}$ is a sequence of i.i.d. exponential random variables of parameter $\frac{2}{\pi}$ and $G$ is stochastically bounded from below by a geometric random variable with parameter $q$. From this our previous claim now easily follows. But, since by \eqref{eq:aw3} we have
\begin{equation}
	\frac{2}{\pi}q \leq \frac{2}{\pi} \max_{w \in \rmB(y;l)} \bbP_{w,y_2} (\tau_{\partial \rmB_l^+}< \tau_w) =l^{-1}+O(l^{\gamma-2}),
\end{equation} we conclude that for any $t>0$,
\begin{equation}\label{eq:aw6b}
\inf_{w \in \rmB(y;l)}\bbP_{w,y_2}(W_{x,z}>t) \geq \rme^{-\lambda t}
\end{equation} for $\lambda$ as in \eqref{eq:paramrep3}. In light of \eqref{eq:aw4b} and the bounds in \eqref{eq:aw1b}, \eqref{eq:aw5b} and \eqref{eq:aw4b}, the lemma now immediately follows.
\end{proof}

We can now give the proof of Proposition~\ref{l:4.6}.
\begin{proof}[Proof of Proposition~\ref{l:4.6}] 
	By Lemma~\ref{lem:rep3}, on the event $\Big\{\sqrt{\widehat{N}_t(y;l)} \geq \sqrt{2}l +s \Big\}$ we have that
\begin{equation}
	\begin{split}
\bbP\Big( \sqrt{L_t(x)} \leq v,\sqrt{L_t(z)} \leq v \,\Big|\, \cF(y;l) \Big) & \leq \bbP\Big( L_t(x)+L_t(z) \leq 2v^2 \,\Big|\, \cF(y;l) \Big) \\
& \leq \bbP\left( \sum_{j=1}^{\lceil \wt{u}_l\rceil}  A_j E_j \leq  2v^2\right) \,,
\end{split}
\end{equation} where $A_j\sim \cB(p)$ and $E_j\sim \cE(\lambda)$ are as in Lemma~\ref{lem:rep3} and $\wt{s}_l:=\frac{(\sqrt{2}l+s)^2}{(c_2-c_1)l^\gamma}$. From here, the proof now continues as that of~\eqref{eq:proofAbulk3}, we omit the details. 
\end{proof}

We now turn to the proof of Proposition~\ref{l:4.5}. Recall that we denote by $\wt{X} = (\wt{X}(s) :\: s \geq 0)$ a continuous time random walk with the usual edge transition rate. Abbreviating $\rmB(r) \equiv \rmB(0;r)$, given $k \geq 0$, we let $\tau$, $\tau^\uparrow$, $\tau^\downarrow$ be the hitting times of $\partial \rmB(k+k^\gamma)$, $\partial_i \rmB(k+k^\gamma/2)$ and $\partial \rmB(k+k^\gamma/4)$ respectively.
Given $x \in \bbZ^2$ we shall write $\bbP_x$ for the law of $\wt{X}$ when $\wt{X}(0) = x$. If $y \in \partial \rmB(k+k^\gamma)$ and $x \in \rmB(k+k^\gamma)$ and we shall write $\bbP_{x,y}$ for the law of $\wt{X}$ under $\bbP_x$ conditioned on $\{\wt{X}(\tau) = y\}$. For such $x$ and $y$, we also define $\Pi^\downarrow_{x,y}$ and $\Pi^\uparrow_{x,y}$ to be the law of $\wt{X}(\tau \wedge \tau^\downarrow)$, resp. $\wt{X}(\tau \wedge \tau^\uparrow)$ under $\bbP_{x,y}$. The latter are probability distributions whose support is included in  $\partial \rmB(k+k^\gamma/4) \cup \{y\}$ and $\partial_i \rmB(k+k^\gamma/2) \cup \{y\}$ respectively. For the sake of comparison between $\Pi^\downarrow_{x,y}$ and $\Pi^\downarrow_{x',y'}$ with $y \neq y'$ we shall identify $y$ with $y'$ in the spaces $\partial \rmB(k+k^\gamma/4) \cup \{y\}$ and $\partial \rmB(k+k^\gamma/4) \cup \{y'\}$ respectively.

We recall that if $\Pi$ and $\Pi'$ are two probability distributions on a finite (or countable) space $\cX$, then there exists a distribution $\Pi^{\rm TV}$ on the product space $\cX \times \cX$, such that its marginals are $\Pi$ and $\Pi'$, and such that $\Pi^{\rm TV}(\{(x,x) :\: x \in \cX\}) = 1-\|\Pi - \Pi'\|_{\rm TV}$, where $\| \cdot \|_{\rm TV}$ denotes the Total Variation Norm. We shall call $\Pi^{\rm TV}$ the Total Variation (TV) coupling distribution of $\Pi$ and $\Pi'$.

\begin{lem}
\label{l:A.17}
Let $k \geq 0$. For all $x,x' \in \partial_i \rmB(k+k^\gamma/2)$ and $y,y' \in \partial \rmB(k+k^\gamma)$,
\begin{equation}
\label{e:A.92}
\Big\| \Pi^\downarrow_{x,y} - \Pi^\downarrow_{x',y'} \Big\|_{\rm TV} \leq \rme^{-ck ^\gamma} \,,
\end{equation}
and for all $x,x' \in \partial \rmB(k+k^\gamma/4)$ and $y,y' \in \partial \rmB(k+k^\gamma)$,
\begin{equation}
\label{e:A.93}
\Big\| \Pi^\uparrow_{x,y} - \Pi^\uparrow_{x',y'} \Big\|_{\rm TV} \leq \rme^{-ck ^\gamma} \,.
\end{equation}
Moreover for all $x \in \partial_i \rmB(k+k^\gamma/2)$ and $y \in \partial \rmB(k+k^\gamma)$,  
\begin{equation}
\label{e:A.94a}
\Pi^\downarrow_{x,y}(\{y\}) = \bbP_{x,y} \big(\tau < \tau^\downarrow) = \frac13 \big(1+O(\rme^{-k^\gamma})\big)
\end{equation}
\end{lem}

\begin{proof}
We shall only prove the first of the two statements as the second is similar and even slightly simpler. 
It will suffice to show that for any
$x,x',y,y'$ as in the first statement of the lemma and any $z \in \partial \rmB(k+k^\gamma/4)$,
\begin{equation}
\label{e:A.192}
\Big|\bbP_{x,y} \big(\wt{X}(\tau \wedge \tau^\downarrow) = z\big)
- \bbP_{x',y'} \big(\wt{X}(\tau \wedge \tau^\downarrow) = z\big)\Big| \leq 
\bbP_{x,y} \big(\wt{X}(\tau \wedge \tau^\downarrow) = z\big) \rme^{-ck^{\gamma}} \,.
\end{equation}
for some $c > 0$. Indeed, the left hand side of~\eqref{e:A.92} is at most twice the sum of the left-hand side above over all $z \in \partial \rmB(k+k^\gamma/4)$.

To show~\eqref{e:A.192}, we write the fist probability therein as
\begin{equation}
\label{e:A.95}
\frac{\bbP_x \big(\tau^\downarrow < \tau \big)
\bbP_x \big(\wt{X}(\tau^\downarrow) = z \,\big|\, \tau^\downarrow < \tau \big) \bbP_z(\wt{X}(\tau) = y)}{\bbP_x(\wt{X}(\tau) = y)} \,.
\end{equation}

Since $w \mapsto \bbP_w(\wt{X}(\tau)=y)$ is positive discrete harmonic in $\rmB(k+k^\gamma)$, it follows from the Harnack Inequality (see \cite[Theorem~6.3.8]{LL}) that for any $w,w' \in \rmB(k+k^\gamma-1)$ nearest neighbors,
\begin{equation}
\bbP_{w'}(\wt{X}(\tau)=y)/\bbP_{w}(\wt{X}(\tau)=y) \leq 1 + C\rme^{-k-k^\gamma} \leq \rme^{C\rme^{-k-k^\gamma}} \,.
\end{equation}
Taking a product of the left hand side above over all pairs of succeeding vertices in any shortest path from $x \in \partial_i \rmB(k+k^\gamma/2)$ to $z \in \partial \rmB(k+k^\gamma/4)$ gives
\begin{equation}
\label{e:A.94}
\bbP_{x}(\wt{X}(\tau)=y)/\bbP_{z}(\wt{X}(\tau)=y)
 \leq \rme^{C\rme^{-k-k^\gamma}\|z-x\|_1} 
\leq 1+C\rme^{-k^{\gamma}/2} \,.
\end{equation}
for some $c > 0$ and all $k$ large enough. Going along the path from $x$ to $z$ in the opposite direction, gives~\eqref{e:A.94} with $x$ and $y$ exchanged, so that altogether,
\FromS{There is already Lemma~\ref{lem:kernel} in the Appendix proving exactly this statement. We should move it to the Preliminaries and use it.}
\ToS{Not exactly - there $x=0$ no?}
\begin{equation}
\label{e:A.97}
\Big|\bbP_{w'}(\wt{X}(\tau)=y)/\bbP_{w}(\wt{X}(\tau)=y) - 1\Big| \leq C\rme^{-k^{\gamma}/2} \,.\,.
\end{equation}

At the same time, it follows from, e.g., Proposition 6.4.1 in~\cite{LL}, that
\begin{equation}
\label{e:A.98}
\bbP_x \big(\tau^\downarrow < \tau) = \frac{k+k^\gamma - \log |x| + O(\rme^{-k-k^\gamma/4})}{3k^\gamma/4}
= \frac{k^\gamma/2 + O(\rme^{-k-k^\gamma/4})}{3k^\gamma/4} 
= \frac23 \big(1+ O(\rme^{-k^\gamma})\big)
\end{equation}
Lemma~A.7 in~\cite{ballot} also shows that for any $x, x' \in \rmB(k+k^\gamma/2)$ and $z \in \rmB(k+k^\gamma/4)$,
\begin{equation}
\label{e:A.99}
\bbP_x \big(\wt{X}(\tau^\downarrow) = z \,\big|\, \tau^\downarrow < \tau \big)
= \bbP_{x'} \big(\wt{X}(\tau^\downarrow) = z \,\big|\, \tau^\downarrow < \tau \big) 
\big(1+O(\rme^{-k^\gamma/4}) \big) 
\end{equation}
Plugging the above estimates in~\eqref{e:A.95} one gets
\begin{equation}
\frac{\bbP_{x,y} \big(\wt{X}(\tau \wedge \tau^\downarrow) = z\big)}{
\bbP_{x',y'} \big(\wt{X}(\tau \wedge \tau^\downarrow) = z\big)} = 1+ O(\rme^{-ck^\gamma})
\end{equation}
as desired. 

For the last statement of the lemma, using~\eqref{e:A.97} in~\eqref{e:A.95} and summing over all $z \in \partial \rmB(k+k^\gamma/4)$, one gets $\bbP_{x,y} \big(\tau^\downarrow < \tau) = 
\bbP_{x} \big(\tau^\downarrow < \tau) (1+O(\rme^{-k^\gamma/2}))$. The result then follows directly from~\eqref{e:A.98}.
\end{proof}

\begin{lem}
\label{l:03.9}
Let $k > 0$ and vertices $x,x' \in \partial_i \rmB(k+k^\gamma/2)$ and $y,y' \in \partial \rmB(k+k^\gamma)$.There exists a coupling $\bbP_{x,x',y,y'}$ under which $\big(\tau, (\wt{X}(s) :\: s \in [0,\tau])\big)$ and $ \big(\tau', (\wt{X}'(s) :\: s \in [0,\tau'])\big)$ have the laws of $\big(\tau, (\wt{X}(s) :\: s \in [0,\tau])\big)$ under $\bbP_{x,y}$ and $\bbP_{x',y'}$ respectively. This coupling satisfies that for all $z \in \rmB(k)$,
\begin{equation}
\label{e:3.126}
\bbP_{x,x',y,y'} \big(\bfL_\tau(z) \neq \bfL'_{\tau'}(z)\big) \leq \rme^{-c k^\gamma} 
\end{equation}
and
\begin{equation}
\label{e:3.125a}
\bbP_{x,x',y,y'} \big(\bfL_\tau(z) \neq \bfL'_{\tau'}(z) \, \big|\, \bfL_\tau(z) \big) \leq \rme^{-c k^\gamma + C\bfL_\tau(z)} \,,
\end{equation}
where $c, C \in (0,\infty)$ are universal.
\end{lem}
\begin{proof}
The coupling is constructed explicitly as follows. 
Setting $Z^\uparrow_0 = x$,
we iteratively (and conditionally on all previous choices) draw $Z^\downarrow_1, Z^\uparrow_1, Z^\downarrow_2, Z^\uparrow_2, \dots, Z^\uparrow_M, Z^\downarrow_{M+1}$, so that $Z^\downarrow_m$ is chosen according to $\Pi^\downarrow_{Z^\uparrow_{m-1},y}$ and $Z^\uparrow_m$ is chosen according to $\Pi^\uparrow_{Z^\downarrow_{m}, y}$. We stop at the first iteration $M+1 \geq 1$ when $Z^\downarrow_{M+1} = y$. Then, for each $m=1, \dots, M+1$, we 
draw $\big(\tau^\downarrow_m, (\wt{X}^\downarrow_m(s) :\: s \in [0,\tau_m^\downarrow])\big)$ 
according to the measure $\bbP_{Z^\uparrow_{m-1}, y}\big(\cdot|\wt{X}(\tau \wedge \tau^\downarrow) = Z^\downarrow_m\big)$ and
for each $m=1, \dots, M+1$ we draw $\big(\tau^\uparrow_m, (\wt{X}^\uparrow_m(s) :\: s \in [0,\tau_m^\uparrow])\big)$ 
according to the measure $\bbP_{Z^\downarrow_m}\big(\cdot|\wt{X}(\tau^\uparrow) = Z^\uparrow_m\big)$\FromS{I think you cannot remove the $y$ in the conditioning, but the effect is negligible for $k$ large enough.}. \ToS{But isn't it the case that it doesn't matter what you do after exiting through $Z_m^\uparrow$}.
Lastly we set
$\wt{X}\big(\sum_{i=1}^{m-1} (\tau_i^\downarrow + \tau_i^\uparrow) + s \big) = \wt{X}_m^\downarrow(s)$
for all $s \in [0,\tau^\downarrow_m]$, $m=1, \dots, M+1$ and $\wt{X}\Big(\sum_{i=1}^{m-1} (\tau_i^\downarrow + \tau_i^\uparrow) + \tau_m^\downarrow + s\Big) = \wt{X}_m^\uparrow(s)$ for all $s \in [0,\tau^\uparrow_m]$ and $\tau := \sum_{m=1}^M (\tau_i^\downarrow + \tau_i^\uparrow) + \tau_{M+1}^\downarrow$.
We use the same procedure only with $x'$ and $y'$ in place of $x$ and $y$ to draw
$\big(\tau', (\wt{X}'(s) :\: s \in [0,\tau']\big)$.

Now, as long as $m \leq M \wedge M' +1$ we use the TV coupling of $\Pi^\downarrow_{Z^\uparrow_{m-1},y}$ and 
$\Pi^\downarrow_{Z^\uparrow_{m-1},y'}$ to jointly draw $(Z^\downarrow_m, Z'^\downarrow_m)$. Similarly, 
as long as $m \leq M \wedge M'$, we use the TV coupling of $\Pi^\uparrow_{Z^\downarrow_{m},y}$ and $\Pi^\uparrow_{Z^\downarrow_{m},y'}$ to jointly draw $(Z^\uparrow_m, Z'^\uparrow_m)$. Also, whenever $m \leq M \wedge M'$, 
$Z^\downarrow_m = Z'^\downarrow_m$ and $Z^\uparrow_m = Z'^\uparrow_m$, we draw $\big(\tau^\uparrow_m, (\wt{X}^\uparrow_m(s) :\: s \in [0,\tau_m^\uparrow])\big)$ and $\big(\tau'^\uparrow_m, (\wt{X}'^\uparrow_m(s) :\: s \in [0,\tau_m'^\uparrow])\big)$ jointly so that 
they are equal, and if $m \leq M \wedge M'$, $Z^\uparrow_{m-1} = Z'^\uparrow_{m-1}$ and $Z^\downarrow_{m} = Z'^\downarrow_{m}$, we draw $\big(\tau^\downarrow_{m}, (\wt{X}^\downarrow_{m}
(s) :\: s \in [0,\tau_{m}^\downarrow])\big)$ and $\big(\tau'^\downarrow_{m}, (\wt{X}'^\downarrow_{m}(s) :
\: s \in [0,\tau_m'^\downarrow])\big)$ jointly so that they are equal.

Denoting the probability measure on the probability space underlying the above construction by $\bbP_{x,x',y,y'}$, it is not difficult to see that by construction the laws of $\big(\tau, (\wt{X}(s) :\: s \in [0,\tau])\big)$ and $\big(\tau', (\wt{X}'(s) :\: s \in [0,\tau'])\big)$ are as required by the coupling in the statement of the lemma. At the same time, by construction again, recalling that we identify $y$ with $y'$, the left hand side of~\eqref{e:3.126} is at most
\begin{equation}
	\begin{split}	
		\bbP \big(\exists &  m \leq M \wedge M'+1 :\: 
		Z^\downarrow_m \neq Z'^\downarrow_m \,
		\text{ or }\ 
		\exists m \leq M \wedge M' :\: 
		Z^\uparrow_m \neq Z'^\uparrow_m
		\big) \\
		& \leq 
		\bbP_{x,x',y,y'} (M > k)
		+ \bbP_{x,x',y,y'} (M' > k) \\
		& + \bbP_{x,x',y,y'} \big(\exists m \leq M \wedge M' \wedge k + 1 :\: 
		Z^\downarrow_m \neq Z'^\downarrow_m \,,
		\text{ or }\ 
		\exists m \leq M \wedge M' \wedge k  :\: 
		Z^\uparrow_m \neq Z'^\uparrow_m
		\big) \\
		& \leq 
		2(3/4)^k + (2k+1) \rme^{-k^\gamma} \,,
	\end{split}
\end{equation}
where the last inequality follows from Lemma~\ref{l:A.17}. This shows~\eqref{e:3.126}.

To show~\eqref{e:3.125a} it is sufficient to prove that for all $u \geq 0$,
\begin{equation}
\label{e:3.125}
\bbP_{x,x',y,y'} \big(\bfL_\tau(z) \neq \bfL'_{\tau'}(z) \, \big|\, \bfL_\tau(z) = u\big) \leq \rme^{-c k^\gamma + Cu} \,,
\end{equation}
where for $u > 0$, conditioning is interpreted in the usual way for random variables whose law is absolutely continuous with respect to the Lebesgue measure.
Indeed, for $u = 0$, we may write,
\begin{equation}
\bbP_{x,x',y,y'} \big(\bfL_\tau(z) \neq \bfL'_{\tau'}(z) \, \big|\, \bfL_\tau(z) = 0 \big)
\leq \frac{\bbP_{x,x',y,y'} \big(\bfL_\tau(z) \neq \bfL'_{\tau'}(z)\big)}
{\bbP_{x,y}(\bfL_\tau(z) = 0)} \,.
\end{equation}
Then~\eqref{e:3.125} follows from~\eqref{e:3.126} and the fact that the denominator is at least $1/6$ by the third part of Lemma~\ref{l:A.17}.

If $u > 0$, we condition on the number of visits $\bfl_\tau(z)$ of $\wt{X}$ to $z$ by time $\tau$ and write the left hand side in~\eqref{e:3.125} as
\begin{multline}
\frac{\sum_{j=1}^\infty \bbP_{x,x',y,y'}\big(\bfl_\tau(z) = j) \bbP_{x,x',y,y'} \big(\bfL_\tau(z) \neq \bfL'_{\tau'}(z) \, \big|\, \bfl_\tau(z) = j \big)
\bbP_{x,x',y,y'} \big(\bfL_\tau(z) \in \rmd u \, \big|\, 
	\bfl_\tau(z) = j \big)} 
{\bbP_{x,x',y,y'}\big(\bfL_\tau(z) \in \rmd u\big)} \\
\leq
\frac{\sum_{j=1}^\infty \bbP_{x,x',y,y'} \big(\bfL_\tau(z) \neq \bfL'_{\tau'}(z) \big)
 \bbP_{x,y} \big(\bfL_\tau(z) \in \rmd u \, \big|\, 
	\bfl_\tau(z) = j \big)}
{\bbP_{x,y}\big(\bfL_\tau(z) \in \rmd u\big)} 
\\
\leq
\rme^{-ck^\gamma}
\frac{\sum_{j=1}^\infty \frac{1}{(j-1)!} \lambda^j u^{j-1} \rme^{-\lambda u}} 
{p_{x,y}(z) \lambda (1-p_{z,y}(z)) \rme^{-u \lambda (1-p_{z,y}(z))}}z
\leq 
p_{x,y}(z)^{-1}(1-p_{z,y}(z))^{-1} \rme^{u \lambda}
\rme^{-ck^\gamma} \,,
\end{multline}
where $p_{x,y}(z)$ is a shorthand for $\bbP_{x,y} (\bfL_\tau(z) > 0)$ and $\lambda = 2/\pi$ is the rate of jumps away from $z$. Above we have used the independence of the time spent at $z$ and the event $\{\bfL_\tau(z) \neq \bfL'_{\tau'}(z)$ conditional on $\bfl_\tau(z)$. A straightforward computation, similar to the one in the proof of the third part of Lemma~\ref{l:A.17}, shows that $p_{x,y}(z) \geq c k^{\gamma - 1}$ and $1-p_{z,y}(z) \geq c/k$. The result follows.
\end{proof}

\begin{proof}[Proof of Proposition~\ref{l:4.5}]
The upper bound (with $l$ and $y$ in place of $k$ and $x$) follows from Proposition~\ref{prop:Abulk2} and the union bound. Indeed, summing the probability on the left hand side of~\eqref{eq:proofAbulk3} over all $x \in \rmB(y;l)$ gives a quantity which is smaller than the desired upper bound.

For the lower bound, if $m \geq k$, then we may again use Proposition~\ref{prop:Abulk2}. This time, we can lower bound the probability in~\eqref{e:4.11} (with $l$ and $y$ in place of $k$ and $x$) by the probability in~\eqref{eq:proofAbulk3b}
 with $x=y$. Then it can easily checked that the left hand side in~\eqref{eq:proofAbulk3b} is larger than the the desired lower bound,

The lower bound when $m < k$ is much more involved and the rest of the proof is devoted to this case. Without loss of generality we may assume that $x=0$ and accordingly abbreviate $\rmB(r) \equiv \rmB(0;r)$. The argument is based on comparison with the DGFF on $\rmB(k+k^\gamma)$ via the Isomorphism statement of Theorem~\ref{t:103.1}. It begins by writing for $t'$ such that $\sqrt{t'} = \sqrt{2}k + m + s k^\gamma$ and $s$ to be chosen later,
\begin{equation}
\begin{split}
		\bbP \Big(\min_{\rmB(k+k^\gamma)} & h_{\rmB(k+k^\gamma)} = \min_{\rmB(k)}  \, h_{\rmB(k+k^\gamma)}\in - \sqrt{t'} + (\sqrt{u}, -\sqrt{u}) \Big) \\
& = \sum_{x \in \rmB(k)} \int_{w = -\sqrt{t'} - \sqrt{u}}^{-\sqrt{t'} + \sqrt{u}}
		\bbP \big(h_{\rmB(k+k^\gamma)}(x) \in \rmd w \big)
		\bbP \Big(\min_{\rmB(k+k^\gamma)} h_{\rmB(k+k^\gamma)} \geq w \, \Big| \,
			h_{\rmB(k+k^\gamma)}(x) = w \big)
		\\
& \geq c \rme^{2k} \frac{\sqrt{u}}{\sqrt{k+k^\gamma}} \rme^{-\frac{(\sqrt{t'} + \sqrt{u})^2}{k+k^\gamma-C}} 
\frac{\sqrt{t'} - \sqrt{u} - m_{k+k^\gamma}}{k+k^\gamma} \\
	& \geq c\, \rme^{2k-2(k+k^\gamma)} \rme^{-2\sqrt{2}(\sqrt{2}k + m + s k^\gamma- m_{k+k^\gamma})}
		\rme^{-\frac{(\sqrt{2}k + m + s k^\gamma- m_{k+k^\gamma})^2}{k+k^\gamma}} \\
	& \geq c' \rme^{(2-2\sqrt{2} s) k^\gamma-2\sqrt{2} m - \frac{m^2}{k+k^\gamma}} \,,
	\end{split}
\end{equation}
where $c,c' > 0$ depend on $u$ and $k$ is taken large enough. Above, to estimate the first probability, we used that $h_{\rmB(k+k^\gamma)}$ is a centered Gaussian with variance $\frac12 G_{B(k+k^\gamma)}$ and then applied Lemma~\eqref{l:103.2} to estimate the variance. For the second probability we used Theorem~1.1 in~\cite{ballot}.

It then follows from~\eqref{e:3.1} that 
\begin{equation}
	\bbP \Big(\exists x \in \rmB(k) :\: L'_{t'}(x) \leq u \Big) \geq \rme^{(2-2\sqrt{2} s)k^\gamma-2\sqrt{2} m - \frac{m^2}{k+k^\gamma}} \,,
\end{equation}
where $L'_{t'}(x)$ is the local time at $x$ for a random walk $X'$ on $\widehat{\rmB(k+k^\gamma)}$ ``when the local time at $\partial \rmB(k+k^\gamma)$ is $t'$''.
This is the desired statement, albeit with number of downcrossings replaced by the time at the boundary of $\partial \rmB(k+k^\gamma)$. To remedy this, we shall couple the random walk $X$ on $\wh{\rmD}_n$ under the conditional measure in~\eqref{e:4.11} and the random walk $X'$ on $\widehat{\rmB(k+k^\gamma)}$ such that with sufficiently high probability the local times on $\rmB(k)$ of $X$ and $X'$ agree.

To construct the coupling, we draw $X$ with all $(x;k)$-excursions erased, conditional on $\{\wh{N}_t(0;k) = \sqrt{2}k + m\}$, that is according to the restriction of $\bbP(-|\wh{N}_t(0;k) = \sqrt{2}k + m)$ to $\cF(0,k)$. We denote the entry and exit points in each downcrossing between $\rmB(k+k^\gamma)$ and $\rmB(k+k^\gamma/2)$
 that $X$ makes by $(Y_j, Z_j)_{j=1}^{\wh{N}_t(0;k)}$ and note that they are indeed measurable w.r.t. $\cF(0,k)$. Next we draw, independently, the number of downcrossings $\wh{N}'_{t'}(0;k)$ that $X'$ makes by time $t'$ and the entry and exit point $(Y'_i, Z'_i)_{i=1}^{\wh{N}'_t(0;k)}$ in each such downcrossing. 
 Then, for each $j=1,\dots, \wh{N}_t(0;k) \wedge \wh{N}'_{t'}(0;k)$, we draw the excursions following the $j$-th downcrossing of both $X$ and one for $X'$ according to the coupling law $\bbP_{Y_j, Y'_j, Z_j, Z'_j}$, which is given by Lemma~\ref{l:03.9}. We then draw the remaining excursions for either $X$ or $X'$ given their entry and exit points. All drawing are done independently of each other.
 
 It is not difficult to see that under this coupling, which we shall henceforth denote by $\wh{\bbP}$, the local times $L_t(\rmB(k))$ and $L'_{t'}(\rmB(k))$ have the right marginals. In particular,
\begin{equation}
\label{e:3.32a}
	\wh{\bbP} \Big(\exists x \in \rmB(k) :\: L'_{t'}(x) \leq u \Big) \geq \rme^{(2-2\sqrt{2} s)k^\gamma-2\sqrt{2} m - \frac{m^2}{k+k^\gamma}} \,,
\end{equation}
At the same time, by conditioning on $\wh{N}'_{t'}(0;k)$, we may write,
\begin{equation}
\label{e:3.133}
	\begin{split}
\wh{\bbP} \Big(\sqrt{\wh{N}'_{t'}(0;k)}  < & \sqrt{2}k+m \,,\,\, \exists x \in \rmB(k) :\: L'_{t'}(x) \leq u \Big) \\
& = \sum_{j=1}^{(\sqrt{2}k+m)^2} \bbP \Big( \wh{N}'_{t'}(0;k) = (\sqrt{2}k+m)^2-j \Big)  \\
& \qquad \qquad \qquad \qquad \bbP \Big(\exists x \in \rmB(k) :\: L'_{t'}(x) \leq u \,\Big|\, \wh{N}'_{t'}(0;k) = (\sqrt{2}k+m)^2-j \Big) \,.
\end{split}
\end{equation}
Then, thanks to Proposition~\ref{prop:Abulk}, the first term in the sum can be bounded from above by
\begin{multline}
2\exp \Bigg(-\frac{\Big(\sqrt{t'} - \sqrt{(\sqrt{2}k+m)^2+j}\Big)^2}{k^{\gamma} + C} \Bigg)
\leq 
C \exp \Big(-c \frac{\big(t' - (\sqrt{2}k+m)^2+j\big)^2}{t' k^{\gamma}} \Big)\\
\leq C \exp \Big(-c \frac{\big((\sqrt{2}k+m+sk^\gamma)^2 - (\sqrt{2}k+m)^2+j\big)^2}{k^{2+\gamma}} \Big) \,,
\end{multline}
which is at most $\rme^{-c s^2 k^\gamma - c' j^2/k^{2+\gamma}}$, for some $c, c' > 0$ and $C < \infty$.
For the second term in the sum in~\eqref{e:3.133}, we may use Proposition~\ref{prop:Abulk2} and the Union bound to upper bound it by
\begin{multline}
1_{\{j > k^2\}} + C \rme^{2k} \exp \bigg(-2(k+k^\gamma) 
- 2\sqrt{2} \Big(\sqrt{(\sqrt{2}k + m)^2 - j} - \sqrt{2} (k+k^\gamma) \Big) \\
- \frac{1}{k+k^\gamma}\Big(\sqrt{(\sqrt{2}k + m)^2 - j} - \sqrt{2} (k+k^\gamma) \Big)^2 + 
 C k^\gamma \bigg) \\
\leq 1_{\{j > k^2\}} + \rme^{C k^\gamma - 2\sqrt{2}m -\frac{m^2}{k+k^\gamma}+ c' \frac{j}{k}} \,.
\end{multline}
Plugging the last two bounds in the summation in~\eqref{e:3.133} and computing the sum gives 
\begin{multline}
\label{e:3.36a}
\wh{\bbP} \Big(\sqrt{\wh{N}'_{t'}(0;k)} < \sqrt{2}k+m \,,\,\, \exists x \in \rmB(k) :\: L'_{t'}(x) \leq u \Big) \\
\leq \rme^{(2\sqrt{2}s -c s^2 + C) k^\gamma} \rme^{(2-2\sqrt{2}s)k^\gamma -2\sqrt{2}m -\frac{m^2}{k+k^\gamma}}  
\leq \rme^{-c k^\gamma} \rme^{(2-2\sqrt{2}s)k^\gamma -2\sqrt{2}m -\frac{m^2}{k+k^\gamma}} \,,
\end{multline}
for some $c > 0$, where the last inequality can be made possible by choosing $s$ large enough.

On the other hand, for any $z \in \rmB(k)$, denoting by $\bfL^{(j)}(z)$ and $\bfL'^{(j)}(z)$ the addition to the local time at $z$ in the $j$-th excursion of $X$ and $X'$ respectively under the coupling, 
\begin{multline}
\wh{\bbP} \Big(L'_{t'}(z) \leq u  \,,\,\, L_{t}(z) > u \,,\,\,  \sqrt{\wh{N}'_{t'}(0;k)} \geq \sqrt{2}k + m\Big) \\
\leq \wh{\bbP} \big(L'_{t'}(z) \leq u \big) \sum_{j=1}^{(\sqrt{2}k + m)^2} \wh{\bbP} \big(\bfL^{(j)}(z) \neq \bfL^{'(j)}(z) \,\big|\, L'_{t'}(z) \leq u \big) \\
\leq C(k+m)^2\rme^{-ck^\gamma + Cu} \bbP \big(L'_{t'}(z) \leq u \big) \,.
\end{multline}
Above we first used the restrictions on the number of downcrossings that $X'$ makes to deduce that the local time contribution to $z$ must be different in one of the first $\sqrt{2}k+m$ excursions. We then drop the latter restrictions and use the product rule and the union bound. To upper bound summand $j$, we first conditioned on 
$\wh{N}'_{t'}(0;k)$, $(Y'_i, Z'_i)_{i=1}^{\wh{N}'_t(0;k)}$ and $(\bfL^{'(j)}(z) :\: j'=1, \dots, \wh{N}'_{t'}(0;k))$, which by the independence of excursions at different downcrossings, amounts to conditioning on $Y'_j$, $Z'_j$ and $\bfL'^{(j)}(z)$ only, and then applied Lemma~\ref{l:03.9}, noting that on $\{L'_{t'}(z) \leq u\}$ we must also have 
$\{\bfL'^{(j)}(z) \leq u\}$.

Using Lemma~\ref{lem:nhprob}, with $n=k+k^\gamma$ and $t'$ in place of $t$, we get 
\begin{equation}
\bbP (L'_{t'}(z) \leq u) \leq C\rme^{-2k+(2-2\sqrt{2} s) k^\gamma -2\sqrt{2}m -\frac{m^2}{k+k^\gamma}} \,,
\end{equation}
 so that summing over all vertices in $\rmB(k)$ we get
\begin{multline}
\label{e:3.39a}
\wh{\bbP} \Big(\sqrt{\wh{N}'_{t'}(0;k)} \geq \sqrt{2}k + m \,,\,\, \exists z \in \rmB(k) :\: L'_{t'}(z) \leq u \,,\,\, L_{t}(z) > u\Big) \\
\leq C \rme^{-c k^\gamma} \rme^{(2-2\sqrt{2} s) k^\gamma - 2\sqrt{2}m - \frac{m^2}{k+k^\gamma}} \,.
\end{multline}
Above we have used that $m \leq k$.

When the event in~\eqref{e:3.32a} occurs but the ones in~\eqref{e:3.36a} and~\eqref{e:3.39a} do not, then the event
\begin{equation}
\big\{\exists z \in \rmB(k) :\: L_{t}(z) \leq u\Big\}
\end{equation}
must also occur. The bounds in~\eqref{e:3.32a},~\eqref{e:3.36a} and~\eqref{e:3.39a} show that this must happen with $\wh{\bbP}$ probability at least
\begin{equation}
\rme^{(2-2\sqrt{2} s)k^\gamma-2\sqrt{2} m - \frac{m^2}{k+k^\gamma}} - 
C\rme^{-ck^\gamma} \rme^{(2-2\sqrt{2} s)k^\gamma-2\sqrt{2} m - \frac{m^2}{k+k^\gamma}} 
\geq \rme^{-Ck^\gamma-2\sqrt{2} m - \frac{m^2}{k}} \,,
\end{equation}
for some $C > 0$, as desired.
\end{proof}

\subsection{Discrete Gaussian free field preliminaries}
\label{a:1}	
In this subsection we include proofs for DGFF related statements. In what follows we set
\begin{equation}
	R_n := n-2\log n -1 \,.
\end{equation}
We start with,
\begin{proof}[Proof of Lemma~\ref{l:3.25}]
	From~\eqref{e:3.19} it follows that
	\begin{equation}
		\bbE \big(\varphi_{\cU_n, \wt{\cV}_{n'}}(x)\big)^2 = \bbE \big(h_{\cU_n}(x)\big)^2 - \bbE \big(h_{\wt{\cV}_{n'}}(x)\big)^2 \,. 
	\end{equation}
	Since $\log \rmd \big(x, (\cU_n)^\rmc\big) < n + C'(\cU)$ and $\log \rmd \big(x, (\wt{\cV}_{n'})^\rmc\big) > n' - 1/q$, the first part is a consequence of Lemma~\ref{l:103.2}. For the second part, we recall that for all $x,y \in \wt{\cV}_{n'}$, 
	\begin{equation}
		\label{e:3.25}
		\bbE \big(\varphi_{\cU_n, \wt{\cV}_{n'}}(x) - \varphi_{\cU_n, \wt{\cV}_{n'}}(y)\big)^2 =  
		\bbE \big(h_{\cU_n}(x) - h_{\cU_n}(y) \big)^2 - \bbE \big(h_{\wt{\cV}_{n'}}(x) - h_{\wt{\cV}_{n'}}(y) \big)^2 \,.
	\end{equation}
	Since $\log \rmd(\{x,y\}, \wt{\cV}_{n'})  \geq \log \rmd(\{x,y\}, \cU_n) \geq n' - q$ for all $x,y \in (\wt{\cV}_{n'})^{n'-q}$, it follows from the third part of Lemma~\ref{l:103.2} that 
	\begin{equation}
		\label{e:3.28}
		\bbE \big(\rme^{(k-l/2)} \varphi_{\cU_n, \wt{\cV}_{n'}}(y) - \rme^{(k-l/2)} \varphi_{\cU_n, \wt{\cV}_{n'}}(y')\big)^2 \leq C \rme^{-l} |y-y'| \,,
	\end{equation}
	for all $y,y' \in \rmB(x;l)$ and some $C=C(\cU,\cV,q)$. An application of Fernique Majorization Lemma together with the Borell-Tsirelson Inequality 
	for the field $\big(\rme^{(k-l/2)} \varphi_{\cU_n, \wt{\cV}_{n'}}(y) ;\; y \in \rmB(x;l)\big)$ then gives~\eqref{e:3.24.2}. See e.g. Lemma~3.8 in~\cite{BL3} for more details. 
\end{proof}
\begin{proof}[Proof of Lemma~\ref{l:103.1}]
	By disjointness of $\rmB(z;k+1)$ for $z \in \rmZ$, the size of $\rmZ$ and hence the size of $\rmY$ must be at most $C\rme^{2(n-k)}$. It follows from the Gibbs-Markov decomposition and Lemma~\ref{l:3.25} that the collections $\big(|\varphi_{\cU_n, \rmV}(y)| :\: y \in \rmY \big)$ and 
	$\big(\max_{x \in \rmB(y;l)} 
	\rme^{(k-l)/2} \big|\varphi_{\cU_n, \rmV} (x) - \varphi_{\cU_n, \rmV} (y) \big| :\: y \in \rmY \big)$ include random variables whose tail decays at least as fast as a Gaussian variance $\tfrac12 (n-k) + C$ and $C$ respectively, for some $C > 0$. Both assertions then follow by a standard union bound.
\end{proof}

\begin{proof}[Proof of Proposition~\ref{p:3.1}]
	This is a reformulation of Theorem 1.1 in~\cite{Ding}.
\end{proof}

For Proposition~\ref{p:3.2} we shall need the following lemma, which is taken from \cite{BZing}:
\begin{lem}
	\label{l:A.1}
	For any nice set $\cU$ there exist some constants $C, c > 0$ depending only on $\cU$ such that for all $u \in \bbR$, $n > 0$ and $\rmV \subset \cU_n$, 
	\begin{equation}
		\limsup_{n \to \infty}
		\bbP \Big(\max_{x \in \rmV} h_{\cU_n}(x) - m_n > u \big) \leq C \bigg(\frac{|\rmV|}{|\cU_n|}\bigg)^{1/2} \rme^{-c u} \,.
	\end{equation}
\end{lem}
\
\begin{proof}[Proof of Proposition~\ref{p:3.2}]
	Tightness in $n$ for the size of the set $\big\{x \in \rmD_n :\: h_n(x) \leq -m_n + \sqrt{u}\big\}$ which includes $\rmG_n(u)$ was shown in~\cite{DingZeitouni}. The second statement follows from tightness of the centered minimum of $h_n$ as shown in Proposition~\ref{p:3.1}, together with,
	\begin{equation}
		\bbP \big(\min_{\rmD_n \setminus \rmD_n^\circ} h_n < -m_n + u \big) \leq \bigg(\frac{\big|\rmD_n \setminus \rmD_n^\circ\big|}{\big|\rmD_n|}\bigg)^{1/2} \rme^{Cu} \underset{n \to \infty} \longrightarrow 0\,, 
	\end{equation}
	which holds thanks to Lemma~\ref{l:A.1}.
\end{proof}

For the proof of the remaining DGFF preliminaries, we shall make use of a decomposition of the DGFF along concentric annuli around a given point inside its domain. This technique, which is sometimes referred to as the {\em concentric decomposition} of the field, was introduced in~\cite{BL3} and is summarized in the following proposition:
\begin{prop}
	\label{prop-concentric}
	Fix a nice set $\cU$ and let $n \geq 1$. Set $n' := \lfloor \log \rmd(0, \cU_n^\rmc) \rfloor$ and define:
	\begin{equation}
		\Delta_k := 
		\left\{
		\begin{array}{lll}
			\{0\}\,, & \quad & k = 0 \\
			\rmB(0;k)\,, & & k = 1, \dots, n'-1 \\
			\cU_n \,, & & k = n' \, . \\
		\end{array}
		\right. \,.
	\end{equation}
	Then on the same probability space as $h_{\cU_n}$ we may define whole plane random fields:
	\begin{equation}
		\label{E:3.10a}
		\bigl\{\varphi_k \colon k=0,\dots, n'\bigr\}\cup\bigl\{\chi_k\colon k=0,\dots,n'\bigr\}\cup\bigl\{h_k'\colon k=0,\dots,n'\bigr\} \,,
	\end{equation}
	such that
	\settowidth{\leftmargini}{(11)}
	\begin{enumerate}
		\item All fields are affine transformations of $h_{\cU_n}$, have centered Gaussian laws and are independent of each other. 
		\item For $k=1,\dots n'$, 
		\begin{equation}
			\label{E:3.10}
			\varphi_k\laweq \bbE\bigl({\Delta^k}\big|\sigma(h_{\Delta^k}\colon z\in \partial \Delta^{k-1})\bigr)
		\end{equation}
		while, for $k=0$,
		\begin{equation}
			\label{E:3.11}
			\varphi_0=\varphi_0(0)\1_{\{0\}},\quad\text{where}\quad\varphi_0(0)\laweq \NN(0,1).
		\end{equation}
		In addition, a.e.\ sample path of $\varphi_k$ is discrete harmonic on $\Delta^k\smallsetminus\partial \Delta^{k-1}$ and zero on $\Z^2\smallsetminus \Delta^k$. 
		\item For all $k=0,\dots,n'$ and $x \in \bbZ^2$, 
		\begin{equation}
			\label{E:3.22r}
			\chi_k(x) = \varphi_k(x)-\bigl(1+\frb_k(x)\bigr)\varphi_k(0) \,,
		\end{equation}
		where 
		\begin{equation}
			\label{E:3.14q}
			\frb_k(x):=\frac1{\Var(\varphi_k(0))}\,\bbE\Bigl(\varphi_k(0)\bigl(\varphi_k(x)-\varphi_k(0)\bigr)\Bigr) \, .
		\end{equation}
		\item For $k=1,\dots,n'$, 
		\begin{equation}
			\label{E:3.12}
			h'_k\laweq h_{\Delta^k\smallsetminus \overline{\Delta^{k-1}}}
		\end{equation}
	\end{enumerate}
	Moreover, for all $x \in \bbZ^2$,
	\begin{equation}
		\label{E:3.14}
		h_{\cU_n}(x)=\sum_{k=0}^{n'}\varphi_k(x)+\sum_{k=0}^{n'} h'_k(x) = 
		\sum_{k=0}^{n'}\bigl(1+\frb_k(x)\bigr)\varphi_k(0)+\sum_{k=0}^{n'}\chi_k(x)+\sum_{k=0}^{n'} h_k'(x).
	\end{equation}
\end{prop}
\begin{proof}
	See~\cite{BL3}.
\end{proof}

Henceforth unless stated otherwise, all constants depend on $\cU$ and can be taken to be the same for all of its translates. Also, for $n \geq 1$, we suppose that the fiels in Proposition~\ref{prop-concentric} are defined along with $h_{\cU_n}$ on thee same probability space.
Basic properties of the fields in the decomposition are given by: 
\begin{prop} 
	\label{p:A.2}
	Let $\cU$ be a nice set. There exist $C, c \in (0,\infty)$ such that for all $n \geq 0$, 
	\begin{enumerate}
		\item If $0 \in \cU_n^{R_n}$ then $n-2 \log n - C \leq n' \leq n+C$.
		\item $\Var\bigl(\varphi_k(0)\bigr) \in (c, C)$ for all $k = 0,\dots, n'$.
		\item $\frb_k$ is discrete harmonic in~$\Z^2\smallsetminus(\partial \Delta^k\cup\partial \Delta^{k-1})$ and satisfies $\frb_k(x)\ge-1$ for all~$x\in\Z^2$. Moreover,
		\begin{equation}
			\label{E:3.22q}
			\frb_k(x)=-1,\qquad x\not\in \Delta^k \ , \ \ 
			\bigl|\frb_k(x)\bigr|\le C\frac{|x|}{\rmd(0,\partial \Delta^{k})},\qquad x\in \Delta^k.
		\end{equation}
		\item $\chi_k$ is discrete harmonic on $\Delta^{k-1}$ and vanishes outside of $\Delta^k$. Moreover, for all $k \leq n'$, $0\le\ell\le k-2$ and $\lambda > 0$, 
		\begin{equation}
			\label{E:3.33}
			\bbE\Bigl(\,\max_{x\in \ol{\Delta^\ell}}\chi_k(x)\Bigr)\le C \rme^{(\ell-k)}
		\end{equation}
		and
		\begin{equation}
			\label{E:3.34}
			\bbP\biggl(\,\Bigl|\,\max_{x\in \ol{\Delta^\ell}}\chi_k(x)-\bbE\Bigl(\,\max_{x\in \ol{\Delta^\ell}}\chi_k(x)\Bigr)\Bigr|>\lambda\biggr)
			\le 2\texte^{-c\rme^{2(k-\ell)}\lambda^2}.
		\end{equation}
	\end{enumerate}
\end{prop}
\begin{proof}
	The proof is similar as in those of Lemmas~3.6 -- 3.10 and Lemma~4.23 in~\cite{BL3}. The only difference is that here we allow the distance between $\Delta^{n'-1}$ and $(\Delta^{n'})^\rmc$ to be sublinear in $n$. By construction, this distance is up to a constant at most the diameter of $\Delta^{n'-1}$. Then, thanks to Part~2 of Lemma~\ref{l:103.2}, $\bbE \varphi_{n'}(x) \varphi_{n'}(0)$ as appears in (3.31) in~\cite{BL3} is still upper bounded by a constant (that depends only on $\cU$, uniformly with respect to its translates). These two facts validate the proofs of the above mentioned lemmas also in our case.
\end{proof}

Next, we set for $k=0, \dots, n'+1$,
\begin{equation}
	\label{E:RW}
	S_k:=\sum_{\ell=0}^{k-1}\varphi_\ell(0),\qquad k\ge0. \,,
\end{equation} 
The last two propositions show that $(S_k)_{k=0}^{n'+1}$ is a random walk with non-homogeneous centered Gaussian steps and that for $k = 0,\dots, n'$, the field
$h_{\cU_n}(\cdot)$ on $\Delta_k \setminus \ol{\Delta_{k-1}}$ is equal to $S_{n'+1} - S_k$, plus terms which are well behaved. To control the latter, we follow~\cite{BL3} and set 
\begin{equation}
	\label{E:4.1}
	\Theta_{k}(\ell):=1+\log\bigl(1+[k\vee(\ell\wedge(n'-\ell))]\bigr)
	\ ; \quad  k,\ell=1,\dots,n' \,.
\end{equation}
We then introduce,
\begin{defn}[Control Variable]
	\label{d:A.4}
	Let $K$ be the minimal (random) ~$k \geq 2$ such that the following holds:
	\begin{enumerate}
		\item[(1)] For each~$\ell=0,\dots,n'$,
		\begin{equation}
			\bigl|\varphi_\ell(0)\bigr|\le\Theta_{k}(\ell),
		\end{equation}
		\item[(2)] For each $\ell=2,\dots,n'$ and each $r=0,\dots,\ell-2$,
		\begin{equation}
			\max_{x\in \Delta^r}\,\bigl|\chi_\ell(x)\bigr|\le \rme^{(r-\ell)}\Theta_{k}(\ell),
		\end{equation}
		\item[(3)] For each $\ell=1,\dots,n'$,
		\begin{equation}
			\label{E:4.4b}
			\Bigl|\,\max_{x\in \Delta^\ell\smallsetminus \Delta^{\ell-1}}\bigl(\chi_\ell(x)+\chi_{\ell-1}(x)+h_\ell'(x)\bigr)-m_{\ell}\Bigr|\le\Theta_{k}(\ell)^2
		\end{equation}
		and for each $\ell=1,\dots,n'-1$,
		\begin{equation}
			\label{E:4.4}
			\Bigl|\,\max_{x\in \Delta^\ell\smallsetminus \Delta^{\ell-1}}\bigl(\chi_\ell(x)+\chi_{\ell-1}(x)+\chi_{\ell+1}(x)+h_\ell'(x)\bigr)-m_{\ell}\Bigr|\le\Theta_{k}(\ell)^2.
		\end{equation}
		\item[(4)]
		For each $\ell=1,\dots,n'-2$ and each $y \in \bbX_{\ell-1}$ such that $0 \in \rmB(y;l-1)$, abbreviating $\eta :=  \sum_{\ell'=\ell}^{\ell+2} \big(\chi_{\ell'} + h'_{\ell'} \big)$, we have
		\begin{equation}
			\label{e:A.18}
			\ol{\eta^2}(y; \ell)/\ell  
			\le\Theta_{k}(\ell)^2 
		\end{equation}
		and
		\begin{equation}
			\label{e:A.19}
			\big|\ol{\eta}(y; \ell)\big|
			\le\Theta_{k}(\ell) \,.
		\end{equation}

	\end{enumerate}
\end{defn}

While conditions (1)-(3) in the definition of $K$ have appeared in~\cite{BL3}, condition (4) is new and to upper bound the probability of violating it, we shall need the following two lemmas:
\begin{lem}
	\label{l:A.6.5}
	Let $\cU$ be a nice set. There exists $C,c \in (0,\infty)$ such that for all $n \geq 0$, $k \geq 0$ and $x \in \bbZ^2$ with $\rmB(x,k) \subset \cU_n$ and all $t \geq 0$,
	\begin{equation}
		\label{e:A.6}
		\bbP \Big(\ol{h_{\cU_n}^2}(x;k) > t\Big) \leq C\rme^{-c \frac{t}{n}} \,.
	\end{equation}
\end{lem}
\begin{proof}
	By Jensen's Inequality, for all $\theta \in \bbR$, 
	\begin{equation}
		\begin{split}
			\bbE \exp\Big(&\theta \ol{h_{\cU_n}^2}(x;k)\Big)
			\leq  \bbE  \Big(\ol{\exp\big(\theta \big(h_{\cU_n}\big)^2\big)}(x;k)\Big)
			=  \ol{\Big(\bbE \exp\big(\theta \big(h_{\cU_n}\big)^2\big)\Big)}(x;k) \\
			& \leq \max_{y \in \partial \rmB_k(x)}  \bbE \exp\big(\theta \big(h_{\cU_n}(y)\big)^2\big) 
			\leq \max_{y \in \partial \rmB_k(x)} \Big(1-2\theta \bbE \big(h_{\cU_n}(y)\big)^2 \Big)^{-1/2} \
			\leq \big(1-2\theta n  \big)^{-1/2} \,.
		\end{split}
	\end{equation}
	The last inequality holds since the variance of $h_{\cU_n}$ is uniformly bounded by $n/2+C$ and the inequality before is true 
	since $(h_{\cU_n}(y))^2$ has the Gamma distribution with shape parameter $\alpha=1/2$ and 
	rate parameter $\lambda = (2\bbE \big(h_{\cU_n}(y)\big)^2)^{-1}$.
	The result then follows by from the exponential Chebyshev's inequality applied with $\theta = 1/(3n).$
\end{proof}
\begin{lem}
	\label{l:A.6}
	Let $\cU$ be a nice set and $q,q' \geq 0$. There exists $C, c \in (0,\infty)$ such that for all $n,k \geq 0$ satisfying $q \leq k \leq n'-q'$ and all $x \in \bbZ^2$ with $\Delta^{k-q} \subseteq \rmB(x;k) \subseteq \Delta^{k+q'}$, setting
	\begin{equation}
		\eta(y) := \sum_{\ell=k-q+1}^{k+q'} \big(\chi_\ell(y) + h'_\ell(y) \big)
		\ ;\quad y \in \bbZ^2\,,
	\end{equation}
	we have for all $t > 0$,
	\begin{equation}
		\label{e:A.29}
		\bbP \Big(\ol{\eta^2}(x;k) > t\Big) \leq C\rme^{-c \frac{t}{k}} \,.
	\end{equation}
	and
	\begin{equation}
		\label{e:A.29a}
		\bbP \Big(\big|\ol{\eta}(x;k)\big| > t\Big) \leq C\rme^{-ct^2} \,.
	\end{equation}
\end{lem}
\begin{proof}
	In analog to~\eqref{E:3.14}, for $y \in \Delta^{k+q'} \setminus \Delta^{k-q}$ we have,
	\begin{equation}
		\label{e:A.34}
		\eta(y) = h_{\rmB(0,k+q')}(y) - \sum_{\ell=k-q+1}^{k+q'} \big(1+\frb_\ell(y)\big)\varphi_\ell(0) \,.
	\end{equation}
	Then, using the uniform bound~\eqref{E:3.22q} on $\frb_\ell$ we get,
	\begin{equation}
		\ol{\eta^2}(x;k) \leq 2\ol{\big(h_{\rmB(0,k+q'+1)}\big)^2}(x;k) + C \sum_{\ell=k-q+1}^{k+q'}\varphi_\ell(0)^2 \,,
	\end{equation}
	for some $C < \infty$.  Then the first statement holds thanks to Lemma~\ref{l:A.6.5},  the uniform Gaussian tails of $\varphi_\ell(0)$ and the union bound. 
	
	Turning to the second statement, from~\eqref{e:A.34} we also get
	\begin{equation}
		\big|\ol{\eta}(x;k)\big| \leq \big|\varphi_{\rmB(0,k+q'), \rmB(x;k)}(x)\big| + C \sum_{\ell=k-q+1}^{k+q'}\big|\varphi_\ell(0) \big|\,,
	\end{equation}
	The result now follows from the uniform Gaussian tails of $\varphi_{\rmB(0,k+q'+1), \rmB(x;k)}(x)$ and $\varphi_\ell(0)$.
\end{proof}

We can now prove the required control on $K$:
\begin{lem}
	\label{lemma-4.2u}
	Let $\cU$ be a nice set. There is $c>0$ such that for all $n \geq 0$, $k \geq 1$,
	\begin{equation}
		\label{E:4.5}
		\bbP(K > k-1|S_{n'+1}=0)\le\texte^{-c(\log k)^2} \,.
	\end{equation}
	In particular, for each $\delta\in(0,1)$ and all~$n$ sufficiently large,
	\begin{equation}
		\label{E:4.6}
		\bbP(K>n^\delta|S_{n'+1}=0)\le n^{-2}.
	\end{equation}
\end{lem}

\begin{proof}
	The proof is the same as that of Lemma~4.2 in~\cite{BL3}, except that we need to add to the bound on the left hand side of~\eqref{E:4.5}, the probability that $(4)$ in the definition of $K$ fails for $k-1$ under the conditioning on $S_{n'+1} = 0$. To this end, we first use the independence of $(\chi_\ell)$ and $(S_\ell)$ to drop the conditioning. Then by the second part of Lemma~\ref{l:A.6} with $\ell$ and $y$ in place of $k$ and $x$ and with $q=q'=1$, noting that the number of $y$-s satisfying the condition in $(4)$ for a given $\ell$ is at most $4$, we upper bound the desired probability by $\sum_{\ell=1}^{n'} C \rme^{-c\log ((k-1) \vee (\ell \wedge (n'-\ell)))^2}$ which is at most $C' \rme^{-c'(\log k)^2}$, with all constant finite, positive and absolute. This shows the first statement in the lemma. The second statement follows by summation.
\end{proof}

For a nice set $\cU \subset \bbR^2$, we also define $\frg_{\cU_n} :\:\ol{\cU_n} \to \bbR^2$ as the unique function which takes the value $1$ on $\partial \cU_n$, vanishes at $0$ and is discrete harmonic in $\cU_n$. This function will be used to convert conditioning on $\{h_{\cU_n}(0) = u\}$ to conditioning on $\{h_{{\cU_n}}(0) = 0\}$ by adding $u \cdot \frg_{\cU_n}$ to the resulting field. The following lemma shows that $\frg_{\cU_n}$ interpolates linearly between $1$ at the origin and $0$ on the boundary, on an exponential scale.

\begin{lem}
	\label{l:A.7}
	Let $\cU$ be a nice set. There exists $C < \infty$ such that for all $n \geq 1$ and all $x \in \rmB(0; n')$, 
	\begin{equation}
		\bigg|\frg_{\cU_n}(x) - \frac{n' - \log |x|}{n'}\bigg| \leq \frac{C}{n'} \,.
	\end{equation}
	In particular, for all $k=1,\dots, n'-1$,
	\begin{equation}
		\label{e:A.33}
		\max_{x\in \Delta^k \smallsetminus \Delta^{k-1}}\Bigl| \,m_{n'}\bigl(1-\frg_{\cU_n}(x)\bigr)-m_{k}\Bigr|\le C \big(1+\log \bigl(1+k\wedge(n'-k)\bigr)\big) \,.
	\end{equation}
\end{lem}
\begin{proof}
	Since $\frg_{\cU_n}(x) := \rmG_{\cU_n}(x,0)/\rmG_{\cU_n}(0,0)$ the first bound follows from Lemma~\ref{l:103.2}. Using this inequality, we may upper bound the quantity inside the absolute value in~\eqref{e:A.33} by $C+\log k - (k/n') \log n'$ which is bounded by the right hand side of~\eqref{e:A.33}, as was shown, e.g., in~\cite{BL3}.
\end{proof}

Lastly, we shall need the following simple algebraic facts:
\begin{lem}
	\label{l:A.9}
	Let $a \in \bbR$, $f: \bbZ^2 \to \bbR$, $z \in \bbZ^2$ and $r \geq 1$. Then 
	\begin{equation}
		\label{e:A.35}
		\Big| \sqrt{\ol{(a+f)^2}(z;r)} -  |a| \Big| \leq \sqrt{\ol{f^2}(z;r)} \,,
	\end{equation}
	\begin{equation}
		\sqrt{\ol{(a+f)^2}(z;r)} \geq |a| - 4|\ol{f}(z;r)|\,,
	\end{equation}
	and
	\begin{equation}
		\sqrt{\ol{(a+f)^2}(z;r)} \leq |a| \vee r + |\ol{f}(z;r)| + \frac{\ol{f^2}(z;r)}{r} \,.
	\end{equation}
\end{lem}
\begin{proof}
	For the first statement, by Jensen's inequality,
	\begin{equation}
		\Big(|a|-\sqrt{\ol{f^2}(z;r)}\Big)^2 \leq \ol{(a+f)^2}(z;r) \leq \Big(|a|+\sqrt{\ol{f^2}(z;r)}\Big)^2 \,.
	\end{equation}
	For the second, by Taylor expansion,
	\begin{equation}
		\begin{split}	
			\sqrt{\ol{(a+f)^2}(z;r)} & \geq
			\bigg( \sqrt{a^2 - 2|a|\ol{f}(z;r)} \bigg) 1_{\{\ol{f}(z;r) \leq |a|/4\}} 
			\geq \bigg( |a| - \frac{2|a|\ol{f}(z;r)}{\sqrt{2}|a|}\bigg)1_{\{\ol{f}(z;r) \leq |a|/4\}}
			\\ 
			& \geq |a|-4|\ol{f}(z;r)| \,.
		\end{split}
	\end{equation}
	For the third by Taylor expansion again,
	\begin{equation}
		\sqrt{\ol{(a+f)^2}(z;r)} \leq
		\sqrt{\big(|a|\vee r\big)^2+2|a|\ol{f}(z;r) + \ol{f^2}(z;r)} \leq  
		|a|\vee r + |\ol{f}(z;r)| + \frac{\ol{f^2}(z;r)}{2r}.
	\end{equation}
\end{proof}
We can now already prove Lemma~\ref{l:3.6}:

\begin{proof}[Proof of Lemma~\ref{l:3.6}]
	By considering $\cU - x/n$ in place of $\cU$ (and shifting $\bbX_k$ accordingly) we may assume w.l.o.g. that $x=0$.
	Also, conditioning on $h_{\cU_n}(0) = v$ is the same as conditioning on $h_{\cU_n}(0) = 0$ and adding $v \frg_{\cU_n}$ to $h_{\cU_n}$. Employing the concentric decomposition, we have for $z \in \partial \rmB(y;k+1) \subset \Delta^{k+2}\setminus \Delta^{k}$,  
	\begin{equation}
		h_{\cU_n}(z) + v \frg_{\cU_n}(z) - \frac{n-k}{n} v = S_{n'+1}- S_{k+3}  + \eta(z) + \varepsilon(z) \,,
	\end{equation}
	where
	\begin{equation}
		\varepsilon(z) := v\Big (\frg_{\cU_n}(z) - \frac{n-k}{n}\Big) + 
		\sum_{\ell=k+1}^{n'}\frb_\ell(z)\varphi_\ell(0) + \sum_{\ell=k+4}^{n'}\chi_\ell(z) 
		, \quad 
		\eta(z) :=  \sum_{\ell=k+1}^{k+3} \big(\chi_\ell(z) + h'_\ell(z) \big) \,,
	\end{equation}
	
	By the first part of Lemma~\ref{l:A.9}, 
	\begin{equation}
		\sqrt{\ol{\Big(h_{\cU_n} + u \frg_{\cU_n} - \frac{n-k}{n} v\Big)^2}(y;k+1)} 
		\leq \big|S_{n'+1}- S_{k+3}\big| + \sqrt{\ol{\eta^2}(y;k+1)} +  \max_{z \in \rmB(y;k+1)} |\varepsilon(z)| \,.
	\end{equation}
	It then follows from Lemma~\ref{l:A.7}, the fact that $|n'-n| \leq C \log n$ by the choice of $x \in \cU_n^{R_n}$ and the range of $v$ that the first term in the definition of $\varepsilon(z)$ is bounded unifromly for all $z \in \partial \rmB(y;k+1)$. For the remaining terms, we can use the definition of $K$ and Proposition~\ref{p:A.2} to obtain $C \Theta_K(k)$ as an upper bound for all such $z$. The same definition also gives $\ol{\eta^2}(y;k+1) \leq C'' k \Theta_K(k)^2$.
	Altogether the probability in the statement of the lemma is upper bounded by
	\begin{multline}
		\bbP \bigg(\big|S_{n'+1}- S_{k+3}\big| + C \sqrt{k} \Theta_K(k) > t \,\bigg|\, S_{n'+1}=0 \bigg) \\
		\leq
		\bbP \Big(\big|S_{n'+1}- S_{k+3}\big| > t/2\,\Big|\, S_{n'+1}=0 \Big) + 
		\bbP \Big(\Theta_K(k) > t/(2C\sqrt{k}) \,\Big|\, S_{n'+1}=0 \Big) \,.
	\end{multline}
	Now, by Proposition~\ref{p:A.2}, conditional on $S_{n'+1} = 0$ the law of $S_{n'+1}-S_{k+3}$ is Gaussian with zero mean and variance at most $C(k \wedge (n'-k)) \leq Ck$, while by Lemma~\ref{lemma-4.2u}, $\Theta_K(k)$ has a Gaussian tail
	with a constant coefficient. The result follows by summation.
\end{proof}

Next we turn to estimates the growth of $f_n^2$ away from min-extreme vertices. To this end, we first recall the following lemma from~\cite{BL3}
\begin{lem}[Lemma~4.3 from~\cite{BL3}]
	\label{lemma-4.3}
	Let $\cU$ be a nice set. There is a constant~$C\in(0,\infty)$ such that for all $n \geq 1$ and $k=0,\dots,n'$,
	\begin{equation}
		\label{E:4.5ww}
		\Bigl|\,\max_{x\in \Delta^k\smallsetminus \Delta^{k-1}}\Big(h_{\cU_n}(x)-m_{n'} \big(1-\frg_{\cU_n}(x)\big)\Big)-(S_{n'+1}-S_k)\Bigr|
		\le C \Theta_{K}(k)^2 \,.
	\end{equation}
\end{lem}

We will also need barrier estimates for non-homogeneous random walks. For what follows let $Z_n := Z_0 + \sum_{k=1}^n \xi_k$ 
be a random walk with steps $\xi_k$ which are independent centered Gaussians with variance in $(\rho, 1/\rho)$ for some $\rho \in (0,1)$. The value of $Z_0$ will be specified formally by conditioning. Then the following are rather standard and can be found or easily adapted from, e.g.~\cite[Appendix C]{ballot} or~\cite[Supplement Material]{CHL17}.
\begin{lem}
	\label{l:A.11a}
	For all $\eta > 0$, there exists $C < \infty$ such that for all $n \geq 1$ and $a,b \in \bbR$,   
	\begin{equation}
		\bbP \big(\max_{k \in [0,n]} Z_k - \eta^{-1} \wedge_n(k)^{1/2-\eta} \leq 0,\, 
		\big| 
		Z_0 = a,\, Z_n = b \big) 
		\leq C \frac{(a^-+1)(b^-+1)}{n} \,.
	\end{equation}
\end{lem}
\begin{lem}
	\label{l:A.10a}
	For all $\eta \in (0,1/2)$, there exists $C < \infty$ such that for all $n \geq 1$, $0 \leq r_0 < r < n/3$ and  $a,b,d \in \bbR$,   
	\begin{multline}
		\bbP \big(\max_{k \in [r_0,n-r_0]} Z_k - \eta^{-1} \wedge_n(k)^{1/2-\eta} \leq 0,\, 
		\max_{k \in [r, n-r]} Z_k  + d + \eta^{-1} \wedge_n(k)^{1/2-\eta} > 0 \, \big| 
		Z_0 = a,\, Z_n = b \big) \\
		\leq C' \frac{(a^-+ 1)(b^- + 1)(d^-+1)}{n} r_0 (r-r_0)^{-\eta/2}\,. 
	\end{multline}
\end{lem}
\begin{lem}
	\label{l:A.12a}
	For all $\eta \in (0,1/2)$, there exists $C < \infty$ such that for all $n \geq 1$, $0 \leq r < n/3$ and  $a,b,d \in \bbR$,   
	\begin{multline}
		\label{e:A.45c}
		\bbP \big(\max_{k \in [0,n]} Z_k - \eta^{-1} \wedge_n(k)^{1/2-\eta} \leq 0,\, 
		\min_{k \in [r, n-r]} Z_k  + d + \eta \wedge_n(k)^{1/2+\eta} < 0 \, \big| 
		Z_0 = a,\, Z_n = b \big) \\
		\leq C \frac{(a^-+1)^2(b^-+1)^2(d^-+1)}{n} r^{-\eta/2} \,.
	\end{multline}
\end{lem}

The next lemma is key to the proof of Proposition~\ref{p:3.3}.
\begin{lem}
	\label{l:A.10b}
	Let $\cU$ be a nice set and $\eta \in (0,1/2)$,
	There exists $C \in(0,\infty)$ such that for all $n \geq 1$, $x \in \cU_n^{R_n}$,
	$v \in [0,n)$ and $u \in (-n, v]$,
	\begin{multline}
		\label{e:A.36}
		\bbP \Big(
		\exists k \in [r, r_n] ,\, y \in \bbX_k :\: x \in \rmB(y;k) \,,
		\sqrt{\ol{(h_{\cU_n}-m_n)^2}(y;k+1)} - \sqrt{2}k \notin \big(k^{1/2-\eta}, k^{1/2+\eta}\big),\, \\
		\max_{\cU_n} h_{\cU_n} \leq m_n + v
		\Big|\, h_{\cU_n}(x) = m_n + u \Big) 
		\leq C \frac{(v+u^-+\delta^++1)^6}{n} r^{-\eta/2} \,,
	\end{multline}
	where $\delta = n - \log \rmd(x, \cU_n^\rmc)$.
\end{lem}
\begin{proof}
	As in Lemma~\ref{l:3.6}, w.l.o.g we can take $x=0$. Then by adding $(m_n + u)\frg_{\cU_n}$ to $h_{\cU_n}$ we may 
	condition on $h_{\cU_n}(z) = 0$ instead of on $h_{\cU_n}(z) = m_n + u$. Now, the last condition in the event whose probability we are after can be recast as:
	\begin{equation}
		\Big\{
		\max_{k=1,\dots, n'} \max_{x \in \Delta^k \setminus \Delta^{k-1}} \Big(h_{\cU_n}(x)-m_{n'} \big(1-\frg_{\cU_n}(x)\big) + u \frg_{\cU_n}(x) - \delta(x)
		\Big) \leq v
		\Big\} \,,
	\end{equation} 
	where $\delta(x) := (m_{n}-m_{n'}) \big(1-\frg_{\cU_n}(x)\big)$.
	Using Lemma~\ref{lemma-4.3} and the fact the $|\frg_{\cU_n}|$ is bounded by $1$, the above event is included in 
	\begin{equation}
		\label{e:A.47}
		\Big\{ \max_{k \in [0,n']} \big( S_{n'+1} - S_k - C\Theta_K(k)^2 \big) \leq v+u^-+ \sqrt{2} \delta^+ \ \Big\} \,,
	\end{equation}
	for some $C < \infty$, where $\delta$ is as in the statement of the proposition.
	
	At the same time, for any $k < n'-2$ and $y \in \bbX_k$ with $0 \in \rmB(y;k)$ and any $x \in \partial \rmB(y;k+1) \subset \rmB(0, k+2) \setminus \rmB(0,k)$,
	\begin{equation}
		\label{e:A.44}
		h_{\cU_n}(x) + (m_n + u) \frg_{\cU_n}(x)- m_n  = S_{n'+1} - S_{k+1} - m_n \frac{k}{n'} + u \big(1-\tfrac{k}{n'}\big) + \eta(x) + \varepsilon(x) \,,
	\end{equation}
	where
	\begin{equation}
		\label{e:A.44b}
		\varepsilon(x) := \sum_{\ell=k+1}^{n'}\frb_\ell(x)\varphi_\ell(0) + \sum_{\ell=k+4}^{n'}\chi_\ell(x) - \big(m_n + u\big) \big(1 - \tfrac{k}{n'} - \frg_{\cU_n}(x)\big)\,,
	\end{equation}
	and
	\begin{equation}
		\label{e:A.44c}
		\eta(x) :=  \sum_{\ell=k+1}^{k+3} \big(\chi_\ell(x) + h'_\ell(x) \big) \,.
	\end{equation} 
	Also, by definition of $K$ and Lemma~\ref{l:A.7} we have $|\varepsilon(x)| \leq C\Theta_K(k)$ for all $x \in \partial \rmB(y;k+1)$ and $|\ol{\eta}(y;k+1)| \leq C\Theta_K(k)$,
	$|\ol{\eta^2}(y;k+1)/k| \leq C\Theta_K(k)^2$.
	
	Then if $\sqrt{\ol{\big(h_{\cU_n} - m_n \big( 1-\frg_{\cU_n}\big) + u\frg_{\cU_n}\big)^2}(y;k+1)} - \sqrt{2}k  > k^{1/2+\eta}$, by Lemma~\ref{l:A.9}
	\begin{equation}
		\begin{split}	
			\big|S_{n'+1} - S_{k+1} &\big| + \sqrt{2}k + |u| + \sqrt{2}|\delta| 
			\geq 
			\Big|m_n\tfrac{k}{n'} - u \big(1-\tfrac{k}{n'}\big) - (S_{n'+1} - S_{k+1})\Big| \\
			& \geq \bigg|\sqrt{\ol{\big(h_{\cU_n} - m_n \big( 1-\frg_{\cU_n}\big) + u\frg_{\cU_n}\big)^2}(y;k+1)} - C \sqrt{\ol{\eta^2 + \varepsilon^2}(y; k+1)} \\
			& \geq \sqrt{2} k + k^{1/2+\eta}  - C\sqrt{k} \Theta_K(k)
		\end{split}
	\end{equation}
	So that
	\begin{equation}
		\big|S_{n'+1} - S_{k+1}\big| \geq k^{1/2+\eta}  - |u| - \sqrt{2}|\delta| - C\sqrt{k} \Theta_K(k) \,.
	\end{equation}
	
	\smallskip
	On the other hand, if $\sqrt{\ol{\big(h_{\cU_n} - m_n \big( 1-\frg_{\cU_n}\big) + u\frg_{\cU_n}\big)^2}(y;k+1)} - \sqrt{2}k  < k^{1/2-\eta}$, by the second part of Lemma~\ref{l:A.9},
	\begin{equation}
		\begin{split}
			\sqrt{2}k - |u| & - (S_{n'+1} - S_{k+1})
			\leq 
			\Big|m_n\tfrac{k}{n'}  - u \big(1-\tfrac{k}{n'}\big) - (S_{n'+1} - S_{k+1})\big| + C \log k 
			\\
			& \leq \sqrt{\ol{\big(h_{\cU_n} - m_n \big( 1-\frg_{\cU_n}\big) + u\frg_{\cU_n}\big)^2}(y;k+1)} + 4 \big|\ol{\eta+\varepsilon}(y;k+1)\big| + C \log k\\
			& < \sqrt{2} k + k^{1/2-\eta} + C \Theta_K(k) \,,
		\end{split}
	\end{equation}
	so that 
	\begin{equation}
		S_{n'+1} - S_{k+1} > -k^{1/2-\eta} - |u| - C \Theta_K(k) \,.
	\end{equation}
	
	Abbreviating $\cA := |u|+|v|+\sqrt{2}|\delta|+1$, altogether the probability in~\eqref{e:A.36} is upper bounded by
	\begin{multline}
		\label{e:A.61}
		\bbP \Big( \max_{\ell \in [0,n']} \big( S_{n'+1} - S_\ell - C\Theta_K(\ell)^2 \big) \leq \cA \,,\,\, \exists \ell \in [r, r_n] :\: \\
		-\big(S_{n'+1} - S_{\ell}\big) \notin \Big(\ell^{1/2-\eta} + \cA + C \Theta_K(\ell)\,,\,\, \ell^{1/2+\eta} - \cA - C \sqrt{\ell} \Theta_K(\ell)\Big) \,\Big|\, S_{n'+1}=0 \Big) \,.
	\end{multline}
	
	To estimate the probability above, we first note that that $|S_{n'+1}-S_{n'-\ell+1}| \vee |S_\ell| \leq \ell \Theta_K(\ell)$ and that $\{K=k+1\} \subset A_k$, where $A_k$ is the event that one of conditions (1)-(4) in the definition of $K$ fails for $l=0, \dots, k$ or $l=n'-k, \dots, n'$. Then on $\{K=k\}$ for $k < r/2$, the probability in~\eqref{e:A.61} is at most
	\begin{equation}
		\label{e:A.53}
		\begin{split}
			\bbP \big(& A_{k-1} \,\big|  S_{n'+1} = 0\big) \times \\
			& \sup_{a,b}
			\bbP \Big( \max_{\ell \in [k,n'-k]} \big(-S_\ell - C \Theta_k(\ell)^2 \big)\leq \cA \,,\,\, \exists \ell \in [r, r_n] :\: \\
			& -S_{\ell} \notin \Big(-\ell^{1/2+\eta} + \cA + C \sqrt{\ell} \Theta_k(\ell)\,,\,\, -\ell^{1/2-\eta} - \cA - C \Theta_k(\ell) \Big)
			\, \Big|\, -S_k = a, -S_{n'-k+1} = b \Big) \,,
		\end{split}
	\end{equation}
	where the supremum is over all $|a| \vee |b| \leq k \Theta_k(k)$. Above, we have first replaced $K$ by $k$ in~\eqref{e:A.61}, then replaced $\{K=k\}$ by $A_{k-1}$ and finally conditioned on $A_{k-1}$ and on the values $S_{n'+1} - S_k$, $S_k$ and used the Markov property of the random walk $(S_k)$ and its independence on $(\chi_\ell)$ and $(h'_\ell)$.
	
	For the first probability in~\eqref{e:A.53} above we can use Lemma~\ref{lemma-4.2u}, since $A_{k-1} \subset \{K > k-1\}$. For the second, we observe that for $\ell=1, \dots, n'-2k+1$,
	\begin{equation}
		\Theta_k(k+\ell)^2 \leq k + \big(\wedge_{n'-2k+1}(\ell)\big)^{1/2-\eta} + C \,
	\end{equation}
	and for $\ell=1, \dots, r_n-k$,
	\begin{equation}
		(k+\ell)^{1/2-\eta} + C\Theta_k(k+\ell) \leq k+ \ell^{1/2-\eta} + C' \,,
	\end{equation}
	and also
	\begin{equation}
		(k+\ell)^{1/2+\eta} - C\sqrt{k+\ell}\Theta_k(k+\ell) \geq \tfrac12 \ell^{1/2+\eta} - C'' \,.
	\end{equation}
	Setting $S'_\ell := -S_{k+\ell}-k-\cA-C$, 
	$n'' := n' - 2k + 1$,
	$r' := r-k$,
	$a' := a - \cA - k - C$
	and $b' := b - \cA-k - C$,
	the second probability in~\eqref{e:A.53} is upper bounded by
	\begin{multline}
		\bbP \Big( \max_{\ell \in [0,n'']} \big(S'_\ell - \wedge_{n''}(\ell)^{1/2-\eta} \big) \leq 0  \,,\,\,\\ \exists \ell \in [r', r_n] :\: 
		S'_{\ell} \notin \Big(-\ell^{1/2+\eta/2}+C''  \,,\,\, -\ell^{1/2-\eta} - 2\cA -2k - C' \Big)
		\, \Big|\, S'_0 = a', S_{n''} = b' \Big) \,,
	\end{multline}
	Since $|a'|\wedge |b'| \leq \cA + k^2 + C$, applying Lemma~\ref{l:A.10a} and Lemma~\ref{l:A.12a} and using the union bound, the last probability is at most $C(\cA^6 k^{10})n^{-1}r^{-\eta/2}$ for all possible $a'$ and $b'$. 
	For $k \in [r/2, \sqrt{n}]$, we upper bound the probability in~\eqref{e:A.61} on $\{K=k\}$ by~\eqref{e:A.53} without the second event in the second probability and use Lemma~\ref{l:A.11a} to obtain the bound $C(\cA^6 k^{10})n^{-1}$. For $k > n^{1/2}$, we upper bound the probability in~\eqref{e:A.61} on $\{K=k\}$ by the first probability in~\eqref{e:A.53} only. 
	
	All together, we upper bound the probability in question by a constant timess
	\begin{equation}
		\sum_{k=0}^{r/2} \rme^{-c (\log k)^2} \frac{\cA^6 k^{10}}{n}r^{-\eta/2}
		+
		\sum_{k=r/2}^{\sqrt{n}} \rme^{-c (\log k)^2} \frac{\cA^6 k^{10}}{n}
		+
		\sum_{k=\sqrt{n}}^{\infty} \rme^{-c (\log k)^2} \,,
	\end{equation}
	which is smaller than the right hand side of~\eqref{e:A.36}.
\end{proof}

We are now ready for:
\begin{proof}[Proof of Proposition~\ref{p:3.3}]
	By the union bound, for any $v > \sqrt{u}$ and any $\eta, \eta' \in (0,1/2)$, the probability in~\eqref{e:3.9a} is at most
	\begin{equation}
		\label{e:A.55}
		\begin{split}
			\bbP & \big(\exists x \in \rmD_n :\: f_n(x) < -v\big) + \bbP \big(\exists x \in \rmD_n \setminus \rmD_n^{n-\eta'^{-1}} :\: f_n(x) < \sqrt{u}\big) \\
			& + \sum_{x}  
			\int_{-\sqrt{u}}^{\sqrt{u}}
			\bbP (f_n(x) \in \rmd u' ) 
			\,\, \bbP \Big(\max_{\rmD_n} h_{\rmD_n} \leq m_n + v,\,
			\exists k \in [r, r_n] ,\, y \in \bbX_k :\: x \in \rmB(y;k) \,, \\
			& \qquad \qquad \qquad \sqrt{\ol{(h_{\rmD_n}-m_n)^2}(y;k+1)} - \sqrt{2}k \notin  \big(k^{1/2-\eta'}, k^{1/2+\eta'}\big) 
			\Big|\, h_{\rmD_n}(x) = m_n + u' \Big)\,,
		\end{split}
	\end{equation}
	where the sum is over all $x \in \rmD_n^{n-\eta'^{-1}}$.
	Proposition~\ref{p:3.1} and Lemma~\ref{l:A.1} show that the first and second terms tend to zero when $n \to \infty$ followed by $v \to \infty$ and
	$\eta' \to 0$. For the remaining sum,
	by Lemma~\ref{l:103.2}, $f_n(x)$ is a Gaussian with mean $m_n$ and variance $n/2 + O(1)$. Using the formula for its density for the first term in the integral and Lemma~\ref{l:A.10b} for the second, the integral is at most
	\begin{equation}
		C \rme^{-2n} r^{-\eta/2} \int_{-\sqrt{u}}^{\sqrt{u}} (v+u'^++1)^6 \rme^{2\sqrt{2} u'} \rmd u' 
		\leq C (v+\sqrt{u} + 1)^6 \rme^{2\sqrt{2} \sqrt{u}} \rme^{-2n} r^{-\eta/2} \,,
	\end{equation}
	where $C$ may depend on $\eta'$. Since there are $O(\rme^{2n})$ terms in the sum~\eqref{e:A.55}, this shows that it goes to $0$ as $n \to \infty$ followed by $r \to \infty$ for any fixed $\eta'$ and any $u$ and $v$.  Together the last two assertions imply the desired result.
\end{proof}

In analog to Lemma~\ref{l:A.10b} we have:
\begin{lem}
	\label{l:A.12}
	Let $\cU$ be a nice set. For all $\eta > 0$ small enough, there exists $C \in(0,\infty)$ such that for all $n \geq 1$, $x \in \cU_n^{R_n}$ and
	$u \in (-n, -n)$,
	\begin{equation}
		\label{e:A.36b}
		\begin{split}
			\bbP \Big(&
			\forall k \in \big[n^\eta, n-n^\eta\big] \,,
			\sqrt{\ol{(h_{\cU_n}-m_n)^2}(x;k+1)} > \sqrt{2}k - n^\eta \big),\, \\
			& \exists k \in [r_n, n-r_n] \,,
			\sqrt{\ol{(h_{\cU_n}-m_n)^2}(x;k+1)} < \sqrt{2}k + n^\eta \big)
			\Big|\, h_{\cU_n}(x) = m_n + u \Big) \\
			& \qquad \qquad \leq C \frac{(u^-+\delta^++1)^4}{n^{1+\eta}}
		\end{split}
	\end{equation}
	where $\delta = n - \log \rmd(x, \cU_n^\rmc)$.
\end{lem}
\begin{proof}
	The proof is similar to that of Lemma~\ref{l:A.10b} and so we shall only highlight the differences.
	By re-centering we can assume again that $x=0$ and by adding $(m_n+u)\frg_{\cU_n}$ we can condition on $h_{\cU_n}(0) = 0$. 
	We also employ the concentric decomposition and use the definitions in~\eqref{e:A.44},~\eqref{e:A.44b} and~\eqref{e:A.44c}.
	If $\sqrt{\ol{\big(h_{\cU_n} - m_n \big( 1-\frg_{\cU_n}\big) + u\frg_{\cU_n}\big)^2}(0;k+1)} > \sqrt{2}k - n^\eta$, by the third part of Lemma~\ref{l:A.9} with $r=k$,
	\begin{equation}
		\begin{split}	
			\big|\sqrt{2}k - \big(S_{n'+1} - S_{k+1}\big) &\big| + |u| + \sqrt{2}|\delta| + C\log k
			\geq 
			\Big|m_n\tfrac{k}{n'} - u \big(1-\tfrac{k}{n'}\big) - (S_{n'+1} - S_{k+1})\Big| \vee k\\
			& \geq \sqrt{2} k - n^\eta - 
			C \Big( k^{-1} \ol{\eta^2}(y;k+1) - \ol{\eta+\varepsilon^2+1}(y; k+1)\Big) \\
			& \geq  \sqrt{2}k - n^\eta - C\Theta_K(k)^2 \,,
		\end{split}
	\end{equation}
	so that
	\begin{equation}
		S_{n'+1} - S_{k+1} \leq 2n^{\eta}  + |u| + \sqrt{2}|\delta| + C\Theta_K(k)^2 
		\quad \text{or} \quad 
		S_{n'+1} - S_{k+1} \geq \sqrt{2} k - |u| - \sqrt{2} \delta   \,.
	\end{equation}
	
	\smallskip
	On the other hand, if $\sqrt{\ol{\big(h_{\cU_n} - m_n \big( 1-\frg_{\cU_n}\big) + u\frg_{\cU_n}\big)^2}(0;k+1)} < \sqrt{2}k + n^{\eta}$, by the second part of Lemma~\ref{l:A.9},
	\begin{equation}
		\begin{split}
			\sqrt{2}k - |u| & - (S_{n'+1} - S_{k+1})
			\leq 
			\Big|m_n\tfrac{k}{n'}  - u \big(1-\tfrac{k}{n'}\big) - (S_{n'+1} - S_{k+1})\big| + C \log k 
			\\
			& \leq \sqrt{2} k + n^\eta + 4 \big|\ol{\eta+\varepsilon}(y;k+1)\big| + C \log k\\
			& < \sqrt{2} k + 2n^\eta + C \Theta_K(k) \,,
		\end{split}
	\end{equation}
	so that 
	\begin{equation}
		S_{n'+1} - S_{k+1} > -2n^{\eta} - |u| - C \Theta_K(k) \,.
	\end{equation}
	
	Altogether, with $\cA := |u|+\sqrt{2}|\delta|+1$, the probability in~\eqref{e:A.36b} is upper bounded by
	\begin{multline}
		\bbP \Big( \max_{\ell \in [n^\eta , n-n^\eta]} \big( S_{n'+1} - S_\ell - C\Theta_K(\ell)^2 \big) \leq \cA + 2n^\eta \,,\,\, \\\exists \ell \in [r_n, n-r_n] :\: 
		S_{n'+1} - S_{\ell} \geq -2n^\eta - \cA - C \Theta_K(\ell) \,\Big|\, S_{n'+1}=0 \Big) \,.
	\end{multline}
	Proceeding as in the proof  of Lemma~\ref{l:A.10b}, we use Lemma~\ref{l:A.10a} to upper bound the last probability by
	\begin{equation}
		C\frac{n^{2\eta}\cA^{4}}{n}r_n^{-\1/4} \,,
	\end{equation}
	It remains to choose $\eta$ small enough.
\end{proof}

We can now prove:
\begin{proof}[Proof of Proposition~\ref{p:3.17}]
	Setting $\delta(x) := n- \log \rmd(x, \rmD_n)$ for $x \in \rmD_n^\circ$, the density of $f_n(x)$ at $u'$ is at most
	\begin{multline}
		C n^{-1/2} \exp \bigg(-\frac{(m_n-u')^2}{n-\delta(x)+C} \bigg) \leq C n^{-1/2} \exp \bigg(-\frac{(m_n-u')^2}{n} \Big(1+\frac{\delta(x)-C}{n}\Big)\bigg) \\
		\leq C' n \rme^{-2n-\delta(x)} \rme^{2\sqrt{2}u'} \,.
	\end{multline}
	Multiplying by the upper bound in Lemma~\ref{l:A.12} with $u = u'$, gives a quantity which is upper bounded by $C' n^{-\eta} \rme^{-2n} |u'|^4 \rme^{2\sqrt{2}u'}$.
	Integrating the last expression from $u'=-\sqrt{u}$ to $u'=\sqrt{u}$, yields a quantity which tends to $0$ faster then $\rme^{-2n}$. But this quantity is an upper bound on the probability in the statement of the proposition.
\end{proof}

Lastly, for the proof of Proposition~\ref{p:3.4} we shall need:
\begin{lem}
	\label{l:A.13}
	Let $\cU$ be a nice set and $\eta \in (0,1/2)$,
	There exists $C \in(0,\infty)$ such that for all $n \geq 1$, $r < r_n$, $x \in \cU_n^{R_n}$ and $u \in \bbR$ 
	\begin{multline}
		\label{e:A.100}
		\bbP \biggl(\sqrt{\ol{(h_{\cU_n}-m_n)^2}(x;r+1)} \leq \sqrt{2}r - r^{1/2-\eta}\,,\,\,\\ \max_{\cU_n \setminus \rmB(x;r)} h_{\cU_n} \leq m_n + r^{1/2-\eta} 
		\Big|\, \ol{h_{\cU_n}}(x;r) = m_{n-r} + u \bigg) 
		\leq 
		C\frac{(|u|+|\delta(x)|+1)^6}{n} \rme^{-c' \big((r^{1/2} - u)^+\big)^2} ,
	\end{multline}
	where $\delta(x) = n - \log \rmd(x, \cU_n^\rmc)$.
\end{lem}

\begin{proof}
	The proof is again similar to that of Lemma~\ref{l:A.10b} and makes use of the concentric decomposition above. 
	We begin by assuming w.l.o.g. that $x=0$. Letting
	\begin{equation}
		\frg_{\cU_n,r}(x) := \frac{\bbE h_{\cU_n}(x)h_{\cU_n}(y)}{\bbE h_{\cU_n}(0)^2 - \bbE h_{\rmB(0;r)}(0)^2} = \frac{\rmG_{\cU_n}(x,0)}{\rmG_{\cU_n}(0,0) - \rmG_{\rmB(0;r)}(0,0)} \,,
	\end{equation}
	we may condition on $\ol{h_{\cU_n}}(0;r) = S_{n'+1}-S_{r+1} = 0$ instead of on $\ol{h_{\cU_n}}(0;r) = m_{n-r} + u$ and instead add $(m_{n-r} + u)\frg_{{\cU_n,r}}(x)$ to $h_{\cU_n}(x)$. Similar considerations as in Lemma~\ref{l:A.7} show that
	\begin{equation}
		\bigg|\frg_{\cU_n,r}(x) - \frac{n' - \log |x|}{n''}\bigg| \leq \frac{C}{n'} \,,
	\end{equation}
	where we recall that $n' = n -\delta$, and also for $k > r$,
	\begin{equation}
		\max_{x\in \Delta^k \smallsetminus \Delta^{k-1}}\Bigl| \,m_n' - m_{n'-r}\frg_{\cU_n,r}(x)-m_{k}\Bigr|\le C \big(1+\log \bigl(1+k\wedge(n'-k)\bigr)\big) \,.
	\end{equation}
	
	In analog to $\Theta_k(\ell)$ from~\eqref{E:4.1} we set
	\begin{equation}
		\Theta_{k,r}(\ell):=
		\left \{ \begin{array}{ll}
			\Big(1+\log\Big(1+[k \vee \big((\ell - r) \wedge                                                                                                                                                                                                                                                                                                                                                                                                                                                                                                                                                                                                                                                                                                                                              (n'-\ell)\bigr)\Big)
			& \ell = r, \dots, n' \\
			\infty  & \ell=0,\dots, r-1 \,.
		\end{array}
		\right.
	\end{equation}
	Notice that the control variable $K$, as defined in Definition~\ref{d:A.4}, now only ``controls'' (i.e. measurable w.r.t. the sigma-algebra generated by) $(\varphi_{\ell}$, $(\chi_{\ell}$ and $h'_{\ell}$ for $\ell \in [r, n']$. In place of Lemma~\ref{lemma-4.3} and Lemma~\ref{lemma-4.2u} we shall now need
	\begin{equation}
		\Bigl|\,\max_{x\in \Delta^k\smallsetminus \Delta^{k-1}}\Big(h_{\cU_n}(x)-\big(m_{n'}-m_{n'-r}\frg_{\cU_n,r}(x)\big)\Big)-(S_{n'+1}-S_k)\Bigr|
		\le C \Theta_{K,r}(k)^2 \,,
	\end{equation}
	and
	\begin{equation}
		\label{E:4.5.1}
		\bbP(K > k-1|S_{n'+1}-S_{r+1}=0)\le\texte^{-c(\log k)^2} \,.
	\end{equation}
	The proofs carry over with straightforward modifications.
	
	Proceeding as in the proof of Lemma~\ref{l:A.10b}, using the concentric decomposition and the control variable, the first condition in the event whose probability we are after in~\eqref{e:A.100} is included in
	\begin{equation}
		\Big \{ S_{n'+1} - S_{r+1} < r^{1/2-\eta} - u - C \Theta_{K,r}(r) \Big\} \,,
	\end{equation}
	while the second is included in 
	\begin{equation}
		\Big \{ \max_{k \in [r+1,n']} \big( S_{n'+1} - S_k - C\Theta_{K,r}(k)^2 \big) \leq r^{1/2-\eta} + u^- + \sqrt{2} \delta(0)^+ \Big\} \,.
	\end{equation}
	Altogether the conditional probability in~\eqref{e:A.100} is upper bounded by
	\begin{multline}
		\bbP \Big(	C \Theta_{K,r}(r) > r^{1/2-\eta} - u \,,\\
		\max_{\ell \in [r+1,n']} \big( S_{n'+1} - S_\ell - C\Theta_{K,r}(\ell)^2 \big) \leq r^{1/2-\eta} + u^- + \sqrt{2} \delta(0)^+
		\, \Big|\, S_{n'+1} - S_{r+1} = 0 \Big) \,.
	\end{multline}
	
	Intersecting as before with $\{K=k\}$ allows us to substitute $K$ by $k$. We then replace the event $\{K=k\}$ by the event $A_{k,r}$ where one of the conditions in Definition~\ref{d:A.4} fails for $l=r, \dots, r+k$ or $l=n'-k+1, \dots, n'$. Conditioning on $A_{k,r}$ and on $S_{n'-k}$ and $S_{r+k}$ and using~\eqref{E:4.5.1} and Lemma~\ref{l:A.11a}, we then get with $\cA := |u|+r+|\delta(0)|+1$,
	\begin{equation}
		\sum_{k=0}^{n^{1/2}} \rme^{-c (\log k)^2}  1_{\{C\Theta_{k,r}(r) > r^{1/2}-u\}} \frac{\cA^6 k^{10}}{n'-r} 
		\leq \frac{(|u|+|\delta(0)|+1)^6}{n} \rme^{-c' \big((r^{1/2} - u)^+\big)^2} \,. 
	\end{equation}
\end{proof}

We can now provide:
\begin{proof}[Proof of Proposition~\ref{p:3.4}]
	By Lemma~\ref{l:A.13} and the union bound, the probability in~\eqref{e:3.9} is at most
	\begin{equation}
		\label{e:A.55.1}
		\begin{split}
			\bbP & \big(\exists x \in \rmD_n :\: f_n(x) < -k^{1/2-\eta}\big) \\
			& + \sum_{x \in \bbX_k \cap \rmD_n^\circ}  
			C \int
			\bbP \big(\ol{f_n}(x;r) - m_n + m_{n-k} \in \rmd u\big)
			\frac{(|u|+|\delta(x)|+1)^6}{n} \rme^{-c' \big((k^{1/2} + u)^+\big)^2} \,. 
		\end{split}
	\end{equation}
	The first probability in the integrand is at most 
	\begin{equation}
		C (n-k)^{-1/2} \exp \bigg(-\frac{(m_{n-k}-u)^2}{n-r-\delta(x)+C} \bigg) 
		\leq C' n \rme^{-2(n-k)-\delta(x)} \rme^{2\sqrt{2}u} \,.
	\end{equation}
	Plugging this in the integrand, the terms in the sum above are upper bounded by a constant multiple of $\rme^{-2(n-k)}\rme^{-cr^{1/2-\eta}}$. Since it contains at most $\rme^{2(n-k)}$ terms, the sum is therefore bounded by $\rme^{-c'k^{1/2-\eta}}$. But then thanks to Proposition~\ref{p:3.1} the first term in~\eqref{e:A.55.1} is bounded by a quantity of the same form.
\end{proof}

\section*{Acknowledgments}
Both authors would like to thank Marek Biskup for many stimulating discussions and the argument for Theorem~\ref{t:2b}. The work of O.L. was supported by the Israeli Science Foundation grant no. 1382/17 and 2870/21, and by the United States-Israel Binational Science Foundation (BSF) grant no. 2018330.
The work of S.S. was supported in part at the Technion by a fellowship from the Lady Davis Foundation, the Israeli Science Foundation grants no. 1723/14 and 765/18 and also by Fondecyt grants no. 11200690 and 1240848.
\bibliographystyle{abbrv}
\bibliography{BoxCoverTime}

\end{document}